\documentclass[a4paper,11pt]{article}

\usepackage{anyfontsize}

\usepackage{booktabs}
\usepackage{colortbl}
\usepackage[pdftex,dvipsnames,table]{xcolor}
\usepackage{xfrac}
\usepackage{xargs}
\usepackage[utf8]{inputenc}
\usepackage[mathscr]{euscript}
\usepackage{xspace}
\usepackage{graphicx}
\usepackage{color}
\usepackage{enumerate}
\usepackage[english]{babel}
\usepackage{verbatim}
\usepackage{float}
\usepackage{soul}
\usepackage{rotating}
\usepackage{caption}
\usepackage[draft]{fixme}
\usepackage{longtable}
\usepackage{mathtools}
\usepackage[shortlabels]{enumitem}
\usepackage{tabularx}
\usepackage{bbm}
\usepackage[inkscapelatex=false]{svg}
\usepackage{subfig}
\usepackage{makecell}
\usepackage{refcount}

\definecolor{lightgray}{gray}{0.9}
\restylefloat{table}

\makeatletter
\newcommand{\multiline}[1]{%
  \begin{tabularx}{\dimexpr\linewidth-\ALG@thistlm}[t]{@{}X@{}}
    #1
  \end{tabularx}
}
\makeatother

\usepackage{algorithm}%
\usepackage[noend]{algpseudocode}%

\usepackage{natbib}
 \bibpunct[, ]{(}{)}{,}{a}{}{,}%

\usepackage{hyperref}
\usepackage{cleveref}

\makeatletter
\newcounter{HALG@line}
\renewcommand{\theHALG@line}{\thealgorithm.\arabic{ALG@line}}
\makeatother

\hypersetup{
    colorlinks,
    linkcolor={red!50!black},
    citecolor={blue!50!black},
    urlcolor={blue!80!black}
}

\newenvironment{tightquote}
  {\begingroup\leftskip=2em \rightskip=2em \parindent=0pt \parskip=0pt}
  {\par\endgroup}
  
\newcommand{\R}{\mathbb{R}}
\newcommand{\Z}{\mathbb{Z}}

\newcommand{\Q}{\mathbb{Q}}

\newcommand{\mb}[1]{\mathbb{#1}}

\newcommand{\tsc}[1]{\textsc{#1}}
\newcommand{\oa}[1]{\vec{#1}}
\newcommand{\mc}[1]{\mathcal{#1}}

\makeatletter
\DeclareRobustCommand{\cev}[1]{%
  {\mathpalette\do@cev{#1}}%
}
\newcommand{\do@cev}[2]{%
  \vbox{\offinterlineskip
    \sbox\z@{$\m@th#1 x$}%
    \ialign{##\cr
      \hidewidth\reflectbox{$\m@th#1\vec{}\mkern4mu$}\hidewidth\cr
      \noalign{\kern-\ht\z@}
      $\m@th#1#2$\cr
    }%
  }%
}
\makeatother

\DeclareMathOperator{\proj}{\tsc{proj}}

\DeclarePairedDelimiter{\ceil}{\lceil}{\rceil}
\DeclareMathOperator{\epi}{\tsc{epi}}
\newcommand{\allones}{\mathbf{1}}
\newcommand{\T}{\mathsf{\scriptscriptstyle T}} %

\DeclareMathOperator{\vrpr}{\hyperref[problem:vrpr]{\textsc{vrpr}}}
\newcommand{\epiQ}[1]{\hyperref[problem:vrpr]{\epi(#1)}}
\newcommand{\X}{\hyperref[assumption:formulation]{\mc{X}}}
\newcommand{\Xsub}{\hyperref[set:subtour]{\mc{X}_{\tsc{sub}}}}
\newcommand{\Xcvrp}{\hyperref[set:cvrp]{\mc{X}_{\tsc{cvrp}}}}
\newcommand{\bark}[1]{\hyperref[set:cvrp]{\bar{k}(#1)}}
\newcommand{\Qc}{\hyperref[eq:formula_dror]{\mc{Q}_C}}
\newcommand{\Qcscen}{\hyperref[eq:formula_dror]{\mc{Q}^\xi_C}}
\newcommand{\routes}{\hyperref[def:realizable_routes]{\Pi}}

\newcommand{\Om}{\hyperref[def:recourse_disaggregation]{\Omega}}
\newcommand{\F}[1]{\hyperref[set:F]{\mathcal{F}(#1)}}
\newcommand{\projrho}{\proj_{(x, \rho)}}
\newcommand{\W}{\hyperref[def:activation_function]{W}}
\renewcommand{\L}{\hyperref[def:recourse_lower_bound]{\mc{L}}}

\newcommand{\Wg}{\hyperref[ineq:gendreau_cut]{W^k_G}}

\renewcommand{\H}{\hyperref[def:partial_route]{H}}
\newcommand{\Xh}[1]{\hyperref[set:x_h]{\mc{X}_{=}(#1)}}
\newcommand{\supsetXh}[1]{\hyperref[set:path_cut]{\mc{X}_{\supseteq}(#1)}}
\newcommand{\Whs}{\hyperref[thm:partial_route_activation_exact_adheres]{W_{HS}}}
\newcommand{\Wdl}[1]{\hyperref[set:path_cut]{W_{DL}(#1)}}
\newcommand{\setWdl}[1]{\hyperref[eq:activation_set]{W_{DL}(#1)}}
\newcommand{\Xset}[1]{\hyperref[set:x_set]{\mc{X}(#1)}}

\newcommand{\supsetWof}[2]{\hyperref[proposition:activation_function_adheres]{W_{OF}(#1 ; \mc{X}_{\supseteq}(#2))}}

\newcommand{\failfunction}[2]{\hyperref[def:fail_function]{\tsc{fail}\left(#1, #2\right)}}
\newcommand{\vrpsdLB}[3]{\hyperref[def:vrpsd_lower_bound]{\mc{L}^{#1}_\xi\left(#2, #3\right)}}
\newcommand{\Lc}[1]{\hyperref[prop:partial_route_lower_bound]{\mc{L}_C(#1)}}
\newcommand{\setLc}[1]{\hyperref[prop:set_recourse_lower_bound]{\mc{L}_C(#1)}}

\newcommand{\algBest}{$\tsc{pr}$+$D1$}
\newcommand{\algBestD}{$\tsc{pr}$+$D2$}
\newcommand{\algBestDset}{$\tsc{pr}$+$D2$+\tsc{set}}

\usepackage[margin=1.2in]{geometry}
\usepackage{amsmath,amsthm, amssymb, amsfonts, amsbsy}
\usepackage{thmtools} %
\usepackage{thm-restate} %
\usepackage{authblk}
\providecommand{\keywords}[1]{\textit{Keywords:} #1}
\usepackage{natbib}
\setcitestyle{authoryear,round}

\DeclareMathOperator*{\argmax}{arg\,max}

\makeatletter
\newtheoremstyle{mystyle}%
{\topsep}{\topsep}
{\itshape}{}
{\bfseries}{}
{0.5em}
{\thmname{\@ifempty{#3}{#1}\@ifnotempty{#3}{#3}}}
\makeatother

\theoremstyle{mystyle}

\newcommand{\Halmos}{\qed}

\theoremstyle{plain}
\newtheorem{theorem}{Theorem}
\newtheorem{proposition}{Proposition}
\newtheorem{corollary}{Corollary}
\newtheorem{lemma}{Lemma}
\newtheorem{claim}{Claim}
\newtheorem{fact}{Fact}

\theoremstyle{definition}

\newtheorem{definition}{Definition}
\newtheorem{example}{Example}
\newtheorem{remark}{Remark}
\newtheorem{assumption}{Assumption}
\newtheorem{contribution}{Contribution}

\newtheoremstyle{named}{}{}{}{}{\bfseries}{.}{.5em}{#1 #3}
\theoremstyle{named}

\newenvironment{APPENDICES}{\appendix}{}

\usepackage{silence}
\begin{document}

\title{On vehicle routing problems with stochastic demands --- Generic disaggregated integer L-shaped formulations}

\author[1]{Matheus Jun Ota\thanks{(mjota@uwaterloo.ca)}}
\author[1]{Ricardo Fukasawa\thanks{(rfukasawa@uwaterloo.ca)}}
\affil[1]{University of Waterloo, Waterloo, Ontario, Canada}

\maketitle

\begin{abstract}
We study the vehicle routing problem with stochastic demands (VRPSD), an important variant of the classical capacitated vehicle routing problem in which customer demands are modeled as random variables. 
We develop the first algorithm for the VRPSD in the case where the demands are given by an empirical probability distribution of scenarios --- a data-driven variant that tackles a significant challenge identified in the literature: dealing with correlations. Indeed, most previous exact algorithms for this problem relied on independence of the random variables.
To address the VRPSD with scenarios, we introduce a unifying framework that generalizes existing integer L-shaped (ILS) formulations developed for other variants of the problem. 
This framework and subsequent analysis allow us to generalize previous ILS cuts and pinpoint which assumptions are needed to apply those generalizations. In particular, our results enable, for the first time, the combination of two previous types of inequalities: \emph{partial route} and \emph{set} cuts, which leads to significant computational improvements.

\end{abstract}


\section{Introduction}
\label{section:intro}

The \textit{Capacitated Vehicle Routing Problem} (CVRP) is a fundamental combinatorial optimization problem in which one seeks minimum-cost routes to serve all customer demands while respecting vehicle capacity constraints. As a cornerstone problem in Operations Research (OR), the CVRP has driven numerous theoretical and practical advances in combinatorial optimization and mathematical programming~\citep{TothV02}. In this paper, we study the \textit{Two-Stage Vehicle Routing Problem with Stochastic Demands} (VRPSD), a variant of the CVRP where a fixed number of routes are decided \emph{a priori}, customer demands are random variables revealed upon vehicle arrival, and a \emph{recourse cost} is incurred whenever a planned route exceeds vehicle capacity. This problem has been investigated for over five decades~\citep{gendreau201650th, tillman1969multiple}, with growing interest in recent years~\citep{louveaux2018exact, salavati2019trsc,florio2022recent,hoogendoorn2023improved,ota2024hardness,parada2024disaggregated,legault2025superadditivity}.

Despite this growing interest, most VRPSD studies have strong assumptions on the random variables, like independent probability distributions with a convolution property (see  Section 3.1 of~\cite{gendreau201650th}). The assumption of independence simplifies model tractability, but it is often unrealistic, as customer demands are frequently correlated in practice. Indeed, handling correlated demands has been explicitly identified as a challenge in the field~\citep{gendreau201650th}. However, to our knowledge, the only progress made in this direction consists of the work of~\cite{florio2022recent}, where the authors consider correlations that can be modeled with a single external factor, and Chapter~5 of Hoogendoorn's thesis~\citep{hoogendoorn2024vehicle}, which addresses correlations of multivariate normal distributions.

In contrast to this landscape, scenario-based approaches are standard and have been extensively studied in the general stochastic optimization literature. They approximate a wide range of probability distributions using empirical samples~\citep{chen2022sample, bertsimas2018robust, narum2024problem}, allowing for correlations while sometimes offering theoretical guarantees via sample average approximation~\citep{kleywegt2002sample, birge2011introduction, luedtke2008sample, swamy2012sampling}.

To bridge this gap, this paper proposes solution methodologies for the VRPSD when uncertainty is modeled using \emph{demand scenarios}. Our computational study considers the \emph{classical recourse policy}, one of the most investigated recourse policies in the literature~\citep{dror89, gendreau95, laporte2002, jabali2014, parada2024disaggregated}, in which the vehicles follow the routes determined in the first stage and perform replenishment trips in the second stage whenever the capacity gets exceeded. Although we focus on the classical recourse policy, our theoretical results also allow us to determine precisely under which conditions our proposed methodologies could be applied to other recourse policies for the VRPSD with scenarios.

We concentrate on \emph{integer L-shaped} (ILS) formulations because they have an extensive literature in the field and are the basis of the most successful branch-and-cut algorithms for the VRPSD to date~\citep{hoogendoorn2023improved,
parada2024disaggregated, legault2025superadditivity}. Additionally, for the VRPSD with scenarios, branch-and-price algorithms face intrinsic hardness issues~\citep{ota2024hardness}, further motivating our focus on branch-and-cut approaches. 

To better position our contributions, we provide a brief overview of the ideas behind state-of-the-art ILS algorithms for the VRPSD (a more detailed and technical description will be given in Section~\ref{section:review}).

\subsection{ILS approaches and our contributions}

Let~$G = (V, E)$ be a complete undirected graph. The basic idea in ILS approaches for the VRPSD is to formulate the problem as
\begin{equation}
    \label{problem:vrpr1}
    \min \{c^\T x + \rho : \rho \geq \mc{Q}(x),~ x \in \mc{X} \cap \Z^E \}.
\end{equation}
where $\mc{X}\cap \Z^E$ is the set of feasible first-stage decisions (routes that are decided a priori), and $\mc{Q}(x)$ is the \emph{recourse function} representing the expected recourse cost incurred by taking those decisions. 

Traditionally, ILS approaches~\citep{gendreau95, laporte2002, jabali2014, Salavati2019175, Salavati2019, salavati2019trsc, louveaux2018exact} replace the constraints on the recourse cost variable $\rho$ in Problem~\eqref{problem:vrpr1} with optimality cuts (or \emph{lower bounding functionals}) defined by an \emph{activation function}, which determines when the ILS cut is ``active'', and a corresponding \emph{recourse lower bound} that applies whenever the cut is ``active''. These inequalities are then used to lower bound the recourse function~$\mc{Q}(x)$ with linear expressions on~$x$.

To improve the approximation of the recourse function one may \emph{disaggregate} the recourse cost into smaller components. For example, instead of using a single recourse cost variable~$\rho$, one may decompose it across the customers as~$\rho = \sum_v \theta_v$, and then generate ILS cuts based on only a subset of the~$\theta_v$-variables. This concept was first explored in Chapter~6 of Séguin's thesis~\citep{seguin}, where the author proposes a \emph{route-based recourse decomposition} that assigns the recourse cost of a whole route to a single representative customer in that route. A similar route-based recourse decomposition was adopted by~\cite{cote2020vehicle} in the context of a vehicle routing problem with stochastic two-dimensional items. For the VRPSD with uncorrelated demands,~\cite{hoogendoorn2023improved} recently combined this idea with the concept of \emph{partial routes}~\citep{hjorring1999new} to derive the so-called \emph{partial route-split inequalities}. 

On the other hand, \emph{customer-based recourse decompositions} underlie the recently proposed \emph{disaggregated integer L-shaped} (DL-shaped) method~\citep{parada2024disaggregated}, where the recourse cost can be distributed more broadly across multiple customers in a route, allowing several customers in the route to attain positive~$\theta_v$-values. This alternative disaggregation enabled the authors to combine two different classes of ILS cuts, namely the \emph{path cuts} and \emph{set cuts}. Under the assumption that customer demands are independent (and some additional assumptions, such as assuming that probability distributions have a certain convolution property, see~\cite{gendreau201650th, christiansen2007}), this approach achieved state-of-the-art results for the VRPSD under both the \emph{classical recourse policy}~\citep{parada2024disaggregated} and the \emph{optimal recourse policy}~\citep{legault2025superadditivity}. The method was also extended by~\cite{PARADA2025} to a stochastic bike-sharing rebalancing problem.

Originally, the DL-shaped method was shown to be valid for \emph{monotone} recourse functions. However,~\cite{legault2025superadditivity} showed that this claim is incorrect, and in fact, when the number of available vehicles is unlimited, they demonstrated that the DL-shaped method is valid if and only if the recourse function~$\mc{Q}$ is \emph{superadditive}. (A more detailed technical discussion on this issue, including a formal definition of such terms, is left to Section~\ref{section:review}.)

For the VRPSD with scenarios, it is not hard to see (and we indeed show this in Example~\ref{example:qc_not_superadditive}) that 
the  recourse function associated with the classical recourse policy is not superadditive, which may lead one to believe that path and set cuts cannot be used. Yet, the result of~\cite{legault2025superadditivity} crucially relies on the assumption of an unlimited number of vehicles to prove the necessity of superadditivity. In contrast, in this paper (and in most other works in the literature, see Table~1 of~\cite{hoogendoorn2025evaluation}), the number of available vehicles is fixed and limited. Therefore, our first (minor) contribution is as follows.

\begin{contribution}
\label{cont:rsuper}
We identify a weaker notion of superadditivity, called, \emph{restricted superadditivity} (Definition~\ref{def:weak_superadditivity}), and show that path cuts are valid if and only if restricted superadditivity holds (Theorem~\ref{thm:supperadditive}). In particular, this result holds both for unlimited or fixed number of vehicles.
\end{contribution}

Still, the classical recourse function is not restrictively superadditive (Example~\ref{example:qc_not_superadditive}), so path cuts cannot be used in our context. Given the results of~\cite{legault2025superadditivity} on the DL-shaped method, one might then conclude that the same also holds true for set cuts. This leads to our second contribution.

\begin{contribution}
    \label{cont:set}
    We derive set cuts that are valid for the VRPSD with scenarios under the classical recourse policy (Proposition~\ref{prop:set_recourse_lower_bound}).
\end{contribution}

\noindent
We also derive partial route cuts for our problem, although we view this as a minor contribution:

\begin{contribution}
    \label{cont:partial}
    We derive partial route inequalities that are valid for the VRPSD with scenarios under the classical recourse policy (Proposition~\ref{prop:partial_route_lower_bound}).
\end{contribution}

Given Contributions~\ref{cont:set} and \ref{cont:partial}, one may be led to believe that an approach to solve the problem is to use both of these cuts at the same time. This is, however, more challenging than it seems at first, as it depends on which type of recourse decomposition that is being used. This leads us to our fourth (and perhaps most important) contribution.

\begin{contribution}
    \label{cont:recoursedisag}
    We formally tie the validity of each class of cuts to a particular way of decomposing the recourse, termed \emph{recourse disaggregation} (Definition~\ref{def:recourse_disaggregation}). In addition, we argue that combining cuts that use different recourse disaggregations may lead to invalid formulations (Section~\ref{subsection:limitations}). 
    This explicit link between validity and recourse disaggregation is developed in the framework proposed in Section~\ref{section:framework}, which allows us to safely combine set cuts with partial route cuts for our problem.
\end{contribution}

We highlight that combining set cuts with partial route cuts has not been done before in the literature, for any variant of the VRPSD. While it is hard to speculate the reason for this, since it has not been explicitly justified in previous works, we believe that this contribution allows us to shed light on an explicit reason why this may or may not be possible.
Moreover, we note that Contribution~\ref{cont:recoursedisag} can be extended to other variants of the VRPSD, as the framework setup in Section~\ref{section:framework} allows us to rigorously determine the technical conditions under which our results hold.

We end this section by outlining the structure of the paper. It is important to note that, while it was convenient to describe the contributions in this particular order to provide context for the reader, the actual technical content does not follow that order. 

Section~\ref{section:setup} introduces the general class of problems our framework addresses and presents the VRPSD with scenarios under the classical recourse policy as the main motivating application. Section~\ref{section:review} reviews the disaggregated ILS formulations of~\cite{hoogendoorn2023improved} and~\cite{parada2024disaggregated}, explaining in detail the limitations of previous approaches and justifying the development of a unifying framework. Section~\ref{section:framework} presents the framework
(Contribution~\ref{cont:recoursedisag}), 
while Section~\ref{section:generalizing} applies it to generalize existing ILS cuts and to characterize when certain ILS cuts yield valid reformulations (Contributions~\ref{cont:rsuper} and \ref{cont:recoursedisag}). Section~\ref{section:application_vrpsd} derives recourse lower bounds for the VRPSD with scenarios under the classical recourse policy (Contributions~\ref{cont:set} and \ref{cont:partial}), and Section~\ref{section:experiments} reports the computational results. Finally, Section~\ref{section:conclusion} presents the concluding remarks. To facilitate reading, Appendix~\ref{appendix:notation} lists the main notation used throughout the paper.

\subsection{Notation}
We use~$\mb{R}_+$ and~$\mb{R}_{++}$ to denote the sets of nonnegative and positive real numbers, respectively. Similar notation applies to~$\Q$ and~$\Z$. For any real number~$a$, we use~$(a)^+ \coloneqq \max\{0, a\}$. Moreover, whenever~$a$ is an integer, we define~$[a] \coloneqq \{1, \ldots, a\}$ if~$a$ is positive, and~$[a] \coloneqq \emptyset$ otherwise. The notation~$\mb{I}( \cdot )$ denotes the indicator function. If~$f$ is a vector and~$i$ is one of its coordinates, we write~$f_i$ and~$f(i)$ interchangeably. For any function (respectively, vector)~$f$ and a subset~$H$ of its domain (respectively, coordinates), we use~$f(H)$ as shorthand for~$\sum_{i \in H} f(i)$. The notations~$\mathbf{1}$ and~$\mathbf{0}$ refer to the all-ones and all-zeroes vectors, respectively.

For any undirected graph~$G$, we use~$V(G)$ and~$E(G)$ to refer to the set of vertices and edges of~$G$, respectively. If~$G$ is a directed graph (digraph), we use~$A(G)$ to refer to the set of arcs of~$G$. To ease the presentation, we sometimes abbreviate an edge~$\{u, v\}$ or an arc~$(u, v)$ simply to~$uv$. Given an undirected graph~$G$ and a set~$S \subseteq V(G)$, the notation~$\delta_G(S)$ (respectively,~$E_G(S)$) denotes the set of edges in~$G$ with exactly one endpoint (respectively, two endpoints) in~$S$, and we may omit the subscript when the graph~$G$ is clear from the context. If~$S \subseteq V(G)$ is a singleton~$\{v\}$, we often refer to~$S$ as simply~$v$, meaning that~$\delta_{G}(v)$ is well-defined. For every pair of disjoint subsets~$S, T \subseteq V(G)$, we define~$E_G(S, T) \coloneqq \{e \in E : |e \cap S| = |e \cap T| = 1\}$. We use~$G' \subseteq G$ to indicate that~$G'$ is a subgraph of~$G$.

\section{The setup}
\label{section:setup}

In this section, we describe the class of problems our framework addresses, the assumptions we make, and we present the VRPSD with scenarios under the
classical recourse policy as the main motivating application.

\subsection{Generic problem definition}
\label{subsection:problem}

From now on, we fix~$G$ to be a complete undirected graph with vertex set~$V \coloneqq V(G) = \{0\} \dot \cup V_+$, edge set~$E \coloneqq E(G)$ and edge weights~$c \in \Q^E_+$. The vertex~$0$ represents the \emph{depot} and~$V_+$ \label{sym:V+} denotes the set of \emph{customers}. We also fix~$\mc{I}$ to be a tuple representing the input of the generic problem that we consider, and we assume that~$\mc{I}$ contains the graph~$G$ and its edge weights. For example, an instance of the VRPSD with demands following independent normal distributions can be represented as~$\mc{I} = (G, c, k, C, \bar{d}, \sigma)$, where~$k \in \Z_{++}$ is the fixed number of vehicles,~$C \in \Q_+$ is the vehicle capacity, $\bar{d} \in \Q^{V_+}_+$ is the vector of expected demands, and~$\sigma \in \Q^{V+}_+$ is the vector of standard deviations.

Our goal is to find feasible routing plans that cover all customers in~$G$ with a collection of routes. Formally, we define a \emph{route}~$R \subseteq G$ as a simple undirected cycle that starts and ends at the depot, i.e.,~$V(R) = \{0, v_1, v_2, \ldots, v_\ell\}$ and~$E(R) = \{\{0, v_1\}, \{v_1, v_2\}, \ldots, \{v_\ell, 0\}\}$, where all customers~$v_i$ are distinct. The notation~$V_+(R)$ refers to the set of customers inside~$R$, that is,~$V_+(R) \coloneqq \{v_1, \ldots, v_\ell\}$. If~$R$ is a route containing a single customer~$v$, then~$E(R)$ denotes a multiset that contains edge~$\{0, v\}$ with multiplicity 2. %
For convenience, we use~$c(R)$ as a shorthand to~$c(E(R))$. Additionally, we often represent~$R$ with the tuple~$(v_1, \ldots, v_\ell)$ and, in this context, we assume that~$v_0 = v_{\ell + 1} = 0$. A \emph{routing plan} is a set of routes~$\mc{R} = \{R_1, \ldots, R_t\}$ such that the sets~$\{V_+(R_i)\}_{i \in [t]}$ form a partition of~$V_+$ (note that~$t$ is not fixed).

Routes are undirected graphs, so if~$R = (v_1, \ldots, v_\ell)$ and~$R' = (v_\ell, \ldots, v_1)$ are both routes, then~$R = R'$. However, we sometimes have to refer explicitly to the different orientations of a route. To this end, we associate with route~$R = (v_1, \ldots, v_\ell)$ two digraphs~$\oa{R}$ and~$\cev{R}$, which we call \emph{directed routes}, or \emph{orientations} of~$R$. Both~$\oa{R}$ and~$\cev{R}$ have the same vertex set as~$R$, but the arcs are in opposite directions, that is,~$A(\oa{R}) = \{(0, v_1) \ldots, (v_\ell, 0)\}$ and~$A(\cev{R}) = \{(0, v_\ell) \ldots, (v_1, 0)\}$ \label{sym:dir_route}. Similarly to the routes, we write~$\oa{R} = (v_1, \ldots, v_\ell)$ and~$\cev{R} = (v_\ell, \ldots, v_1)$, and since these are directed graphs,~$\oa{R}$ differs from~$\cev{R}$ whenever~$\ell \geq 2$. We need to clarify a detail here: strictly speaking, our notation is ambiguous, since if~$R = (v_1, \ldots, v_\ell)$ and~$R' = (v_\ell, \ldots, v_1)$ are both routes, then~$R = R'$ and the notation~$\oa{R}$ might refer to either~$(v_1, \ldots, v_\ell)$ or~$(v_\ell, \ldots, v_1)$. In such situations, we always assume that the arrows in the notation are according to how we first write the tuple for the underlying (undirected) route, so even though~$R = R'$, we have that~$\oa{R} \neq \oa{R'}$ and~$\cev{R} = \oa{R'}$. When the sequence of customers defining~$R$ is not specified,~$\oa{R}$ is an arbitrary (but fixed) orientation of~$R$.

In the rest of this paper, we fix~$\mc{Q}$ to denote a generic \emph{recourse function}~\citep{ota2024hardness}\label{sym:Q}, that is,~$\mc{Q}$ is a function that takes the input~$\mc{I}$ as a parameter and maps each route~$R$ to a value~$\mc{Q}(R ; \mc{I}) \in \Q_+$. Since the instance~$\mc{I}$ is fixed, we write~$\mc{Q}(R)$ instead of~$\mc{Q}(R ; \mc{I})$. While the correctness of the approaches proposed here depends only on~$\mc{Q}$ returning nonnegative rational numbers, our algorithms evaluate~$\mc{Q}$ at multiple routes. Therefore, in practice, we also assume access to an algorithm that, given a route~$R$, computes~$\mc{Q}(R; \mc{I})$ efficiently (say in polynomial or pseudo-polynomial time in the size of~$\mc{I}$).

We next turn our attention to formalizing our assumptions on the set of feasible routing plans by means of an edge-based formulation for the problem. By using the classical \emph{subtour elimination constraints} (SECs)~\citep{TothV02}, we have a bijection between the set of all routing plans and the integer vectors inside
\begin{equation}
\tag{$\mc{X}_{\tsc{sub}}$}
\label{set:subtour}
\mc{X}_{\tsc{sub}} =
\left\{x \in [0, 2]^{E}:~
\begin{aligned}
& x(\delta(v)) = 2, & \forall v \in V_+ \\
& x(E(S)) \leq |S| - 1, & \forall \emptyset \subsetneq S \subseteq V_+
\end{aligned}
\right\}.
\end{equation}
With each vector~$\bar{x} \in \Xsub \cap \Z^E$ we denote its corresponding routing plan with the notation~$\mc{R}(\bar{x})$. Specifically, for any route~$R$, let~$\mathbbm{1}_R$ denote the characteristic vector of~$E(R)$ (if~$R = (v)$, then~$(\mathbbm{1}_R)_{0v} = 2$). The vector~$\bar{x}$ can then be written as~$\bar{x} = \sum_{R \in \mc{R}(\bar{x})} \mathbbm{1}_R$. Additionally, for any vector~$x' \in \R^E_+$, we use~$G(x')$ to refer to its \emph{support graph}, that is,~$V(G(x')) = V$ and~$E(G(x')) = \{e \in E : x'_e > 0\}$\label{sym:support_graph}.

In many vehicle routing problems, additional  problem-specific intra/inter-route constraints are imposed to define what is a feasible routing plan; for example, bounds on the number of routes, capacity restrictions, and time windows~\citep{TothV02}. These constraints are frequently handled well by existing formulations with additional inequalities and variables. To isolate the role of the recourse function~$\mc{Q}$, we make the following assumption on the set of feasible routing plans (note that the set~$\mc{X}$ below could potentially be given as the projection of a higher-dimensional polyhedron).

\begin{assumption}
    \label{assumption:formulation}
    We are given a linear programming (LP) formulation of a polytope~$\mc{X} \subseteq \Xsub$ such that the set of feasible routing plans is given by~$\{\mc{R}(x) : x \in \mc{X} \cap \Z^E\}$.
\end{assumption}

We can now formally define the class of problems that we address (for which~$\mc{I}$ serves as input):

\begin{definition}
    \label{def:problem_vrp}
    The \emph{Vehicle Routing Problem with Recourse} (VRPR) with respect to~$\mc{Q}$ and~$\mc{X}$ seeks a routing plan~$x\in \mc{X}\cap \Z^E$ that minimizes~$\sum_{R \in \mc{R}(x)} [c(R) + \mc{Q}(R)]$.
\end{definition}

\noindent
Similarly to how we treat recourse functions, we define all the mathematical objects in this paper (such as functions and sets) relative to~$\mc{I}$, without making this dependence explicit in the notation.

Defining~$\mc{Q}(\bar{x}) \coloneqq \sum_{R \in \mc{R}(\bar{x})} \mc{Q}(R)$, for every~$\bar{x} \in \X \cap \Z^E$, we then express problem VRPR as
\begin{equation}
    \label{problem:vrpr}
    \tag{$\tsc{vrpr}(\mc{Q}, \mc{X})$}
    \min \{c^\T x + \rho : \rho \geq \mc{Q}(x),~ x \in \X \cap \Z^E \}.
\end{equation}

\noindent
The feasible region of~\ref{problem:vrpr} is denoted as~$\epi(\mc{Q}, \mc{X})$, since it can be interpreted as the \emph{epigraph} (see Section 2.11 (c) of~\cite{birge2011introduction}) of the function that maps each~$x \in \R^E$ to~$\mc{Q}(x)$ if~$x \in \X \cap \Z^E$, and to~$+\infty$ otherwise.

\subsection{Motivating application: the VRPSD with scenarios under the classical recourse policy}
\label{subsection:vrpsd}

Besides the original graph~$G = (V, E)$ and the edge costs vector~$c \in \Q^E_+$, the VRPSD also receives as input the vehicle capacity~$C \in \Q_{++}$ and the number of available vehicles~$k \in \Z_{++}$. The input data related to the stochastic customer demands is described next. 

Let~$d$ be a random vector following a probability distribution~$\mb{P}$, where each component~$d(v)$ indicates the random demand of customer~$v \in V_+$. We assume that~$\mb{P}$ is \emph{given by scenarios}, meaning that we receive vectors~$d^1, \ldots, d^{N} \in \Q^{V_+}_+$ in the input, each representing a certain scenario. We are also given probabilities~$p_1, \ldots, p_N \in \Q_+$ that sum up to one and such that~$\mb{P}(d = d^\xi) = p_\xi$, for every~$\xi \in [N]$. The expected demand vector is denoted~$\bar{d} \coloneqq \mb{E}[d]$ and every entry of~$\bar{d}$ is assumed to be strictly positive\label{sym:demands}. The input for the VRPSD with scenarios is represented by the tuple~$\mc{I}_{\tsc{vrpsd}} = (G, c, k, C, N, d^1, \ldots, d^N, p_1, \ldots, p_N)$.

\subsubsection{CVRP formulation}
\label{subsection:cvrp}

We say that a routing plan~$\mc{R}$ is feasible for the VRPSD if it is feasible for the CVRP with respect to the demand vector~$\bar{d}$, that is,~$\bar{d}(R) \coloneqq \sum_{v \in V_+(R)} \bar{d}(v) \leq C$, for every~$R \in \mc{R}$. (Instead, some works~\citep{florio2020, legault2025superadditivity, hoogendoorn2025evaluation} use~$\bar{d}(R) \leq fC$, where~$f > 0$ is called a \emph{load factor}. Most of the literature uses~$f = 1$, but~\cite{hoogendoorn2025evaluation} recently considered a version of the problem where~$f$ is large enough so that no capacity constraint is imposed.)

Using the classical CVRP formulation of~\cite{laporte1983branch}, we model the set of routing plans feasible for the VRPSD as the integer vectors belonging to the polytope
\begin{equation}
\tag{$\mc{X}_{\tsc{cvrp}}$}
\label{set:cvrp}
\mc{X}_{\tsc{cvrp}} := \mc{X}_{\tsc{sub}} \cap 
\left\{x \in [0, 2]^{E}:~
\begin{aligned}
& x(\delta(0)) = 2 k, & \\
& x(S) \leq |S| - \bar{k}(S), & \forall \emptyset \subsetneq S \subseteq V_+
\end{aligned}
\right\},
\end{equation}
where~$\bar{k}(S) \coloneqq \lceil \bar{d}(S) / C\rceil$.  Inequalities~$x(S) \leq |S| - \bark{S}$ are the well-known \emph{rounded capacity inequalities} (RCIs) (in their ``inside form'') and they imply the SECs since~$\bar{d}(S)$ is positive.

\subsubsection{The classical recourse policy}
\label{subsection:classical_recourse}

When demands are independent, several recourse policies have been proposed in the VRPSD literature~\citep{dror89, Yee1980, Salavati2019175, salavati2019trsc}. However, to our knowledge, no existing work explicitly addresses recourse policies for the VRPSD with scenarios. Still, the \emph{classical recourse policy} can be easily adapted as follows.

Consider traversing a directed route~$\oa{R} = (v_1, \ldots, v_\ell)$ and suppose that the sum of the realized demands exceeds the vehicle capacity when we reach customer~$v_j$, with~$j \in [\ell]$. In this case, following standard conventions in the literature (see~\cite{OYOLA2018193} for variants that consider exact stockouts and nonsplittable demands), the classical recourse policy prescribes that the vehicle executes a back-and-forth trip between the depot and~$v_j$ before continuing the route. Based on the formula of~\cite{dror89}, the recourse cost of the directed route~$\oa{R}$ in scenario~$\xi \in [N]$ under the classical recourse policy is given by
\begin{equation}
    \tag{$\mc{Q}^\xi_C$}
    \label{eq:formula_dror}
    \mc{Q}^\xi_C(\oa{R}) \coloneqq \sum_{j \in [\ell]} 2 \, c_{0v_j} \sum_{t = 1}^\infty \, \mb{I}\left(\sum_{i \in [j - 1]} d^\xi(v_i) \leq t C < \sum_{i \in [j]} d^\xi(v_i)\right).
\end{equation}

The expected recourse cost of the directed route~$\oa{R}$ is then defined as~$\mc{Q}_C(\oa{R}) \coloneqq \sum_{\xi \in [N]} p_\xi \Qcscen(\oa{R})$, and the recourse cost of the (undirected) route~$R$ is set as~$\mc{Q}_C(R) \coloneqq \min\left\{ \mc{Q}_C(\oa{R}),  \mc{Q}_C(\cev{R}) \right\}$. From now on, we shall refer to~$\Qc$ as the \emph{classical recourse function}. By computing the accumulated demands along the directed routes in each scenario, we can evaluate~$\Qc(R)$ in polynomial time on the number of scenarios~$N$ and the size of~$V_+$. The \emph{VRPSD with scenarios under the classical recourse policy} can now be concisely expressed as~$\vrpr(\Qc, \Xcvrp)$.

\section{Existing disaggregated ILS formulations and the motivation for a unifying framework}
\label{section:review}

In the context of the general VRPR introduced in Definition~\ref{def:problem_vrp}, we now briefly review two existing ILS formulations for the VRPSD: the partial route-based formulation of~\cite{hoogendoorn2023improved} (Section~\ref{subsection:review_hoogendoorn}), and the DL-shaped method of~\cite{parada2024disaggregated} (Section~\ref{subsection:review_dl-shaped}). We focus on these approaches because the formulation of~\cite{hoogendoorn2023improved} is valid for any recourse function and yields the best-performing algorithm among all methods with this generality. On the other hand, the DL-shaped method was only shown to be valid for superadditive functions~\citep{legault2025superadditivity}, in which case it outperforms the algorithm of~\cite{hoogendoorn2023improved}. 

As mentioned in Section~\ref{section:intro}, both formulations replace the epigraph variable~$\rho$ with the sum~$\sum_{v \in V_+} \theta_v$, but they differ in how the value of~$\rho$ is decomposed across the~$\theta_v$-variables. This distinction is examined further in Section~\ref{subsection:limitations}, where we highlight the limitations of the existing approaches and motivate the development of a unifying framework.

In the remainder, we shall repeatedly use the following definition.

\begin{definition}
    \label{def:activation_function}
    Let~$\mc{X}' \subseteq \X \cap \Z^E$. A function~$W( \, \cdot \, ; \mc{X}') : \R^E \to \R$ is an \emph{activation function} with respect to~$\mc{X}'$ if~$W( \, \cdot \, ; \mc{X}')$ is affine and, for every~$x \in \X \cap \Z^E$, 
    \begin{align*} 
    W(x ; \mc{X}') 
    \begin{cases} = 1, & \text{if } x \in \mc{X}', \\ 
    \leq 0, & \text{otherwise.} 
    \end{cases}
    \end{align*}
    We say that~$W(x ; \mc{X}')$ is \emph{active} at~$\mc{X}'$ and \emph{inactive} at~$(\mc{X} \cap \Z^E) \setminus \mc{X}'$.
\end{definition}

\noindent
Note that Definition~\ref{def:activation_function} depends on the set~$\X$ that we fixed in Assumption~\ref{assumption:formulation}. To avoid repeating ourselves, whenever we say that~$\W( x ; \mc{X}')$ is an activation function, it is implicit that it is with respect to~$\mc{X}' \subseteq \X \cap \Z^E$.

\subsection{Partial route inequalities and route-based recourse disaggregation}
\label{subsection:review_hoogendoorn}

In this subsection, we consider ILS cuts that are active at solutions adhering to certain \emph{partial route} structures, which are formally defined as follows.
\begin{definition}
    \label{def:partial_route}
    Let~$H = (S_1, \ldots, S_\ell)$ be a tuple of disjoint subsets of~$V_+$. For each~$i \in [\ell]$, we call~$S_i$ an \textit{unstructured component} if~$|S_i| > 1$. We say that~$H$ is a \emph{partial route} if there exists no index~$i \in [\ell - 1]$ such that both~$S_i$ and~$S_{i + 1}$ are unstructured components.
\end{definition}
For convenience, we set~$S_0 = S_{\ell+ 1} = \{0\}$ and we sometimes regard~$\H$ as an undirected graph, meaning that we define~$V(\H) \coloneqq \cup_{i \in [\ell + 1]} S_i$ and~$E(\H) \coloneqq \cup_{i \in [\ell]} (E(S_i, S_{i - 1} \cup S_{i + 1}) \cup E(S_i))$. Following the notation for routes, the customers in~$\H$ are denoted by~$V_+(\H) \coloneqq V_+ \cap V(\H)$.

Note that partial routes can be viewed as generalizations of routes~\citep{ hoogendoorn2023improved}, since we may interpret a route~$R = (v_1, \ldots, v_\ell)$ as a partial route~$\H = (S_1, \ldots, S_\ell)$, where~$S_i = \{v_i\}$ for every~$i \in [\ell]$. In such cases, we say that the partial route~$\H$ \emph{corresponds} to route~$R$, and we write~$\H = R$.

A partial route~$\H = (S_1, \ldots, S_\ell)$ describes a partial ordering of the customers in a route as follows:
\begin{definition}
    \label{def:exact_adherence}
    Let~$\H = (S_1, \ldots, S_\ell)$ be a partial route. For convenience, for each~$i \in [\ell]$, define~$h_i \coloneqq \sum_{j \in [i]} |S_j|$. A directed route~$\oa{R} = (v_1, \ldots, v_t)$ \emph{adheres} to~$\H$ if~$t = h_\ell$ and, for every~$i \in [\ell]$,~$S_i = \left\{v_{h_{i - 1} + 1}, \ldots, v_{h_i}\right\}$.
    In other words,
    $$\oa{R} = (\underbrace{v_1, \ldots, v_{h_1}}_{S_1}, \ldots, \underbrace{v_{h_{i - 1} + 1}, \ldots, v_{h_i}}_{S_i}, \ldots, \underbrace{v_{h_{\ell - 1} + 1}, \ldots, v_{h_\ell}}_{S_\ell}).$$

    \noindent
    An (undirected) route~$R$ \emph{adheres} to~$\H$ if either~$\oa{R}$ or~$\cev{R}$ adheres to~$H$.
\end{definition}

With these definitions in hand, we use~$\mc{X}_{=}(\H)$ to refer to the set of solutions containing a route that adheres to partial route~$\H$:
\begin{align}
    & \mc{X}_{=}(H) \coloneqq \left\{ x \in \X \cap \Z^E :~\text{there exists~$R \in \mc{R}(x)$ such that~$R$ adheres to~$\H$} \right\}. \tag{$\mc{X}_{=}(H)$} \label{set:x_h}
\end{align}
\cite{hoogendoorn2023improved} use these concepts to define an activation function~$\Whs(x ; \mc{X}_{=}(\H))$ which, by definition, is active only at the solutions containing a route that adheres to~$\H$.

\begin{remark}
\label{remark:minor_contribution}
As a minor contribution, Appendix~\ref{appendix:activation_function} presents a derivation of~$\Whs(x ; \Xh{H})$ that we believe is simpler and more intuitive than the one in~\cite{hoogendoorn2023improved}. This simplification might be particularly valuable given the history of partial route activation functions: they were first proposed for the multiple vehicle case by~\cite{laporte2002}, later shown to be incorrect by~\cite{jabali2014}, whose own activation functions were also recently shown to be incorrect by~\cite{hoogendoorn2023improved}.~\hfill\Halmos
\end{remark}

Recall from Section~\ref{section:intro} that~\cite{hoogendoorn2023improved} adopt a \emph{route-based recourse disaggregation}, where the variable~$\rho$ is replaced by the sum of variables~$\sum_{v \in V_+} \theta_v$. The recourse cost of a partial route~$\H$ is then mapped to the customer~$v_H$ in the partial route with the lowest index. In particular, the recourse cost~$\mc{Q}(R)$ of a route~$R$ is mapped to the lowest-index customer~$v_R \in V_+(R)$. As observed by the authors, this disaggregation enables ILS cuts that capture the recourse cost of a route~$R$ at every solution~$x \in \X \cap \Z^E$ with~$R \in \mc{R}(x)$. Such cuts were not available in previous formulations, which do not use recourse disaggregation and capture the recourse cost only through the~$\rho$ variable.

Now suppose that, for every partial route~$\H$, we have access to a lower bound~$\mc{L}(\H)$ satisfying~$\mc{Q}(R) \geq \mc{L}(\H)$, for every route~$R$ that adheres to~$\H$. Moreover, whenever~$\H$ corresponds to a route~$R$, we assume that~$\mc{L}(\H) = \mc{Q}(R)$. The formulation proposed by~\cite{hoogendoorn2023improved} is as follows:
\begin{subequations}
\label{formulation:partial_route}
\begin{align} 
\min ~~& c^\T x + \mathbf{1}^\T \theta, & \nonumber \\
\text{s.t.~~} & \theta_{v_H} \geq \mc{L}(H) \cdot \Whs(x ; \Xh{H}), & \forall \text{~partial route~$H$}, \label{partial_route:cuts} \\
& (x, \theta) \in (\X \cap \Z^E) \times \R^{V_+}_+. &
\end{align}
\end{subequations}

\noindent
Inequalities~\eqref{partial_route:cuts} are their \emph{partial route-split inequalities}. Since they generalize the so-called \emph{route-split inequalities}~$\theta_{v_R} \geq \mc{Q}(R) \cdot \Whs(x ; \Xh{R})$ (for every route~$R$), it follows that Formulation~\eqref{formulation:partial_route} correctly models problem~$\vrpr(\mc{Q}, \X)$. Although this claim was already stated in~\cite{hoogendoorn2023improved} in the context of the VRPSD, a formal proof was not provided. Using our framework, we prove a more general result later in Theorem~\ref{thm:route_formulation}.

Lastly, we observe that~\cite{hoogendoorn2023improved} separate the partial route-split inequalities using a heuristic procedure, with the separation being exact at the integer solutions. Since their heuristic is described somewhat informally, we give a precise implementation of their ideas in Appendix~\ref{appendix:partial_route_separation}. We note in passing that~\cite{hoogendoorn2023improved} also introduced the \emph{multi-route-split inequalities}, but we do not consider them in this work, since their associated activation functions may return values greater than one and thus fall outside the scope of Definition~\ref{def:activation_function} (which we use later to formally define ILS cuts).

\subsection{The DL-shaped method and customer-based recourse disaggregation}
\label{subsection:review_dl-shaped}

In the context of the VRPSD with independent demands, the approach of~\cite{hoogendoorn2023improved} was significantly improved by the \emph{DL-shaped method}~\citep{parada2024disaggregated, legault2025superadditivity}. Rather than separating partial route inequalities, the DL-shaped method separates two classes of ILS cuts: \emph{path cuts} and \emph{set cuts}. Following the presentations of the authors, we assume throughout this subsection that every first-stage solution~$x \in \X \cap \Z^E$ and route~$R \in \mc{R}(x)$ satisfies~$\bar{d}(R) \le C$.

To discuss the path cuts, we first need some concepts. Let~$R = (v_1, \ldots, v_\ell)$ be a route. A route~$R'$ is a \emph{subroute} of~$R$ if there exists~$j \in [\ell]$ and~$i \in [j]$ such that~$R' = (v_i, \ldots, v_j)$. We write~$R' \subseteq R$ to
indicate that~$R'$ is a subroute of~$R$ (even though~$R'$ is not necessarily a subgraph of~$R$). The notation~$\mc{X}_{\supseteq}(R')$ is then used to refer to the set of solutions that contains~$R'$ as a subroute, that is,
\begin{align}
    & \mc{X}_{\supseteq}(R') \coloneqq \left\{ x \in \X \cap \Z^E :~\text{there exists~$R \in \mc{R}(x)$ such that~$R' \subseteq R$} \right\}. \tag{$\mc{X}_{\supseteq}(R')$} \label{set:path_cut}
\end{align}
(The proof in Appendix~\ref{appendix:activation_function} explains why we use~$=$ and~$\supseteq$ in the subscripts of~$\Xh{H}$ and~$\supsetXh{R'}$.) A \emph{path cut} is an inequality of the form~$\theta(V_+(R')) \geq \mc{Q}(R') \cdot W_{DL}(x ; \supsetXh{R'})$, where~$W_{DL}(x ; \supsetXh{R'}) \coloneqq 1 + \sum_{e \in E(R') \setminus \delta(0)} (x_e - 1)$ is the activation function proposed by~\cite{parada2024disaggregated}.

On the other hand, a set cut is an ILS inequality that is active whenever a set of customers is served with the minimum required number of routes. Formally, for every~$\emptyset \subsetneq S \subseteq V_+$ and~$k' \in \Z_{++}$, define
\begin{align}
    & \mc{X}(S, k') \coloneqq \left\{ x \in \X \cap \Z^E : x(E(S)) = |S| - k' \right\}~ \text{ and } \label{set:x_set} \tag{$\mc{X}(S, k')$} \\
    & W_{DL}(x ; \mc{X}(S, k')) \coloneqq 1 + (x(E(S)) - |S| + k'). \label{eq:activation_set} \tag{$W_{DL}(x ; \mc{X}(S, k'))$}
\end{align}
A \emph{set cut} is an inequality of the form~$\theta(S) \geq \mc{L}(S) \cdot \setWdl{x ; \Xset{S, \bark{S}}}$, where~$\mc{L}(S)$ is a ``properly defined'' recourse lower bound and we recall that~$\bark{S} = \lceil \bar{d}(S) / C \rceil$\label{sym:LS}.

One of the main contributions of this paper is to precisely characterize what constitutes a valid recourse lower bound (this is part of Contribution~\ref{cont:recoursedisag}). For the moment, we describe the specific recourse lower bound~$\mc{L}(S) = \mc{L}_{DL}(S)$ used by the DL-shaped method. This bound is defined so that it provides a lower bound on the recourse cost of any collection of routes~$R$ (with~$\bar{d}(R) \leq C$) that forms a partition of~$S$. More specifically, we have that~$\mc{L}_{DL}(S) \leq \sum_{i = 1}^{\bark{S}} \mc{Q}(R_i)$ for every collection of routes~$R_1, \ldots, R_{\bark{S}}$ such that~$\{V_+(R_i)\}_{i \in [\bark{S}]}$ forms a partition of~$S$ and~$\bar{d}(R_i) \le C$ for every~$i \in [\bark{S}]$. 

Combining the path and set cuts yields the formulation below:
\begin{subequations}
\label{formulation:DL-shaped}
\begin{align} 
\min ~~& c^\T x + \mathbf{1}^\T \theta, & \nonumber \\
\text{s.t.~~} & \theta(V_+(R)) \geq \mc{Q}(R) \cdot \Wdl{x ; \supsetXh{R}}, & \forall \text{~route~$R$}, \label{DL-shaped:path_cuts} \\
& \theta(S) \geq \mc{L}_{DL}(S) \cdot \setWdl{x ; \Xset{S, \bark{S}}}, & \forall \emptyset \subsetneq S \subseteq V_+, \label{DL-shaped:set_cuts}\\
& (x, \theta) \in (\X \cap \Z^E) \times \R^{V_+}_+, &
\end{align}
\end{subequations}
and we refer the reader to~\citep{parada2024disaggregated} for the separation of the path and set cuts.

In contrast to Formulation~\eqref{formulation:partial_route}, an optimal solution~$(\bar{x}, \bar{\theta})$ to Formulation~\eqref{formulation:DL-shaped} may attain
positive~$\bar{\theta}_v$-values at multiple customers~$v$ on the same route~$R \in \mc{R}(\bar{x})$. In this sense,~\cite{legault2025superadditivity}~state that the DL-shaped method uses a \emph{customer-based recourse disaggregation}. This alternative disaggregation allows the \mbox{DL-shaped} method to use the set cuts to strengthen the bounds obtained from the path cuts, which leads to strong dual bounds in practice, as the set cuts can be active over a wide range of solutions. However, as we discuss in the next subsection, Formulation~\eqref{formulation:DL-shaped} is not, in general, a valid reformulation of problem~$\vrpr(\mc{Q}, \X)$. In particular, Formulation~\eqref{formulation:DL-shaped} cannot be used to solve our target problem~$\vrpr(\Qc, \Xcvrp)$.

\subsection{Limitations and the need for a unifying framework}
\label{subsection:limitations}

As mentioned in Section~\ref{subsection:review_hoogendoorn}, the approach of~\cite{hoogendoorn2023improved} is valid for any recourse function~$\mc{Q}$. On the other hand, although the DL-shaped method achieved superior performance for VRPSDs with independent demands, the method was only shown to be valid for \emph{superadditive} recourse functions.

\subsubsection{Superadditive recourse functions}
\label{subsection:limitation_superadditive}

To explain superadditivity, we need some additional notation. Given two routes~$R_1 = (v_1, \ldots, v_a)$ and~$R_2 = (v_{a+1}, \ldots, v_\ell)$ with~$V_+(R_1) \cap V_+(R_2) = \emptyset$, we denote by~$R_1 \oplus R_2$ the concatenation of~$R_1$ and~$R_2$, i.e.,~$R_1 \oplus R_2 = (v_1, \ldots, v_\ell)$\label{sym:concat}. Thus, whenever we write~$R_1 \oplus R_2$~(without explicitly declaring~$R_1$ and~$R_2$ beforehand), we implicitly assume that~$R_1$ and~$R_2$ are subroutes of~$R_1 \oplus R_2$ visiting disjoint customer sets whose concatenation yields~$R_1 \oplus R_2$.~\cite{legault2025superadditivity} define a recourse function~$\mc{Q}$ as \emph{superadditive} if, for every route~$ R_1 \oplus R_2$ satisfying~$\bar{d}(R_1 \oplus R_2) \leq C$, we have that~$\mc{Q}(R_1 \oplus R_2) \geq \mc{Q}(R_1) + \mc{Q}(R_2)$ (recall that the DL-shaped method assumes that a feasible route~$R$ satisfies~$\bar{d}(R) \leq C$).

The main theorem of~\cite{legault2025superadditivity} states that the DL-shaped method yields a valid reformulation of problem~$\vrpr(\mc{Q}, \X)$ if and only if~$\mc{Q}$ is superadditive. However, to show necessity, they rely on the assumption that an unlimited number of vehicles is available, leaving open whether the DL-shaped method could still be valid when~$k$ is fixed, as is typical in the literature. In Section~\ref{subsection:path_cuts}, we close this gap by introducing a weaker notion of \emph{restricted superadditivity} that characterizes the validity of Formulation~\eqref{formulation:DL-shaped} without additional assumptions on~$\mc{X} \cap \Z^E$ (see Contribution~\ref{cont:rsuper}). Note that, even with our new result the DL-shaped method  cannot be used in our setting, since  the recourse function~$\Qc$ is not restrictively superadditive (Example~\ref{example:qc_not_superadditive}).

\subsubsection{Combining partial route inequalities and set cuts}
\label{subsection:limitation_combining}

Although the DL-shaped method cannot be applied to solve problem~$\vrpr(\Qc, \Xcvrp)$, much of its empirical success comes from the use of the set cuts, whose activation functions are typically active over a large set of points in~$\X$. Indeed, for a given route~$R$ and~$\bar{x} \in \X$, we have that~$\Wdl{\bar{x} ; \supsetXh{R}} = 1$ if and only if~$\bar{x}_e$ is integral for every~$e \in E(R) \setminus \delta(0)$. In contrast, for~$S = V_+(R)$, we have that~$\setWdl{\bar{x} ; \Xset{S, \bark{S}}}$ might be equal to one even when there exists an edge~$e \in E(R) \setminus \delta(0)$ with a fractional value~$\bar{x}_e \in (0, 1)$.

This naturally raises the question of whether the set cuts could also be incorporated into the more flexible Formulation~\eqref{formulation:partial_route}. Unfortunately, as noted in Contribution~\ref{cont:recoursedisag}, this cannot be done effectively, since the route-based disaggregation may prevent the use of meaningful recourse lower bounds in several of the set cuts. 

To better explain this issue, consider adding (possibly many) ILS cuts of the form
\begin{equation}
    \label{ineq:ils_limitation}
    \theta(S) \geq \mc{L}(S) \cdot W(x ; \mc{X}'),
\end{equation}
where~$\emptyset \subsetneq S \subseteq V_+$,~$\mc{X}' \subseteq \X \cap \Z^E$,~$\mc{L}(S) \in \R_{+}$, and~$W(x ; \mc{X}')$ is an activation function. For the moment, we say that these inequalities are \emph{valid} if, for every~$\bar{x} \in \X \cap \Z^E$, there exists~$\bar{\theta} \in \R^{V_+}_+$ satisfying every ILS cut~$\eqref{ineq:ils_limitation}$ and such that~$\allones^\T \bar{\theta} = \mc{Q}(\bar{x})$. In this way, the validity of the inequalities guarantees that we do not cut off any feasible solution in the~$(x, \rho)$-space under the projection~$\rho = \allones^\T \theta$.

Now consider the instance of~$\vrpr(\Qc, \Xcvrp)$ illustrated in Figure~\ref{figure:example_set_cut}, which is based on the instance for the VRPSD with independent demands used in Theorem~4 of~\cite{legault2025superadditivity}. The figure also shows a feasible solution~$x' \in \Xcvrp \cap \Z^E$, with~$\mc{R}(x') = \{(v_1, v_2, v_3, v_4), (v_5)\}$. (This instance will be used later in Example~\ref{example:qc_not_superadditive} to show that~$\Qc$ is not restrictively superadditive.)

\begin{figure}[htb!]
  \centering
  \includegraphics[width=\textwidth]{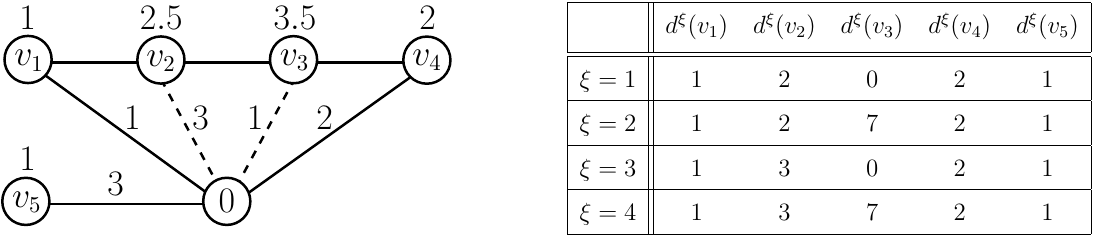}
  \caption{Instance of the VRPSD with scenarios with~5 customers,~$k = 2$,~$C = 10$ and~$N = 4$. The numbers next to the edges correspond to the cost of the edges incident to the depot, while the numbers on top of the vertices refer to the vector of expected demands~$\bar{d} \in \Q^{V_+}_{++}$. Each scenario~$\xi \in [4]$ have a realization probability of~$p_\xi = 1/4$. The table on the right shows the scenario demands vectors~$d^\xi \in \Q^{V_+}_{+}$. Note that failures may occur only in scenarios~$\xi = 2$ and~$\xi = 4$. \label{figure:example_set_cut}}
\end{figure}

Let~$R = (v_1, v_2, v_3, v_4)$. Since~$\Qc(R) = \Qc(\oa{R}) = (2 \cdot c_{0 v_3} + 2 \cdot c_{0 v_4}) / 4 = 3 / 2$, the following (partial) \mbox{route-split} inequality is valid:
\begin{equation} 
\label{ineq:route_cut_limitation} 
\theta_{v_1} \geq \frac{3}{2} \cdot \Whs(x ; \Xh{R}) = \frac{3}{2} \cdot (x_{0, v_1} + 2 x_{v_1, v_2} + x_{v_2, v_3} + 2 x_{v_3, v_4} + x_{0 v_4} - 6). 
\end{equation}
Indeed, for~$\theta^1 = [3/2, 0, 0, 0, 0]^\T$ we have that~$(x', \theta^1)$ satisfies~\eqref{ineq:route_cut_limitation} and~$\allones^\T \theta^1 = \Qc(x')$. For all other solutions~$\bar{x} \in \X \cap \Z^E \setminus \{x'\}$, we have~$\Whs(\bar{x} ; \Xh{R}) \leq 0$, so we can take~$\bar{\theta} = [\mc{Q}(\bar{x}), 0, 0, 0, 0]^\T$. (Note that validity here concerns only inequality~\eqref{ineq:route_cut_limitation}, not all inequalities in Formulation~\eqref{formulation:partial_route}.)

Now let~$S = \{v_2, v_3, v_4\}$ and consider the set cut
\begin{equation} \label{ineq:set_cut_limitation} \theta_{v_2} + \theta_{v_3} + \theta_{v_4} \geq \mc{L}(S) \cdot \setWdl{x ; \mc{X}(S, 1)} = \mc{L}(S) \cdot (x_{v_2, v_3} + x_{v_3, v_4} + x_{v_2, v_4} - 1),
\end{equation}
where~$\mc{L}(S) = \min_{v \in S}\{2 c_{0v}\} \cdot ( \sum_{\xi \in [4] : d^\xi(S) > C} p_{\xi}) = 1$. We prove in Section~\ref{section:application_vrpsd} that inequality~\eqref{ineq:set_cut_limitation} is valid in general (see Contribution~\ref{cont:set}). For the small instance in Figure~\ref{figure:example_set_cut}, one may instead verify validity by inspecting every feasible solution~$x \in \X \cap \Z^E$. For example, for~$x = x'$, we set~$\theta^2 = [0, 0, (2 c_{0 v_3}) / 4, (2 c_{0 v_4}) / 4]^\T = [0, 0, 1/2, 1]^\T$ and observe that~$(x', \theta^2)$ satisfies~\eqref{ineq:set_cut_limitation} and~$\allones^\T \theta^2 = \Qc(x')$. 

We thus conclude that each of inequalities~\eqref{ineq:route_cut_limitation} and~\eqref{ineq:set_cut_limitation} is individually valid, but they are not \emph{jointly} valid. In particular, by summing the two inequalities, we learn that, for~$x = x'$, any~$\bar{\theta} \in \R^{V_+}_+$ that satisfies~\eqref{ineq:route_cut_limitation} and~\eqref{ineq:set_cut_limitation} also satisfies~$\allones^\T \bar{\theta} \geq 5 / 2$, overestimating the recourse cost of~$\mc{Q}(x') = 3/2$.

This example illustrates the challenges of extending the existing approaches to combine these different classes of cuts and motivates the development of our unifying framework (Contribution~\ref{cont:recoursedisag}). By explicitly defining the recourse disaggregation and precisely specifying the valid range of recourse lower bounds for ILS cuts, our approach allows the combination of (generalized) partial route inequalities~\citep{hoogendoorn2023improved} and set cuts~\citep{parada2024disaggregated} for essentially any recourse function, including~$\Qc$.

\section{The framework}
\label{section:framework}

This section introduces a unifying framework for ILS-based formulations of VRPRs, based on formal definitions of recourse disaggregation (Definition~\ref{def:recourse_disaggregation}) and recourse lower bounds (Definition~\ref{def:recourse_lower_bound}). The framework represents the epigraph~$\epi(\mc{Q}, \mc{X})$ in an extended space where the recourse variable~$\rho$ is expressed as the sum of disaggregated recourse variables. In Theorem~\ref{theorem:ils}, we then characterize when valid inequalities for this extended formulation recover the epigraph after projection.

\subsection{Recourse disaggregation}
\label{subsection:disaggregated}

Recall from Section~\ref{section:review} that previous works~\citep{seguin, cote2020vehicle, hoogendoorn2023improved, parada2024disaggregated, legault2025superadditivity} propose disaggregating the recourse variable~$\rho$ over the set of customers~$V_+$ by writing~$\rho = \sum_{v \in V_+} \theta_v$. We generalize the domain~$V_+$ of the~\mbox{$\theta$-variables} to an arbitrary finite set~$\Omega$ as follows.

\begin{definition}
    \label{def:recourse_disaggregation}
    Let~$\Omega$ be a nonempty finite set and let~$\hat{\mc{Q}}$ be a function that maps routes and elements in~$\Omega$ to nonnegative rational values, i.e.,~$\hat{\mc{Q}}(R, \omega) \in \Q_+$, for every route~$R$ and~$\omega \in \Omega$. We say that~$\hat{\mc{Q}}$ is a \emph{disaggregation of~$\mc{Q}$ along~$\Omega$} if~$\mc{Q}(R) = \sum_{\omega \in \Omega} \hat{\mc{Q}}(R, \omega)$, for every~$\bar{x} \in \X \cap \Z^E$ and~$R \in \mc{R}(\bar{x})$.
\end{definition}

\noindent
We also define~$\Omega(R)$ as the \emph{support} of~$\hat{\mc{Q}}$ with respect to route~$R$, i.e.,~$\Omega(R) \coloneqq \{\omega \in \Omega : \hat{\mc{Q}}(R, \omega) > 0\}$.

Given a disaggregation~$\hat{\mc{Q}}$ as in Definition~\ref{def:recourse_disaggregation}, we define the \emph{feasible region}
\begin{equation}
    \tag{$\mc{F}$}
    \label{set:F}
    \mc{F}(\hat{\mc{Q}}, \mc{X}, \Omega) \coloneqq \left\{ (x, \theta) \in  (\X \cap \Z^E) \times \R^{\Omega}_+ : \theta_\omega \geq \sum_{R \in \mc{R}(x)} \hat{\mc{Q}}(R, \omega),~~\forall \omega \in \Om \right\}.
\end{equation}

\noindent
The next claim verifies that the definition of~$\F{\hat{\mc{Q}}, \mc{X}, \Omega}$ indeed yields a reformulation of~\ref{problem:vrpr} in the~\mbox{$(x, \theta)$-space}. From now on, for any set~$P \subseteq \R^E \times \R^\Omega$, we define~$$\projrho (P) \coloneqq \{(x, \rho) : \rho = \allones^\T \theta,~(x, \theta) \in P\} \label{sym:projrho}.$$
\begin{claim}
    \label{claim:rewrite_epigraph}
    Let~$\Om$ be a nonempty finite set and let~$\hat{\mc{Q}}$ be a disaggregation of~$\mc{Q}$ along~$\Om$, then~$$\epiQ{\mc{Q}, \X} = \projrho(\mc{F}(\hat{\mc{Q}}, \mc{X}, \Omega)).$$
\end{claim}
\proof
\begin{align}
    \epiQ{\mc{Q}, \X} & = \left\{(x, \rho) \in (\X \cap \Z^E) \times \R_+ : \rho \geq \sum_{R \in \mc{R}(x)} \sum_{\omega \in \Om} \hat{\mc{Q}}(R, \omega) \right\} \nonumber \\
    & = \left\{(x, \rho) \in (\X \cap \Z^E) \times \R_+ : \rho \geq \sum_{\omega \in \Om} \theta_\omega,~\theta_\omega = \sum_{R \in \mc{R}(x)} \hat{\mc{Q}}(R, \omega),~\forall \omega \in \Om \right\} \nonumber \\
    & = \projrho(\mc{F}(\hat{\mc{Q}}, \mc{X}, \Omega)). & \tag*{\Halmos}
\end{align}

\begin{remark}
\label{remark:disaggregation}
A disaggregation~$\hat{\mc{Q}}$ of~$\mc{Q}$ along~$\Om$ always exists, since we can choose an arbitrary element~$\omega \in \Om$ and set our disaggregation so that~$\hat{\mc{Q}}(R, \omega) = \mc{Q}(R)$, for every route~$R$. More interestingly, the structure of the given recourse function~$\mc{Q}$ often suggests disaggregations where the recourse cost of route~$R$ is distributed across multiple customers, allowing several customers to attain positive~$\hat{\mc{Q}}(R, \omega)$ values. For example, consider the formula for the VRPSD recourse function~$\Qc$ and set~$\Om = V_+$. Let~$R = (v_1, \ldots, v_\ell)$ be a route and assume without loss of generality that~$\Qc(R) = \Qc(\oa{R})$. In this case, we may set~$$\hat{\mc{Q}}_C(R, v_j) = 2 \, c_{0v_j} \sum_{t = 1}^\infty \mb{P}\left(\sum_{i \in [j - 1]} d(v_i) \leq t C < \sum_{i \in [j]} d(v_i)\right),$$
for each~$j \in [\ell]$ (and consequently,~$\hat{\mc{Q}}_C(R, v) = 0$, for all~$v \in V_+ \setminus V_+(R)$).~\hfill\Halmos
\end{remark}

For convenience, in the rest of this paper, we fix a nonempty finite set~$\Om$ and a corresponding disaggregation~$\hat{\mc{Q}}$. With a slight abuse of notation, we then write~$\mc{Q}(R, \omega)$ rather than~$\hat{\mc{Q}}(R, \omega)$. In particular, this implies that the notation~$\F{\mc{Q}, \mc{X}, \Omega}$ is well-defined, and for every route~$R$,~$\Om(R) = \{\omega \in \Omega : \mc{Q}(R, \omega) > 0\}$.

\paragraph{\textbf{Why define}~$\F{\mc{Q}, \mc{X}, \Omega}$\textbf{?}} 

Previous ILS-based algorithms that use recourse disaggregation do not explicitly specify the choice of the recourse disaggregation (as in Definition~\ref{def:recourse_disaggregation}) nor the corresponding feasible region~$\F{\mc{Q}, \mc{X}, \Omega}$. Instead, they demonstrate that their ILS cuts are \emph{valid} (in the sense discussed in Section~\ref{subsection:limitation_combining}). In other words, they show that feasible solutions~$(\bar{x}, \bar{\theta}) \in (\X \cap \Z^E) \times \Q^{V_+}_+$ to their formulation satisfy~$\allones^\T \bar{\theta} \geq \mc{Q}(\bar{x})$ and equality can be achieved.

This approach limits the ways in which new inequalities can be derived. For example, to show that the path cuts and set cuts of the DL-shaped method are \emph{jointly valid},~\cite{parada2024disaggregated} and~\cite{legault2025superadditivity} rely on the dominance of path cuts over set cuts at integer solutions, so their approach does not extend to formulations where path cuts are not used. In addition, the path cuts are valid only when the recourse function is \emph{restrictively superadditive} (Contribution~\ref{cont:rsuper}), which further limits the applicability of their set cuts.

A closely related drawback is that previous formulations may hide the fact that the validity of the inequalities depends on the choice of the recourse disaggregation. Suppose that~$\hat{\mc{Q}}^1$ and~$\hat{\mc{Q}}^2$ are two disaggregations of~$\mc{Q}$ along the same set~$\Om$. By Claim~\ref{claim:rewrite_epigraph}, both~$\F{\hat{\mc{Q}}^1, \mc{X}, \Omega}$ and~$\F{\hat{\mc{Q}}^2, \mc{X}, \Omega}$ are (nonlinear) \emph{extended formulations} of~$\epiQ{\mc{Q}, \mc{X}}$, that is,~$\projrho(\F{\hat{\mc{Q}}^1, \mc{X}, \Omega}) = \projrho(\F{\hat{\mc{Q}}^2, \mc{X}, \Omega}) = \epiQ{\mc{Q}, \X}$.
However, inequalities valid for~$\F{\hat{\mc{Q}}^1, \mc{X}, \Omega}$ may not be valid for~$\F{\hat{\mc{Q}}^2, \mc{X}, \Omega}$ (or vice versa), so one cannot simply combine inequalities derived under different disaggregations, i.e., it may be the case that~$\projrho\left(\F{\hat{\mc{Q}}^1, \mc{X}, \Omega} \cap \F{\hat{\mc{Q}}^2, \mc{X}, \Omega}\right) \subsetneq \epiQ{\mc{Q}, \X}$. Indeed, this issue was illustrated earlier by the example in Section~\ref{subsection:limitations}.

Our framework addresses both of the previous obstacles by explicitly defining the recourse disaggregation and the associated feasible region~$\F{\mc{Q}, \mc{X}, \Omega}$. This allows us to precisely determine the range of admissible recourse lower bounds (Claim~\ref{claim:valid_lower_bound}). As a result, we can verify the validity of different ILS cuts independently, without relying on dominance relationships or superadditivity assumptions. 

\subsection{Key definitions: recourse lower bounds and ILS cuts}
\label{subsection:activation_function}

In what follows, we generalize the expositions in~\citep{hoogendoorn2023improved, parada2024disaggregated, jabali2014}, and present ILS cuts for~\ref{problem:vrpr} using the concepts of activation functions and recourse lower bounds.

Given an activation function~$\W(x ; \mc{X}')$, consider the inequality~$\theta(U) \geq \tilde{\mc{L}} \cdot \W(x ; \mc{X}')$, where~$U \subseteq \Om$ and~$\tilde{\mc{L}} \in \Q_+$. The range of values of~$\tilde{\mc{L}}$ that yield a valid inequality for~$\F{\mc{Q}, \mc{X}, \Omega}$ can be easily determined.
\begin{claim}
\label{claim:valid_lower_bound}
Let~$\W(x ; \mc{X}')$ be an activation function,~$U \subseteq \Om$ and~$\tilde{\mc{L}} \in \Q_+$. Inequality~$\theta(U) \geq \tilde{\mc{L}} \cdot \W(x ; \mc{X}')$ is valid for~$\F{\mc{Q}, \mc{X}, \Omega}$ if and only if~$\sum_{\omega \in U} \sum_{R \in \mc{R}(\bar{x})} \mc{Q}(R, \omega) \geq \tilde{\mc{L}}$, for every~$\bar{x} \in \mc{X}'$.
\end{claim}
\begin{proof}
    By the definition of the set~$\F{\mc{Q}, \mc{X}, \Omega}$, we have that~$\theta(U) \geq \tilde{\mc{L}} \cdot \W(x ; \mc{X}')$ is a valid inequality for~$\F{\mc{Q}, \mc{X}, \Omega}$ if and only if
    \begin{equation}
        \label{ineq:proof_claim_lower_bound}
        \sum_{\omega \in U} \sum_{R \in \mc{R}(\bar{x})} \mc{Q}(R, \omega) \geq \tilde{\mc{L}} \cdot \W(\bar{x} ; \mc{X}'),
    \end{equation}
    for every~$\bar{x} \in \X \cap \Z^E$. By nonnegativity of~$\sum_{\omega \in U} \sum_{R \in \mc{R}(\bar{x})} \mc{Q}(R, \omega)$ and~$\tilde{\mc{L}}$, inequality~\eqref{ineq:proof_claim_lower_bound} always holds for~$\bar{x} \in (\X \cap \Z^E) \setminus \mc{X}'$. Otherwise,~$\bar{x} \in \mc{X}'$ and~$\W(\bar{x} ; \mc{X}') = 1$, meaning that inequality~\eqref{ineq:proof_claim_lower_bound} reduces to~$\sum_{\omega \in U} \sum_{R \in \mc{R}(\bar{x})} \mc{Q}(R, \omega) \geq \tilde{\mc{L}}$. 
\end{proof}

Claim~\ref{claim:valid_lower_bound} leads to the following definition.
\begin{definition}
    \label{def:recourse_lower_bound}
    Let~$U \subseteq \Om$ and~$\mc{X}' \subseteq \X \cap \Z^E$. We say that~$\mc{L}(U, \mc{X}') \in \Q_+$ is a \textit{recourse lower bound} with respect to~$U$ and~$\mc{X}'$ if~$\sum_{\omega \in U} \sum_{R \in \mc{R}(\bar{x})} \mc{Q}(R, \omega) \geq \mc{L}(U, \mc{X}')$, for all~$\bar{x} \in \mc{X}'$.
\end{definition}

\noindent
Note that recourse lower bounds depend on~$\X$ and on the disaggregation of~$\mc{Q}$ that we fixed earlier. For simplicity, whenever we say that~$\L(U, \mc{X}')$ is a recourse lower bound, it is implicit that it is with respect to~$U$ and~$\mc{X}'$.

Having established Definitions~\ref{def:activation_function} and~\ref{def:recourse_lower_bound}, we define \emph{ILS cuts} as follows.

\begin{definition}
    \label{def:ils_cut}
    An \emph{ILS cut} is an inequality of the form~$\theta(U) \geq \L(U, \mc{X}') \cdot \W(x ; \mc{X}')$, where~$U \subseteq \Om$,~$\mc{X}' \subseteq \X \cap \Z^E$,~$\L(U, \mc{X}')$ is a recourse lower bound and~$\W(x; \mc{X}')$ is an activation function. Moreover, a \emph{family of ILS cuts} is a set~$\mc{C}$ of tuples~$(U, \mc{X}', \mc{L}(U, \mc{X}'), \alpha, \beta)$, where~$\theta(U) \geq \L(U, \mc{X}') \cdot (\alpha^\T x + \beta)$ is an ILS cut, so~$\alpha^\T x + \beta$ is an activation function with respect to~$\mc{X}' \subseteq \X \cap \Z^E$. 
\end{definition}

\noindent
With each family of ILS cuts~$\mc{C}$, we associate the MILP formulation below.
\begin{equation}
\tag{ILS($\mc{X}, \mc{C}, \Omega$)}
\label{formulation:ILS}
\begin{aligned} 
\min ~~& c^\T x + \allones^\T \theta, & \nonumber \\
\text{s.t.~~} & \theta(U) \geq \mc{L}(U, \mc{X}') \cdot (\alpha^\T x + \beta), & ~~~~\forall (U, \mc{X}', \mc{L}(U, \mc{X}'), \alpha, \beta) \in \mc{C}, \\
& (x, \theta) \in (\X \cap \Z^E) \times \R^{\Om}_+.
\end{aligned}
\end{equation}

Next, we prove that the feasible region~$\mc{F}'$ of Formulation~\ref{formulation:ILS} is a relaxation of~$\F{\mc{Q}, \mc{X}, \Omega}$, and we characterize when the projection of~$\mc{F}'$ onto the~$(x, \rho)$-space is equivalent to the epigraph~$\epiQ{\mc{Q}, \X}$.

\begin{theorem}
    \label{theorem:ils}
    Let~$\mc{C}$ be a family of ILS cuts and let~$\mc{F}'$ be the feasible region of Formulation~\ref{formulation:ILS}. Then the following holds:
    \begin{enumerate}[(i)]
        \item $\F{\mc{Q}, \mc{X}, \Omega} \subseteq \mc{F}'$; and \label{item:ils_relaxation}
        \item $\epiQ{\mc{Q}, \X} = \projrho (\mc{F}')$ if and only if~$\allones^\T \bar{\theta} \geq \mc{Q}(\bar{x})$, for all~$(\bar{x}, \bar{\theta}) \in \mc{F}'$. \label{item:ils_projection}
    \end{enumerate}
\end{theorem}
\begin{proof}
    Item~\ref{item:ils_relaxation} follows from Claim~\ref{claim:valid_lower_bound} and the definitions
    of ILS cuts. To prove item~\ref{item:ils_projection}, we write
    \begin{equation*}
        \epiQ{\mc{Q}, \X} \overset{\text{Claim~\ref{claim:rewrite_epigraph}}}{=} \projrho (\F{\mc{Q}, \mc{X}, \Omega}) \overset{\text{Item~\ref{item:ils_relaxation}}}{\subseteq} \projrho (\mc{F}').
    \end{equation*}
    Hence, for any~$(\bar{x}, \bar{\theta}) \in \mc{F}'$, we have that~$(\bar{x}, \allones^\T \bar{\theta}) \in \epiQ{\mc{Q}, \X}$ if and only if~$\allones^\T \bar{\theta} \geq \mc{Q}(\bar{x})$, as desired.
\end{proof}

Observe that since problem~\ref{problem:vrpr} can be written as a minimization problem over~$\epiQ{\mc{Q}, \X}$, there is no need to guarantee that a candidate solution~$(\bar{x}, \bar{\theta}) \in (\X \cap \Z^E) \times \R^{V_+}_+$ belongs to~$\F{\mc{Q}, \mc{X}, \Omega}$. Instead, it suffices to check whether the projection of~$(\bar{x}, \bar{\theta})$ under~$\projrho(\,\cdot\,)$ lies inside~$\epiQ{\mc{Q}, \X}$, which can be easily verified using item~\ref{item:ils_projection} of Theorem~\ref{theorem:ils}. 

We also mention in passing that
in the original ILS paper,~\cite{LAPORTE1993133} named as \emph{valid optimality cuts} a family of inequalities in the~$(x, \rho)$-space that implies~$\rho \geq \mc{Q}(x)$, for all~$x \in \X \cap \Z^E$ (adapting the notation to our context). This is similar to our statement in item~\ref{item:ils_projection}, except that here the family of ILS cuts~$\mc{C}$ is defined in the extended space of the~$(x, \theta)$-variables.

\subsection{Applying the framework}
\label{subsection:applying}

The first step in applying our framework is to choose~$\Om$ and a corresponding recourse disaggregation. Recall that~$\Om(R)$ denotes the support of the disaggregation of~$\mc{Q}$ with respect to route~$R$. Based on the approaches reviewed in Section~\ref{section:review}, a natural choice here is as follows (a formal justification is later provided in Theorem~\ref{thm:route_formulation}).
\begin{assumption}
    \label{assumption:customer_disaggregation}
    The disaggregation of~$\mc{Q}$ is along~$\Om = V_+$ and, for every route~$R$,~$\Om(R) \subseteq V_+(R)$.
\end{assumption}

Let~$\mc{C}$ be a family of inequalities of the form~$\theta(U) \geq \L(U, \mc{X}') \cdot \W(x ; \mc{X}')$ (with~$\L(U, \mc{X}') \in \Q_+$ and~$\W(x ; \mc{X}')$ affine in~$x$). To express problem~\ref{problem:vrpr} as an MILP in the form of Formulation~\ref{formulation:ILS}, we need to show that:
\begin{enumerate}[label=(\arabic*)]
    \item every inequality~$\theta(U) \geq \L(U, \mc{X}') \cdot \W(x ; \mc{X}')$ in the family~$\mc{C}$ is an ILS cut, that is, \label{item:framework_a}
    \begin{enumerate}[label=(\arabic{enumi}\alph*)]
        \item $\W(x ; \mc{X}')$ is an activation function (Definition~\ref{def:activation_function}); and \label{item:framework_a1}
        \item $\L(U, \mc{X}')$ is a recourse lower bound (Definition~\ref{def:recourse_lower_bound});\label{item:framework_a2}
    \end{enumerate}
    \item every~$(\bar{x}, \bar{\theta})$ feasible for Formulation~\ref{formulation:ILS} satisfies~$\allones^\T \bar{\theta} \geq \mc{Q}(\bar{x})$. \label{item:framework_b}
\end{enumerate}

In Section~\ref{section:generalizing}, we revisit the partial-route-based Formulation~\eqref{formulation:partial_route} and formally establish its validity for a general recourse function~$\mc{Q}$. We also generalize the partial route inequalities using the notion of recourse disaggregation introduced here. The resulting approach relies on trivially valid recourse lower bounds (such as~$\mc{Q}(R)$ or~$\mc{Q}(\bar{x})$), which ensure the above item~\ref{item:framework_b}. Although there may be exponentially many such inequalities, they can be separated efficiently for integer vectors~$\bar{x} \in \X \cap \Z^E$, yielding a branch-and-cut algorithm without additional assumptions on~$\mc{Q}$ or~$\X$.

However, the described algorithm may converge slowly. To improve performance, in Section~\ref{section:generalizing} we also generalize the set cuts from the DL-shaped method using the concept of recourse disaggregation. Section~\ref{section:application_vrpsd} then derives \mbox{problem-specific} recourse lower bounds (item~\ref{item:framework_a2}) for the partial route and set cuts introduced in Section~\ref{section:generalizing}.

\section{Characterizing disaggregations and generalizing existing ILS cuts}
\label{section:generalizing}

In this section, we show that our framework unifies and generalizes several existing VRPSD formulations~\citep{gendreau95, laporte2002, hoogendoorn2023improved, parada2024disaggregated, legault2025superadditivity}. Specifically, Section~\ref{subsection:gendreau} shows how the well-known nondisaggregated formulation of~\cite{gendreau95} fit within our framework. Section~\ref{subsection:route_partial_route} revisits the partial route inequalities of~\cite{hoogendoorn2023improved}, and proves that a formulation based on the so-called route cuts yields a valid reformulation of problem~\ref{problem:vrpr} if and only if the disaggregation of~$\mc{Q}$ is \emph{route-disjoint} (Theorem~\ref{thm:route_formulation}).

Lastly, Section~\ref{subsection:parada_cuts} builds on our framework to generalize and extend the DL-shaped method~\cite{parada2024disaggregated, legault2025superadditivity} in two important ways. First, we show that the path cuts are valid if and only if the recourse function~$\mc{Q}$ is \emph{restrictively superadditive} (Theorem~\ref{thm:supperadditive}). Second, we leverage the definition of recourse lower bounds (Definition~\ref{def:recourse_lower_bound}) to generalize the set cuts of the DL-shaped method, allowing these inequalities to be applied even when the recourse function~$\mc{Q}$ is not restrictively superadditive.

\subsection{The ILS cuts of Gendreau et al. (1995)}
\label{subsection:gendreau}

As a first application of Theorem~\ref{theorem:ils}, we consider the classic nondisaggregated formulation of~\cite{gendreau95}. In their setup, feasible routing plans use exactly~$k \in \Z_{++}$ vehicles, that is,~$\X \subseteq \{x \in \R^E : x(\delta(0)) = 2k\}$. Let~$\bar{x}$ be an arbitrary point in~$\X \cap \Z^E$ and recall that~$G(\bar{x})$ is the support graph associated with~$\bar{x}$. Their activation function is defined as~$W^k_G(x ; \{\bar{x}\}) \coloneqq 1 + x(E(G(\bar{x})) \setminus \delta(0)) - |V_+| + k$\label{sym:Wg} and the corresponding ILS cut is
\begin{equation}
    \label{ineq:gendreau_cut}
    \allones^\T \theta \geq \mc{Q}(\bar{x}) \cdot W^k_G(x ; \{\bar{x}\}).
\end{equation} 

It is easy to see that~$\Wg(x ; \{\bar{x}\})$ is indeed an activation function. The next result is an immediate consequence of Theorem~\ref{theorem:ils} and summarizes their formulation. (To recover the exact same formulation as theirs, one applies Corollary~\ref{corollary:gendreau} with~$\Om$ being a singleton.)

\begin{corollary}
    \label{corollary:gendreau}
    Let~$k \in \Z_{++}$ and suppose that~$\X \subseteq \{x \in \R^E : x(\delta(0)) = 2k\}$. Let~$$\mc{C}^k_G = \{(\Om, \{\bar{x}\}, \mc{Q}(\bar{x}), \alpha^{\bar{x}}, \beta^{\bar{x}})\}_{\bar{x} \in \mc{X} \cap \Z^E}$$ be a family of ILS cuts, where, for every~$\bar{x} \in \X \cap \Z^E$,~$\Wg(x ; \{\bar{x}\}) = (\alpha^{\bar{x}})^\T x + \beta^{\bar{x}}$. Let~$\mc{F}^k_G$ be the feasible region of formulation~\hyperref[formulation:ILS]{ILS($\mc{X}, \mc{C}^k_G, \Omega$)}. Then~$\epiQ{\mc{Q}, \X} = \projrho(\mc{F}^k_G)$.~\hfill\Halmos
\end{corollary}

Several ideas have been proposed to improve the basic ILS approach of~\cite{gendreau95}. One such idea, first explored by~\cite{laporte2002}, derives \emph{globally valid lower bounds} that enable a strengthening the coefficients in the ILS cuts (see Appendix~\ref{appendix:laporte} for how this fits into our framework). As discussed in Section~\ref{section:review}, a more promising direction is to instead consider approaches that disaggregate the recourse variable~$\rho$ (i.e.,~$|\Om| > 1$). In this sense, the remainder of this section generalizes the ILS cuts presented in Section~\ref{section:review} using the concepts introduced by our framework.

\subsection{Route and partial route cuts}
\label{subsection:route_partial_route}

Let~$R$ be a route. Recall that~$\Om(R)$ is the support of route~$R$ (with respect to the fixed disaggregation of~$\mc{Q}$) and that~$\mc{X}_{=}(R)$ is the set of solutions containing route~$R$. We define an \emph{ILS route cut} (or simply \emph{route cut}) as an inequality of the form
\begin{equation}
    \label{ineq:route_cut}
    \theta(\Om(R)) \geq \mc{Q}(R) \cdot \W(x ; \hyperlink{set:x_r}{\mc{X}_{=}(R)}),
\end{equation}
where~$\W(x ; \hyperlink{set:x_r}{\mc{X}_{=}(R)})$ is an activation function. 

For example, suppose that we use the activation function~$\Whs$ of~\cite{hoogendoorn2023improved} and set~$\Om = V_+$. If the fixed disaggregation of~$\mc{Q}$ is such that, for each route~$R$, there exists~$v_R \in V_+(R)$ with~$\mc{Q}(R) = \mc{Q}(R, v_R)$, then inequality~\eqref{ineq:route_cut} reduces to their route-split inequality~$\theta_{v_R} \geq \mc{Q}(R) \cdot \Whs(x ; \Xh{R})$.

One may easily verify that the route cuts~\eqref{ineq:route_cut} are indeed ILS cuts (Definition~\ref{def:ils_cut}). That is, for every~$\bar{x} \in \hyperlink{set:x_r}{\mc{X}_{=}(R)}$, we have that~$\sum_{R' \in \mc{R}(\bar{x})} \sum_{\omega \in \Om(R)} \mc{Q}(R', \omega) \geq \sum_{\omega \in \Om(R)} \mc{Q}(R, \omega) = \mc{Q}(R)$, which implies that~$\mc{Q}(R)$ is a recourse lower bound with respect to~$\Om(R)$ and~$\hyperlink{set:x_r}{\mc{X}_{=}(R)}$. Therefore, route cuts yield a relaxation of the feasible region~$\F{\mc{Q}, \X, \Om}$ (item~\ref{item:ils_relaxation} of Theorem~\ref{theorem:ils}).

It turns out that the following property exactly determines whether route cuts satisfy item~\ref{item:ils_projection} of Theorem~\ref{theorem:ils}, meaning that they can be used to reformulate the feasible region of~\ref{problem:vrpr}. This result justifies the disaggregation in Assumption~\ref{assumption:customer_disaggregation}. Moreover, as observed in Section~\ref{subsection:review_hoogendoorn}, this also formally establishes the correctness of Formulation~\eqref{formulation:partial_route} for any choice of~$\mc{Q}$ and~$\X$.

\begin{definition}
    \label{def:route_disjoint}
    The disaggregation of~$\mc{Q}$ is \emph{route-disjoint} if, for every~$\bar{x} \in \X \cap \Z^E$ and distinct routes~$R, R' \in \mc{R}(\bar{x})$, it holds that~$\Omega(R) \cap \Omega(R') = \emptyset$.
\end{definition}

\begin{restatable}{theorem}{thmrouteformulation}
    \label{thm:route_formulation}
    Consider the family of ILS cuts
    $$\mc{C}_{R} = \{(\Om(R), \hyperlink{set:x_r}{\mc{X}_{=}(R)}, \mc{Q}(R), \alpha^R, \beta^R) : \text{$R$ is a route}\},$$
    where, for every route~$R$,~$\W(x ; \hyperlink{set:x_r}{\mc{X}_{=}(R)}) = (\alpha^{R})^\T x + \beta^{R}$ is an activation function. Let~$\mc{F}_R$ be the feasible region of~\hyperref[formulation:ILS]{ILS($\mc{X}, \mc{C}_{R}, \Omega$)}. Then~$\epiQ{\mc{Q}, \mc{X}} = \projrho(\mc{F}_R)$ if and only if the disaggregation of~$\mc{Q}$ is route-disjoint.
\end{restatable}
\begin{proof}
Suppose first that the disaggregation of~$\mc{Q}$ is route-disjoint and let~$(\bar{x}, \bar{\theta}) \in \mc{F}_R$. Consider an ILS route cut~$\theta(\Om(R)) \geq \mc{Q}(R) \cdot \W(x ; \hyperlink{set:x_r}{\mc{X}_{=}(R)})$ associated with a route~$R \in \mc{R}(\bar{x})$. Since~$\W(\bar{x} ; \hyperlink{set:x_r}{\mc{X}_{=}(R)}) = 1$, we know that~$\bar{\theta}(\Om(R)) \geq \mc{Q}(R)$. Hence, by route-disjointness,~$\allones^\T \bar{\theta} \geq \sum_{R \in \mc{R}(\bar{x})} \bar{\theta}(\Om(R)) \geq \sum_{R \in \mc{R}(\bar{x})} \mc{Q}(R) = \mc{Q}(\bar{x})$. Item~\ref{item:ils_projection} of Theorem~\ref{theorem:ils} then gives~$\epiQ{\mc{Q}, \mc{X}} = \projrho(\mc{F}_R)$.

Assume now that the disaggregation of~$\mc{Q}$ is not route-disjoint. In this case, there exists~$\hat{x} \in \X \cap \Z^E$ such that, for some~$\omega' \in \Om$, we have~$|\{R \in \mc{R}(\hat{x}) : \omega' \in \Om(R)\}| \geq 2$. For each~$\omega \in \Om$, set~$\hat{\theta}_\omega = \max_{R \in \mc{R}(\hat{x})}\{ \mc{Q}(R, \omega)\}$. Since~$\hat{\theta}_\omega \leq \sum_{R \in \mc{R}(\hat{x})} \mc{Q}(R, \omega)$, for all~$\omega \in \Om$, and~$\hat{\theta}_{\omega'} < \sum_{R \in \mc{R}(\hat{x})} \mc{Q}(R, \omega')$, we know that~$\allones^\T \hat{\theta} < \mc{Q}(\hat{x})$. Therefore, by item~\ref{item:ils_projection} of Theorem~\ref{theorem:ils}, it suffices to show that~$(\hat{x}, \hat{\theta})$ belongs to~$\mc{F}_R$. To do this, take an arbitrary route~$R'$. If~$R' \in \mc{R}(\hat{x})$, then~$(\alpha^{R'})^\T \hat{x} + \beta^{R'} = 1$ and~$\mc{Q}(R') \cdot ((\alpha^{R'})^\T \hat{x} + \beta^{R'}) = \sum_{\omega \in \Omega} \mc{Q}(R', \omega) \leq \sum_{\omega \in \Omega} \hat{\theta}_\omega$, where the last inequality follows from how we set~$\hat{\theta}$. Otherwise,~$(\alpha^{R'})^\T \hat{x} + \beta^{R'} \leq 0$ and we are done.
\end{proof}

To strengthen the relaxations that arise when solving Formulation~\hyperref[formulation:ILS]{ILS($\mc{X}, \mc{C}_{R}, \Omega$)}, and following the development of~\cite{hoogendoorn2023improved}, we also generalize the partial route-split inequalities reviewed in Section~\ref{subsection:review_hoogendoorn}. Given a partial route~$\H$, a set~$U \subseteq \Om$, a recourse lower bound~$\L(U, \Xh{H})$ and an activation function~$\W(x ; \Xh{H})$, we define a \emph{partial route cut} as an inequality of the form
\begin{equation}
    \label{ineq:partial_route_cut}
    \theta(U) \geq \L(U, \Xh{H}) \cdot \W(x ; \Xh{H}).
\end{equation}
\noindent
It is clear that~\eqref{ineq:partial_route_cut} generalizes the partial route-split inequalities~\eqref{partial_route:cuts} (and the route cuts~\eqref{ineq:route_cut}). 

In this work, we shall focus on the following special case of~\eqref{ineq:partial_route_cut}. Suppose that~$\Om = V_+$,~$U = V_+(\H)$ and~$\mc{L}(\H)$ is lower bound on~$\mc{Q}(R)$, for every route~$R$ that adheres to~$\H$. Then inequality~\eqref{ineq:partial_route_cut} reduces to~$\theta(V_+(\H)) \geq \mc{L}(\H) \cdot \W(x ; \Xh{H})$.

\subsection{Inequalities from the DL-shaped method}
\label{subsection:parada_cuts}

Recall from Section~\ref{subsection:review_dl-shaped} that the DL-shaped method is a branch-and-cut algorithm that is fundamentally based on two new classes of ILS cuts: the \emph{path cuts} and the \emph{set cuts}. As observed in Section~\ref{section:intro},~\cite{legault2025superadditivity} have recently shown that, when an unlimited number of vehicles is available, the \mbox{DL-shaped method} is valid if and only if the recourse function~$\mc{Q}$ is superadditive. 

Building on our framework, we obtain in Section~\ref{subsection:path_cuts} a weaker variant of superadditivity, named \emph{restricted superadditivity}, that characterizes the validity of path cuts without making any additional assumptions on the set of feasible routing plans (see Contribution~\ref{cont:rsuper}). Section~\ref{subsection:correctness_path_cuts} then proves that, when~$\mc{Q}$ is restrictively superadditive, the path cuts are sufficient to reformulate problem~\ref{problem:vrpr}. Rather than proving this directly (as in~\citep{parada2024disaggregated, legault2025superadditivity}), we obtain this result by showing a simple dominance relationship between the linear programming (LP) relaxations associated with the path cuts and the inequalities of~\cite{gendreau95}. Lastly, Section~\ref{subsection:monotonicity_lower_bounds} highlights that if set cuts are expressed using the concept of recourse lower bounds introduced in Definition~\ref{def:recourse_lower_bound} (rather than using the lower bound~$\mc{L}_{DL}(S)$ discussed in Section~\ref{subsection:review_dl-shaped}), then they are always valid for~$\F{\mc{Q}, \mc{X}, \Omega}$, regardless if~$\mc{Q}$ is restrictively superadditive. Thus, while path cuts depend on specific properties of~$\mc{Q}$, the set cuts can be applied much more broadly. 

We assume in the remainder that the fixed disaggregation of~$\mc{Q}$ satisfies Assumption~\ref{assumption:customer_disaggregation}.

\paragraph{\textbf{Realizable routes.}}

Throughout our discussion, we repeatedly refer to the set of routes that can appear in a feasible routing plan:
\begin{definition}
    \label{def:realizable_routes}
    We say that a route~$R$ is \emph{realizable} if there exists a solution~$\bar{x} \in \X \cap \Z^E$ such that~$R \in \mc{R}(\bar{x})$. The set of all realizable routes is defined as~$\Pi \coloneqq \bigcup_{x \in \X \cap \Z^E} \mathcal{R}(x)$.
\end{definition}

We emphasize that Definition~\ref{def:realizable_routes} differs from the use of the term ``feasible route'' in the vehicle routing literature, where it commonly refers to routes satisfying given \emph{intra-route} constraints. In contrast to Definition~\ref{def:realizable_routes}, the notion of feasible routes typically ignores the \emph{inter-route} constraints, i.e., constraints involving multiple vehicles. For example,~\cite{parada2024disaggregated} and~\cite{legault2025superadditivity} define a feasible route (or path) as a sequence of customers whose total expected demand does not exceed the vehicle capacity. In this case, if the number of available vehicles is limited, short routes may satisfy all intra-route constraints but still never appear in any feasible solution, meaning that they are not realizable. On the other hand, the assumption of an unlimited number of vehicles in~\cite{legault2025superadditivity} ensures that their feasible routes are always realizable.

\subsubsection{Monotonicity, superadditivity and the path cuts}
\label{subsection:path_cuts}

Recall from Section~\ref{subsection:review_dl-shaped} that an \emph{ILS path cut} (or simply \emph{path cut}) associated with a route~$R$ is an inequality of the form
\begin{equation}
    \label{ineq:path_cut}
    \theta(V_+(R)) \geq \mc{Q}(R) \cdot \Wdl{x; \supsetXh{R}},
\end{equation}
where~$\Wdl{x; \supsetXh{R}} = 1 + \sum_{e \in E(R) \setminus \delta(0)} (x_e - 1)$. 

Applying Assumption~\ref{assumption:customer_disaggregation} and Claim~\ref{claim:valid_lower_bound} to inequality~\eqref{ineq:path_cut} we obtain the following:

\begin{definition}
    \label{def:monotonicity}
    We say that a disaggregation~$\hat{\mc{Q}}$ of~$\mc{Q}$ is \emph{monotone} if it satisfies the conditions in Assumption~\ref{assumption:customer_disaggregation} and, for every~$R \in \routes$ and~$R' \subseteq R$, we have that~$\sum_{v \in V_+(R')} \hat{\mc{Q}}(R, v) \geq \mc{Q}(R')$.
\end{definition}

\begin{fact}
    \label{fact:equivalence_monotonicity}
    Suppose Assumption~\ref{assumption:customer_disaggregation} holds. Then the fixed disaggregation of~$\mc{Q}$ is monotone if and only if, for every~$R' \subseteq R$, with~$R \in \routes$, the value~$\mc{Q}(R')$ is a recourse lower bound with respect to~$V_+(R')$ and~$\supsetXh{R'}$. 
\end{fact}

Definition~\ref{def:monotonicity} is similar to the inequality used in Proposition~2 of~\cite{parada2024disaggregated}, with the following important differences: (1) it explicitly accounts for the disaggregated recourse terms~$\hat{\mc{Q}}(R, v)$; (2) it is restricted to realizable routes (i.e.,~$R \in \routes$); and (3) it only considers subroutes of~$R$, rather than subsequences. (Note that~\cite{parada2024disaggregated} also define a monotonicity property, but their property specifically concerns the probability distributions associated with the VRPSD. Their Proposition~2 is then presented as a consequence of this property.) However, as mentioned in Section~\ref{section:intro},~\cite{legault2025superadditivity} have recently shown that Proposition~2 of~\cite{parada2024disaggregated} does not guarantee the validity of ILS path cuts. On the other hand, Fact~\ref{fact:equivalence_monotonicity} shows that Definition~\ref{def:monotonicity} is both necessary and sufficient for the validity of ILS path cuts with respect to~$\F{\mc{Q}, \mc{X}, \Omega}$. In fact, as we show next, Definition~\ref{def:monotonicity} is also equivalent to a weaker variant of the \emph{superadditivity property} proposed by~\cite{legault2025superadditivity} (see Section~\ref{subsection:limitation_superadditive}).

Recall that~$R_1 \oplus R_2$ denotes the route obtained by concatenating routes~$R_1$ and~$R_2$. By slightly generalizing the definition introduced by~\cite{legault2025superadditivity} to our more general setting (where not every ``feasible route'' is realizable), we define their superadditivity property as follows.
\begin{definition}
    \label{def:superadditivity}
    The recourse function~$\mc{Q}$ is \emph{superadditive} if, for every~$R \in \routes$ and~$R_1 \oplus R_2 \subseteq R$, we have that~$\mc{Q}(R_1 \oplus R_2) \geq \mc{Q}(R_1) + \mc{Q}(R_2)$.
\end{definition}

Building on the construction of~\cite{legault2025superadditivity}, we can show that whenever~$\mc{Q}$ is superadditive, a monotone disaggregation exists. Due to space constraints, the proof is left to Appendix~\ref{appendix:proof_claim_monotone}.

\begin{restatable}{claim}{claimmonotone}
    \label{claim:monotone}
    Suppose that~$\mc{Q}$ is superadditive and let~$\hat{\mc{Q}}$ be a disaggregation of~$\mc{Q}$ along~$\Om = V_+$ such that, for every route~$R = (v_1, \ldots, v_\ell) \in \routes$,
    \begin{equation*}
        \hat{\mc{Q}}(R, v_i) = 
        \begin{cases}
            \mc{Q}((v_1, \ldots, v_i)) - \mc{Q}((v_1, \ldots, v_{i - 1})), & \text{if~$i \in [\ell] \setminus \{1\}$,} \\
            \mc{Q}((v_1)), & \text{otherwise.}
        \end{cases}
    \end{equation*}
    Then~$\hat{\mc{Q}}$ is monotone.
\end{restatable}

On the other hand, the following example shows that the converse may not hold: 
\begin{example}
\label{example:superadditive}
Suppose that~$\mc{X} \cap \Z^E$ is a singleton~$\{\bar{x}\}$ with~$\mc{R}(\bar{x}) = \{(v_1, v_2, v_3), (v_4)\}$, so~$\routes = \mc{R}(\bar{x})$. Let~$R = (v_1, v_2, v_3)$ and assume that the recourse function~$\mc{Q}$ is such that~$\mc{Q}(R) = 3$,~$\mc{Q}((v_1)) = \mc{Q}((v_2)) = \mc{Q}((v_3)) = \mc{Q}((v_1, v_2)) = 1$ and~$\mc{Q}((v_2, v_3)) = \mc{Q}((v_4)) = 0$. Let~$\hat{\mc{Q}}$ be a disaggregation of~$\mc{Q}$ such that~$\hat{\mc{Q}}(R, v_1) = \hat{\mc{Q}}(R, v_2) = \hat{\mc{Q}}(R, v_3) = 1$. Note that for every subroute~$R' \subseteq R$, we have that~$\sum_{v \in V_+(R')} \hat{\mc{Q}}(R, v) \geq \mc{Q}(R')$, so~$\hat{\mc{Q}}$ is monotone. However, since~$0 = \mc{Q}((v_2, v_3)) < \mc{Q}((v_2)) + \mc{Q}((v_3)) = 2$, the recourse function~$\mc{Q}$ is not superadditive.~\hfill\Halmos
\end{example}

The situation depicted in Example~\ref{example:superadditive} does not arise in the work of~\cite{legault2025superadditivity}, as their proof of the necessity of superadditivity relies on the assumption that~$\routes$ is \emph{downward closed}, meaning that if~$R \in \routes$, then every subroute~$R' \subseteq R$ also belongs to~$\routes$. Under this assumption, Example~\ref{example:superadditive} cannot occur, since by downward closedness,~$(v_2, v_3) \subseteq R$ belongs to~$\routes$, meaning that if~$\hat{\mc{Q}}$ is a monotone disaggregation of~$\mc{Q}$, we must have~$\mc{Q}((v_2, v_3)) = \hat{\mc{Q}}((v_2, v_3), v_2) + \hat{\mc{Q}}((v_2, v_3), v_3)
\geq \mc{Q}((v_2)) + \mc{Q}((v_3))$. However, most of the instances considered in the VRPSD literature use~$\mc{X} = \Xcvrp$, in which case the number of vehicles is fixed to a constant~$k \in \Z_{++}$ and~$\routes$ may not be downward closed.

Since the monotonicity of a disaggregation of~$\mc{Q}$ is equivalent to the validity of path cuts (Fact~\ref{fact:equivalence_monotonicity}), Example~\ref{example:superadditive} shows that path cuts may be applied even if~$\mc{Q}$ is not superadditive. This naturally raises the question of whether one can characterize the conditions on the recourse function that ensure the validity of path cuts, even when~$\routes$ is not downward closed. To address this question, we introduce the following property.

\begin{definition}
    \label{def:weak_superadditivity}
    The recourse function~$\mc{Q}$ is \emph{restrictively superadditive} if, for every~$R \in \routes$, we have that~$\mc{Q}(R) \geq \mc{Q}(R_1) + \ldots + \mc{Q}(R_t)$,
    for all disjoint subroutes~$R_1, \ldots, R_t \subseteq R$, that is,~$V_+(R_i) \cap V_+(R_j) = \emptyset$, for every~$i \in [t]$ and~$j \in [i - 1]$. (Recall that~$[0] = \emptyset$.)
\end{definition}

We emphasize that restricted superadditivity only requires~$\mc{Q}(R) \geq \mc{Q}(R_1) + \ldots + \mc{Q}(R_t)$ for realizable routes~$R \in \routes$. For instance, in Example~\ref{example:superadditive}, for any choice of disjoint routes~$R_1, \ldots, R_t \subseteq R$, we have that~$\mc{Q}(R_1) + \ldots + \mc{Q}(R_t) \leq 3$, meaning that~$\mc{Q}$ is restrictively superadditive but not superadditive (as~$\mc{Q}((v_2, v_3)) < \mc{Q}((v_2)) + \mc{Q}((v_3))$). On the other hand, we can verify that superadditivity implies restricted superadditivity (see Appendix~\ref{appendix:proof_claim_restrictively_superadditve}).

\begin{restatable}{claim}{claimrestrictivelysuperadditive}
\label{claim:restrictively_superadditive}
Suppose that~$\mc{Q}$ is superadditive, then~$\mc{Q}$ is also restrictively superadditive.
\end{restatable}

The merit of Definition~\ref{def:weak_superadditivity} is shown by the following result: restricted superadditivity precisely characterizes when a monotone disaggregation of~$\mc{Q}$ exists. We defer the proof to Appendix~\ref{appendix:proof_thm_superadditive}.

\begin{restatable}{theorem}{thmsuperadditive}
    \label{thm:supperadditive}
    The following holds:
    \begin{enumerate}[(a)]
        \item if~$\hat{\mc{Q}}$ is a monotone disaggregation of~$\mc{Q}$, then~$\mc{Q}$ is restrictively superadditive; \label{item:superadditive_a}
        \item conversely, if~$\mc{Q}$ is restrictively superadditive, then there exists a monotone disaggregation of~$\mc{Q}$. \label{item:superadditive_b}
    \end{enumerate}
\end{restatable}

As an application of Theorem~\ref{thm:supperadditive}, consider problem~$\vrpr(\Xcvrp, \Qc)$. Since the set of realizable routes~$\routes$ is not necessarily downward closed, we cannot use the result of~\cite{legault2025superadditivity} to establish that path cuts are not valid for~$\F{\mc{X}, \Qc, V_+}$. Instead, Example~\ref{example:qc_not_superadditive} shows that~$\Qc$ is not restrictively superadditive, which implies that, regardless of the chosen disaggregation of~$\Qc$, path cuts cannot be safely applied.

\begin{example}
    \label{example:qc_not_superadditive}
    Consider the instance illustrated in Figure~\ref{figure:example_set_cut} and assume that the disaggregation of~$\Qc$ is according to Remark~\ref{remark:disaggregation}. Let~$R = (v_1, v_2, v_3, v_4)$ and~$R' = (v_2, v_3, v_4)$. Note that both routes are realizable, as there exists~$\bar{x}, x' \in \Xcvrp \cap \Z^E$ such that~$\mc{R}(\bar{x}) = \{R, (v_5)\}$ and~$\mc{R}(x') = \{R', (v_5, v_1)\}$. Since~$\Qc(R) = \Qc(\oa{R})$ and~$\Qc(R') = \Qc(\oa{R}')$, one can verify that~$\Qc(R) = \Qc(R, v_3) + \Qc(R, v_4) = (2 \cdot c_{0 v_3} + 2 \cdot c_{0 v_4}) / 4 = 3 / 2$ and~$\Qc(R') = \Qc(R', v_4) = (2 \cdot 2 \cdot c_{0 v_4}) / 4 = 2$. Therefore,~$3 / 2 = \Qc(R) = \sum_{v \in V_+(R')} \Qc(R, v) < \Qc(R') = 2$. This shows not only that the disaggregation of~$\Qc$ in Remark~\ref{remark:disaggregation} is not monotone (Definition~\ref{def:monotonicity}), but also that~$\Qc$ is not restrictively superadditive (Definition~\ref{def:weak_superadditivity}). Hence, by Theorem~\ref{thm:supperadditive}, there is no disaggregation of~$\Qc$ that is monotone.~\hfill\Halmos
\end{example}

To further clarify the connection between Theorem~\ref{thm:supperadditive} and the result of~\cite{legault2025superadditivity}, we close this discussion by showing that restricted superadditivity collapses to superadditivity if~$\routes$ is downward closed.
\begin{claim}
    Suppose that~$\mc{Q}$ is restrictively superadditive. If~$\routes$ is downward closed, then~$\mc{Q}$ is also superadditive. Furthermore, if~$\routes$ is not downward closed, then there exists an instance of problem~\ref{problem:vrpr} where~$\mc{Q}$ is not superadditive.
\end{claim}
\begin{proof}
    Let~$R \in \routes$ and~$R_1 \oplus R_2 \subseteq R$. If~$\routes$ is downward closed,~$R_1 \oplus R_2 \in \routes$, so by restricted superadditivity,~$\mc{Q}(R_1 \oplus R_2) \geq \mc{Q}(R_1) + \mc{Q}(R_2)$. If~$\routes$ is not downward closed, consider Example~\ref{example:superadditive}.
\end{proof}

\subsubsection{Correctness of ILS path cuts formulations}
\label{subsection:correctness_path_cuts}

We now turn to proving the correctness of a VRPR formulation based on ILS path cuts. We start by showing that fractional solutions that satisfy the ILS path cuts also satisfy the inequalities of~\cite{gendreau95}.

\begin{proposition}
    \label{proposition:dominance_parada_gendreau}
    Suppose that~$\mc{Q}$ satisfies Assumption~\ref{assumption:customer_disaggregation}. Let~$\bar{x} \in \X \cap \Z^E$ with~$k = |\mc{R}(\bar{x})|$. Let~$(x', \theta') \in \mc{X} \times \R^{V_+}_+$ satisfy the ILS path cuts~\eqref{ineq:path_cut} with respect to the routes in~$\mc{R}(\bar{x})$. Then~$\mathbbm{1}^\T \theta' \geq \mc{Q}(\bar{x}) \cdot \Wg(x' ; \{\bar{x}\})$.
\end{proposition}
\proof
    Since the fixed disaggregation of~$\mc{Q}$ is route-disjoint and, for every~$R \in \mc{R}(\bar{x})$, we have~$\Wdl{x; \supsetXh{R}} = 1 + \sum_{e \in E(R) \setminus \delta(0)} (x'_e - 1)$ and~$\sum_{e \in E(R) \setminus \delta(0)} (x'_e - 1) \leq 0$, it follows that
    \begin{align*}
        \allones^\T \theta' & \geq \sum_{R \in \mc{R}(\bar{x})} \theta'(V_+(R))  \geq \sum_{R \in \mc{R}(\bar{x})} \left( \mc{Q}(R) \cdot \Wdl{x; \supsetXh{R}} \right) \\
        & = \sum_{R \in \mc{R}(\bar{x})} \mc{Q}(R) \cdot \left(1 + \sum_{e \in E(R) \setminus \delta(0)} (x'_e - 1) \right) = \mc{Q}(\bar{x}) + \sum_{R \in \mc{R}(\bar{x})} \mc{Q}(R) \cdot \left(\sum_{e \in E(R) \setminus \delta(0)} (x'_e - 1) \right) \\
        & \geq \mc{Q}(\bar{x}) \cdot \left(1 + \sum_{R \in \mc{R}(\bar{x})} \sum_{e \in E(R) \setminus \delta(0)} (x'_e - 1) \right)  = \mc{Q}(\bar{x}) \cdot \Wg(x' ; \{\bar{x}\}). & \tag*{\Halmos}
    \end{align*}

\noindent
It follows from Proposition~\ref{proposition:dominance_parada_gendreau} that the ILS path cuts satisfy item~\ref{item:ils_projection} of Theorem~\ref{theorem:ils}, which gives the desired correctness result:

\begin{corollary}
    \label{corollary:parada_correctness}
    Suppose that the fixed disaggregation of~$\mc{Q}$ satisfies Assumption~\ref{assumption:customer_disaggregation} and is monotone. Consider the family of ILS cuts~$$\mc{C}_P = \{(V_+(R), \mc{Q}(R), \alpha^R, \beta^R) : R \in \routes\},$$
    where~$\Wdl{x; \supsetXh{R}} = (\alpha^R)^\T x + \beta^R$, for every route~$R \in \routes$. Let~$\mc{F}_P$ be the feasible region of~\hyperref[formulation:ILS]{ILS($\mc{X}, \mc{C}_{P}, \Omega$)}. Then~$\epiQ{\mc{Q}, \mc{X}} = \projrho (\mc{F}_P)$.
\end{corollary}
\begin{proof}
    By Fact~\ref{fact:equivalence_monotonicity},~$\mc{F}_P \supseteq \F{\mc{Q}, \mc{X}, V_+}$, so item~\ref{item:ils_relaxation} of Theorem~\ref{theorem:ils} holds. To show item~\ref{item:ils_projection}, we use Proposition~\ref{proposition:dominance_parada_gendreau} to learn that any~$(\bar{x}, \bar{\theta}) \in \mc{F}_P$ satisfies~$\allones^\T \bar{\theta} \geq \mc{Q}(\bar{x}) \cdot \Wg(\bar{x} ; \{\bar{x}\}) = \mc{Q}(\bar{x})$.
\end{proof}

\subsubsection{Recourse lower bounds for ILS set cuts}
\label{subsection:monotonicity_lower_bounds}

Let~$S \subseteq V_+$ and suppose that~$k'$ is a lower bound on the number of routes that we need to attend the customers in~$S$, i.e., inequality~$x(E(S)) \leq |S| - k'$ is valid for~$\X \cap \Z^E$. Recall that~$\Xset{S, k'} = \left\{ x \in \X \cap \Z^E : x(E(S)) = |S| - k' \right\}$ and~$\Wdl{x ; \Xset{S, k'}} = 1 + (x(E(S)) - |S| + k')$. We define
an \emph{ILS set cut} (or simply \emph{set cut}) as an inequality of the form
\begin{equation}
    \label{ineq:set_cut}
    \theta(S) \geq \L(S, \Xset{S, k'}) \cdot \setWdl{x ; \mc{X}(S, k')},
\end{equation}
where~$\L(S, \Xset{S, k'})$ is a recourse lower bound.

Observe that~\eqref{ineq:set_cut} generalizes inequalities~\eqref{DL-shaped:set_cuts}, as the recourse lower bound~$\L(S, \Xset{S, k'})$ is not necessarily the same as~$\mc{L}_{DL}(S)$. Instead,~$\mc{L}_{DL}(S)$ is just one particular choice of recourse lower bound for the special case where~$\mc{Q}$ has a monotone disaggregation (which, by Theorem~\ref{thm:supperadditive}, is equivalent to stating that~$\mc{Q}$ is restrictively superadditive).

\begin{claim}
\label{claim:parada_set_cut}
Suppose that the fixed disaggregation of~$\mc{Q}$ satisfies Assumption~\ref{assumption:customer_disaggregation} and is monotone. Let~$\mc{L}_{DL}(S)$ be such that~$\mc{L}_{DL}(S) \leq \sum_{i = 1}^{k'} \mc{Q}(R_i)$, for every collection of routes~$R_1,\ldots,R_{k'}$ that are subroutes of routes in~$\routes$ and whose vertex sets~$\{V_+(R_i)\}_{i\in[k']}$ form a partition of~$S$. Then~$\mc{L}_{DL}(S)$ is a recourse lower bound with respect to~$S$ and~$\Xset{S,k'}$.
\end{claim}
\begin{proof}
    Let~$\bar{x} \in \Xset{S, k'}$ and~$\mc{R}(\bar{x}) = \{R_1, \ldots, R_k\}$. Recall that~$G(\bar{x})$ represents the support graph associated with~$\bar{x}$. The subgraph of~$G(\bar{x})$ induced by~$S$ is made of~$k'$ disjoint paths~$P_1, \ldots, P_{k'}$, each~$P_j$ corresponding to a route~$\hat{R}_j$ such that~$P_j = \hat{R}_j \setminus \{0\}$. For each route~$R_i \in \mc{R}(\bar{x})$, let~$\mc{H}_i$ be the set of all subroutes of~$R_i$ that belong to the set~$\{\hat{R}_1, \ldots, \hat{R}_{k'}\}$. Note that~$\cup_{i \in [k]} \mc{H}_i = \{\hat{R}_1, \ldots, \hat{R}_{k'}\}$ and~$\cup_{i \in [k']} V_+(\hat{R}_i) = S$. By monotonicity and the choice of~$\mc{L}_{DL}(S)$,~$$\sum_{i = 1}^k \sum_{v \in S} \mc{Q}(R_i, v) = \sum_{i = 1}^k \sum_{R' \in \mc{H}_i} \sum_{v \in V_+(R')} \mc{Q}(R_i, v) \geq \sum_{i = 1}^k \sum_{R' \in \mc{H}_i} \mc{Q}(R') = \sum_{i = 1}^{k'} \mc{Q}(\hat{R}_i) \geq \mc{L}_{DL}(S).$$
\end{proof}

The argument in Claim~\ref{claim:parada_set_cut} is limited to restrictively superadditive functions. On the other hand, the proposed framework determines exactly which recourse lower bounds are valid for a set cut (Claim~\ref{claim:valid_lower_bound}). As a result, set cuts can be used even when the recourse function is not restrictively superadditive. Section~\ref{section:application_vrpsd}, for example, applies these cuts to solve problem~$\vrpr(\Qc, \Xcvrp)$, even though~$\Qc$ is not restrictively superadditive (as shown in Example~\ref{example:qc_not_superadditive}).

\subsection{Summary}
\label{subsection:cut_summary}

Table~\ref{table:cutsummary} summarizes the results from this section. For each derived cut, column \tsc{rlb} indicates the recourse lower bound used, column \tsc{valid} refers to the conditions that the disaggregation or the set~$\X$ must satisfy for the cut to be valid for~$\F{\mc{Q}, \mc{X}, \Omega}$, and column \mbox{\tsc{formulate}} indicates the conditions under which the inequalities are sufficient to describe $\epiQ{\mc{Q}, \mc{X}}$ after projection (if such conditions are known). When \tsc{rlb} is listed as ``Generic'', it means that no recourse lower bound has been specified (except that it satisfies Definition~\ref{def:recourse_lower_bound}).

\begin{table}[H]
\setlength\extrarowheight{-4pt}
\centering
\small
\begin{tabular}{l@{\hskip 1em}c@{\hskip 1em}c@{\hskip 0.5em}c}
\toprule
\tsc{cut}  & \tsc{rlb} & \multicolumn{2}{c}{\tsc{condition}} \\ 
\cmidrule{3-4}
& & \tsc{valid} & \tsc{formulate} \\ \hline \\
\eqref{ineq:gendreau_cut} -- Gendreau et al. & $\mc{Q}(\bar{x})$ & $\X \subseteq \{x : x(\delta(0)) = 2k\}$ & $\X \subseteq \{x : x(\delta(0)) = 2k\}$ \\
\eqref{ineq:route_cut} -- Route cut & $\mc{Q}(R)$ & Any & Route disjoint (Def.~\ref{def:route_disjoint}) \\
\eqref{ineq:partial_route_cut} -- Partial route cut & Generic & Any &  - \\
\eqref{ineq:path_cut} -- Path cut & $\mc{Q}(R)$ & Monotone (Def.~\ref{def:monotonicity}) & Monotone (Def.~\ref{def:monotonicity})\\
\eqref{ineq:set_cut} -- Set cut & Generic & $\X \cap \Z^E \subseteq \{x : x(E(S)) \leq |S| - k'\}$ & - \\
\bottomrule
\end{tabular}
\caption{Summary of the cuts from Section~\ref{section:generalizing}.}
\label{table:cutsummary}
\end{table}

\section{Application to the VRPSD with scenarios under the classical recourse policy}
\label{section:application_vrpsd}

In this section, we apply our framework to the VRPSD with scenarios under the classical recourse policy, i.e., problem~$\vrpr(\Qc, \Xcvrp)$ defined in Section~\ref{subsection:vrpsd}. As before, we assume that we have a disaggregation of~$\Qc$ that satisfies Assumption~\ref{assumption:customer_disaggregation}. Following the ``recipe'' in Section~\ref{subsection:applying}, we immediately obtain an algorithm for solving~$\vrpr(\Qc, \Xcvrp)$ using either the cuts of~\cite{gendreau95} (Corollary~\ref{corollary:gendreau}) or the ILS route cuts (Theorem~\ref{thm:route_formulation}). The former approach is not expected to perform well, given the superior performance of recent algorithms that use recourse disaggregation (see the experiments in~\cite{hoogendoorn2023improved}). Therefore, our branch-and-cut algorithm captures the recourse cost using the ILS route cuts, and we concentrate here on developing recourse lower bounds for the additional ILS cuts introduced in Section~\ref{section:generalizing}. 

In fact, as the recourse function~$\Qc$ is not restrictively superadditive (Example~\ref{example:qc_not_superadditive}), the ILS path cuts~\eqref{ineq:path_cut} cannot be applied. Instead, Section~\ref{subsection:vrpsd_lower_bounds} focus on deriving recourse lower bounds for the partial route inequalities and the ILS set cuts. These bounds are then incorporated into our separation routine in Section~\ref{subsection:vrpsd_separation}.
 
\subsection{Deriving recourse lower bounds by counting failures}
\label{subsection:vrpsd_lower_bounds}

To derive recourse lower bounds for the VRPSD with scenarios under the classical recourse policy, we adopt the following strategy. We first compute a lower bound on the number of failures observed by a route while serving a given set of customers. This number is then used to construct a recourse lower bound by selecting that many customers in the set that are closest to the depot.

Based on the formula for~$\Qcscen(\oa{R})$ in Section~\ref{subsection:classical_recourse}, we introduce the following definition:
\begin{definition}
    \label{def:observe_failure}
    Let~$\oa{R} = (v_1, \ldots, v_\ell)$ be a directed route and let~$j \in [\ell]$. The \emph{number of failures observed by customer~$v_j$ in~$\oa{R}$} with respect to scenario~$\xi \in [N]$ is given by the expression~$\sum_{t = 1}^\infty \, \mb{I}\left(\sum_{i \in [j - 1]} d^\xi(v_i) \leq t C < \sum_{i \in [j]} d^\xi(v_i)\right)$.
    When the scenario is clear from context, we simply say that the previous summation gives the number of failures observed by~$v_j$ in~$\oa{R}$.
\end{definition}

Given a scenario~$\xi \in [N]$, we extend Definition~\ref{def:observe_failure} to undirected routes and solutions in~$\Xcvrp \cap \Z^E$ in a natural way: for an undirected route~$R$, the \emph{number of failures observed by~$v \in V_+(R)$} in~$R$ is the same as the number of failures observed by~$v$ in~$\oa{R}$, where~$\Qc(R) = \Qc(\oa{R})$ (ties are broken in an arbitrary but fixed way); similarly, the \emph{number of failures observed by~$v \in V_+$} in~$\bar{x} \in \Xcvrp \cap \Z^E$ is the number of failures observed by~$v$ in~$R \in \mc{R}(\bar{x})$, where~$R$ is the route in~$\mc{R}(\bar{x})$ that contains~$v$. 

We define~$\tsc{fail}(\alpha, \beta)$ as a function that counts the number of failures observed while collecting a demand of~$\beta \in \R_+$, assuming that the vehicle has already loaded a demand of~$\alpha \in \R_+$. In other words,
\begin{equation}
    \label{def:fail_function}
    \tag{$\tsc{fail}$}
    \failfunction{\alpha}{\beta} \coloneqq \sum_{t = 1}^\infty \mb{I}\left(\alpha \leq t C < \alpha + \beta \right).
\end{equation}

\noindent
Hence, given a directed route~$\oa{R} = (v_1, \ldots, v_\ell)$ and a scenario~$\xi \in [N]$, we have that~$\Qcscen(\oa{R}) = \sum_{j \in [\ell]} 2 \, c_{0 v_j} \failfunction{\sum_{i \in [j - 1]} d^\xi(v_i)}{d^\xi(v_j)}$.

To ease the notation, for every scenario~$\xi \in [N]$ and route~$R$, we abbreviate~$d^\xi(V_+(R))$ as~$d^\xi(R)$. The total number of failures observed along a route can be computed as follows.
\begin{lemma}
    \label{lemma:lower_bound_failures}
    For every route~$R = (v_1, \ldots, v_\ell)$, scenario~$\xi \in [N]$ and~$\alpha \in \R_+$, we have that~$\failfunction{\alpha}{d^\xi(R)} = \sum_{j \in [\ell]} \failfunction{\alpha + \sum_{i \in [j - 1]} d^\xi(v_i)}{d^\xi(v_j)}$.
\end{lemma}
\begin{proof}
    By the definition of~\ref{def:fail_function}($\,\cdot\,$),
    \begin{equation}
        \label{eq:proof_lemma_bound_failure}
        \sum_{j \in [\ell]} \failfunction{\alpha + \sum_{i \in [j - 1]} d^\xi(v_i)}{d^\xi(v_j)} = \sum_{j \in [\ell]} \sum_{t = 1}^\infty \mb{I}\left(\alpha + \sum_{i \in [j - 1]} d^\xi(v_i) \leq t C < \alpha + \sum_{i \in [j]} d^\xi(v_i)\right).
    \end{equation}
    Additionally, for every~$t \in \Z_{++}$,~$\sum_{j \in [\ell]} \mb{I}\left(\alpha + \sum_{i \in [j - 1]} d^\xi(v_i) \leq t C < \alpha + \sum_{i \in [j]} d^\xi(v_i)\right) \leq 1$ and equality holds if and only if~$\alpha \leq tC < \alpha + d^\xi(R)$. In this way, we rearrange the summations in~\eqref{eq:proof_lemma_bound_failure} to obtain~$\sum_{t = 1}^\infty \mb{I}\left(\alpha \leq t C < \alpha + d^\xi(R) \right) = \failfunction{\alpha}{d^\xi(R)}$.
\end{proof}

We also rewrite the expression defining~$\failfunction{\alpha}{\beta}$ to obtain a lower bound that does not depend on~$\alpha$:
\begin{restatable}{lemma}{lemmacountfailures}
    \label{lemma:count_failures}
    Let~$\alpha, \beta \in \R_+$ and define~$r(\alpha) \coloneqq \alpha - C \lfloor \alpha / C \rfloor$. Then~$\failfunction{\alpha}{\beta} = \left(\ceil[\Big]{\frac{\beta}{C}}- 1\right)^+$ whenever~$\alpha = 0$, and~$\failfunction{\alpha}{\beta} = \ceil[\Big]{\frac{r(\alpha) + \beta}{C}} -  \ceil[\Big]{\frac{r(\alpha)}{C}}$ otherwise. Consequently,~$\failfunction{\alpha}{\beta} \geq \left(\ceil[\Big]{\frac{\beta}{C}} - 1\right)^+$.
\end{restatable}
\begin{proof}
    Rewriting the definition of~$\failfunction{\alpha}{\beta}$ yields~$\failfunction{\alpha}{\beta} = \sum_{t = 1}^\infty \mb{I}(tC < \alpha + \beta) - \sum_{t = 1}^\infty \mb{I}(tC < \alpha)$, and this last expression can be further simplified using the following simple fact.
    \begin{fact}
    \begin{tightquote}
    \label{claim:proof_lemma_count_failures}
    For every~$\gamma \in \R_+$,~$\sum_{t = 1}^\infty \mb{I}(t < \gamma) = (\lceil \gamma \rceil - 1)^+$.
    \end{tightquote}
    \end{fact}
    \noindent
    Therefore, if~$\alpha = 0$, then~$\failfunction{\alpha}{\beta} = \sum_{t = 1}^\infty \mb{I}(tC < \alpha + \beta) = \left(\ceil[\Big]{\frac{\beta}{C}}- 1\right)^+$. Alternatively, when~$\alpha = qC + r(\alpha) > 0$, with~$q = \lfloor \alpha / C \rfloor$,
    \begin{align*}
        \failfunction{\alpha}{\beta} & = \left(\ceil[\Bigg]{\frac{\alpha + \beta}{C}} - 1\right) -  \left(\ceil[\Bigg]{\frac{\alpha}{C}} - 1 \right) = \ceil[\Bigg]{q + \frac{r(\alpha) + \beta}{C}} -  \ceil[\Bigg]{q + \frac{r(\alpha)}{C}} \\
        & = \ceil[\Bigg]{\frac{r(\alpha) + \beta}{C}} -  \ceil[\Bigg]{\frac{r(\alpha)}{C}} \geq \left(\ceil[\Bigg]{\frac{\beta}{C}} -  1\right)^+,
    \end{align*}
    where the last inequality follows from~$r(\alpha) \in [0, C)$.
\end{proof}

By Lemma~\ref{lemma:count_failures}, for every customer~$v \in V_+$ and scenario~$\xi \in [N]$, we may assume that~$d^\xi(v) \leq C$, as otherwise~$v$ observes at least~$\lceil d^\xi(v) / C \rceil - 1$ failures in any solution. So we can preprocess the instance by decrementing~$d^\xi(v)$ by~$q C$, with~$q = \lceil d^\xi(v) / C \rceil - 1$, and adding a constant term of~$q (2 c_{0v} p_\xi)$ to the objective function. This motivates the next assumption, which implies that every customer observes at most one failure (in any route and scenario~$\xi \in [N]$).
\begin{assumption}
    \label{assumption:customer_demands}
    For every scenario~$\xi \in [N]$ and customer~$v \in V_+$,~$d^\xi(v) \leq C$.
\end{assumption}

Having established Assumption~\ref{assumption:customer_demands} and the lower bound in Lemma~\ref{lemma:lower_bound_failures}, we now derive the desired recourse lower bounds. Let~$S = \{v_1, \ldots, v_\ell\} \subseteq V_+$ be such that~$c_{0v_1} \leq \cdots \leq c_{0 v_\ell}$. We define below a function, which, intuitively, represents a lower bound on the recourse cost of serving~$S$ (in scenario~$\xi \in [N]$) with~$\nu \in \Z_{++}$ vehicles and assuming that~$\alpha \in \R_+$ demand has already been collected:
\begin{equation}
    \label{def:vrpsd_lower_bound}
    \tag{$\mc{L}^\nu_\xi$}
    \mc{L}^\nu_\xi(\alpha, S) \coloneqq \sum_{j \in \left[\failfunction{\alpha}{d^\xi(S)} - \nu + 1\right]} 2 \, c_{0 v_j}.
\end{equation}

Lemma~\ref{lemma:lower_bound_failures} determines exactly the number of failures observed when traversing route~$R$ after collecting a value~$\alpha$ of demand. Combining this result with Lemma~\ref{lemma:count_failures} and Assumption~\ref{assumption:customer_demands} yields the following recourse lower bound for serving a set of customers with a given number of routes.

\begin{restatable}{proposition}{propsetlowerbound}
    \label{prop:set_recourse_lower_bound}
    Suppose that the customer demands satisfy Assumption~\ref{assumption:customer_demands} and let~$\hat{\mc{Q}}_C$ be the disaggregation of~$\Qc$ described in Remark~\ref{remark:disaggregation}. Let~$S \subseteq V_+$,~$k' \in \Z_{++}$, and define~$\mc{L}_C(S, k') \coloneqq \sum_{\xi \in [N]} p_\xi \,\vrpsdLB{k'}{0}{S}$. Then, for every~$\bar{x} \in \{x \in \X \cap \Z^E : x(E(S)) = |S| - k'\}$, we have that~$$\sum_{R \in \mc{R}(\bar{x})} \sum_{v \in S} \hat{\mc{Q}}_C(R, v) \geq \mc{L}_C(S, k').$$
    In particular, if~$x(E(S)) \leq |S| - k'$ is valid for~$\X \cap \Z^E$, then~$\vrpsdLB{k'}{0}{S}$ is a recourse lower bound for~$S$ with respect to~$\Xset{S, k'}$ and~$\hat{\mc{Q}}_C$.
\end{restatable}
\proof
    Let~$\bar{x} \in \X \cap \Z^E$ be as in the statement. We rearrange the term~$\sum_{R \in \mc{R}(\bar{x})} \sum_{v \in S} \hat{\mc{Q}}_C(R, v)$ by introducing the following definitions. For each route~$R = (v_1, \ldots, v_\ell)$, recall that~$\Qc(R) = \Qc(\oa{R})$. In this way, for each~$\xi \in [N]$ and~$i \in [\ell]$, we denote by~$\alpha^\xi_{R, v_i}$ the sum of the demands of the customers in~$\oa{R}$ preceding~$v_i$ in scenario~$\xi$, that is,~$\alpha^\xi_{R, v_i} \coloneqq \sum_{j \in [i - 1]} d^\xi(v_j)$. Using this definition, we have that~$\hat{\mc{Q}}_C(R, v_i) = 2 c_{0 v_i} \sum_{\xi \in [N]} p_\xi\, \failfunction{\alpha^\xi_{R, v_i}}{d^\xi(v_i)}$. 
    In the remainder of the proof, we show that, for every~$\xi \in [N]$,~$$\sum_{R \in \mc{R}(\bar{x})} \sum_{v \in V_+(R) \cap S} 2 c_{0 v} \, \failfunction{\alpha^\xi_{R, v}}{d^\xi(v)} \geq \vrpsdLB{k'}{0}{S}.$$

    Fix~$\xi \in [N]$ and let~$S = \{v'_1, \ldots v'_{|S|}\}$ be such that~$c_{0v'_1} \leq \ldots \leq c_{0 v'_{|S|}}$. Then
    ~$$\vrpsdLB{k'}{0}{S} = \sum_{j \in [1 - k' + (\lceil d^\xi(S) / C \rceil - 1)^{+}]} 2\,c_{0 v'_j} = \sum_{j \in [\lceil d^\xi(S) / C \rceil - k']} 2\,c_{0 v'_j}.$$
    Observe that if~$\lceil d^\xi(S) / C \rceil < 1$, then~$d^\xi(S) = 0$ and~$[\lceil d^\xi(S) / C \rceil - k'] = \emptyset$. 
    
    Hence, without loss of generality, we may assume that~$k' < \lceil d^\xi(S) / C \rceil$. Moreover, since the customers in~$S$ are ordered by their distance to the depot, it follows from Assumption~\ref{assumption:customer_demands} that it suffices to show that~$\sum_{R \in \mc{R}(\bar{x})} \sum_{v \in V_+(R) \cap S} \failfunction{\alpha^\xi_{R, v}}{d^\xi(v)} \geq \lceil d^\xi(S) / C \rceil - k'$.
    
    Since~$\bar{x}(E(S)) = |S| - k'$, the subgraph of the support graph~$G(\bar{x})$ induced by~$S$ consists of~$k'$ disjoint paths. These paths correspond to (customer-)disjoint routes~$\hat{R}_1, \ldots, \hat{R}_{k'}$ such that~$S = \bigcup_{i \in [k']} V_+(\hat{R}_i)$. For each~$i \in [k']$, write~$\hat{R}_i = (u^i_1, \ldots, u^i_{q_i})$ and let~$R_i$ be the unique route in~$\mc{R}(\bar{x})$ that contains~$\hat{R}_i$. Assume that~$u^i_1$ precedes~$u^i_q$ (if~$q \ge 2$) in~$\oa{R}_i$. Then
    \begin{align*}
        & \sum_{R \in \mc{R}(\bar{x})} \sum_{v \in V_+(R) \cap S} \failfunction{\alpha^\xi_{R, v}}{d^\xi(v)} = \sum_{i \in [k']} \sum_{v \in V_+(\hat{R}_i)} \failfunction{\alpha^\xi_{R, v}}{d^\xi(v)} \\
        & \overset{\text{Lemma~\ref{lemma:lower_bound_failures}}}{=} \sum_{i \in [k']} \failfunction{\alpha^\xi_{R, u^i_1}}{d^\xi(\hat{R}_i)} \overset{\text{Lemma~\ref{lemma:count_failures}}}{\geq} \sum_{i \in [k']} \left(\ceil[\Bigg]{\frac{d^\xi(\hat{R}_i)}{C}} - 1 \right) \geq \ceil[\Bigg]{\frac{d^\xi(S)}{C}} - k'. & \tag*{\Halmos} 
    \end{align*}

Finally, we derive recourse lower bounds for a partial route~$\H = (S_1, \ldots, S_\ell)$. In the case of independent demands, previous works~\citep{hjorring1999new, laporte2002, jabali2014} obtain such bounds by contracting each unstructured set~$S_j$ into a single vertex~$v_{S_j}$, whose expected demand (and variance) is the sum of the expected demands (and variances) of the vertices in~$S$, and whose recourse cost is~$\min_{v \in S_j} \{2 c_{0v}\}$. A natural way of adapting this approach to the VRPSD with scenarios under the classical recourse policy is to lower bound the recourse cost with respect to~$S_j$ by~$$\min_{v \in S_j} \{2 c_{0v}\} \cdot \sum_{\xi \in [N]} p_\xi \, \failfunction{\sum_{i \in [j - 1]} d^\xi(S_i)}{d^\xi(S_j)}.$$ Since this bound is always
at most~$\vrpsdLB{1}{\sum_{i \in [j - 1]} d^\xi(S_i)}{S_j}$, we improve this basic idea as follows:
\begin{proposition}
    \label{prop:partial_route_lower_bound}
    Suppose that customer demands satisfy Assumption~\ref{assumption:customer_demands}.
    For every partial route~$\H = (S_1, \ldots, S_\ell)$, define
    $$\oa{\mc{L}}_C(\H) \coloneqq \sum_{\xi \in [N]} p_\xi \sum_{j \in [\ell]} \vrpsdLB{1}{\sum_{i \in [j - 1]} d^\xi(S_i)}{S_j},$$
    and~$\cev{\mc{L}}_C(\H) \coloneqq \oa{\mc{L}}_C(\H')$, where~$\H' = (S_\ell, \ldots, S_1)$. Then, for every route~$R$ that adheres to~$H$,~$\Qc(R) \geq \mc{L}_C(H)$.
\end{proposition}
\begin{proof}
Fix~$\xi \in [N]$ and let~$\oa{R}$ be a directed route that exactly adheres to~$\H$. For every~$j \in [\ell]$, define~$h_j \coloneqq \sum_{i = 1}^j |S_i|$. By Definition~\ref{def:exact_adherence}, we may write~$\oa{R} = (v_1, \ldots, v_{h_\ell})$ where, for every~$j \in [\ell]$,~$S_j = \{v_{h_{j - 1} + 1}, \ldots, v_{h_j}\}$. 

Fix an arbitrary~$j \in [\ell]$ and let~$F \subseteq S_j$ be the set of customers in~$S_j$ that observe a failure in~$\oa{R}$ (with respect to scenario~$\xi \in [N]$). By Assumption~\ref{assumption:customer_demands} and Lemma~\ref{lemma:count_failures}, we know that~$|F| = \failfunction{\alpha + \sum_{i \in [j - 1]} d^\xi(v_i)}{d^\xi(S_j)}$. Moreover, as the lower bound~$\vrpsdLB{1}{\sum_{i \in [h_{j - 1}]} d^\xi(v_i)}{S_j}$ selects a set of~$\failfunction{\alpha + \sum_{i \in [j - 1]} d^\xi(v_i)}{d^\xi(S_j)}$ customers that are closest to the depot, it follows that
$$\sum_{q = h_{j - 1} + 1}^{h_j} 2 c_{0v_q} \sum_{t = 1}^\infty \mb{I}\left(\sum_{i \in [q - 1]} d^\xi(v_i) \leq t C < \sum_{i \in [q]} d^\xi(v_i)\right) = \sum_{v \in F} 2 c_{0v} \geq \vrpsdLB{1}{\sum_{i \in [h_{j - 1}]} d^\xi(v_i)}{S_j}.$$

This shows that~$\oa{\mc{L}}_C(\H) \leq \mc{Q}(\oa{R})$. A symmetrical reasoning yields~$\cev{\mc{L}}_C(\H) \leq \mc{Q}(\cev{R})$.
\end{proof}

\subsection{Separation routine}
\label{subsection:vrpsd_separation}

Our branch-and-cut algorithm for problem~$\vrpr(\Qc, \Xcvrp)$ uses two families of ILS cuts: the partial route cuts~\eqref{ineq:partial_route_cut} and the set cuts~\eqref{ineq:set_cut}. The discussion in Section~\ref{subsection:limitation_combining} shows that the set cuts may be invalid under a route-based recourse disaggregation of~$\Qc$ (as in Formulation~\eqref{formulation:partial_route}). Therefore, the set cuts are only used with the customer-based disaggregation described in Remark~\ref{remark:disaggregation}. To refer more easily to the different choices of recourse disaggregation, we introduce a parameter~$D \in \{1, 2\}$:~$D = 1$ corresponds to the route-based disaggregation in Formulation~\eqref{formulation:partial_route}, whereas~$D = 2$ corresponds to the one in Remark~\ref{remark:disaggregation}.\label{sym:D} In both cases, the disaggregation is along~$\Om = V_+$. 

Given a candidate solution~$(\bar{x}, \bar{\theta}) \in \R^E_+ \times \Q^{V_+}_+$, we separate valid inequalities for the feasible region~$\F{\mc{Q}_C, \mc{X}_{\tsc{cvrp}}}$ as follows. We first call the \emph{CVRPSEP package}~\citep{Lysgaard2004} to separate violated RCIs and get a corresponding family of customer sets~$\mc{S} \subseteq 2^{V_+}$. Although CVRPSEP separates RCIs heuristically, it also separates exactly the \emph{fractional capacity inequalities}, ensuring that the integer solutions found by our algorithm are indeed feasible (see the first two heuristics described in Section~2.1 of~\cite{Lysgaard2004}).

If~$\mc{S} \neq \emptyset$ and~$D = 2$, we also consider adding the set cuts~$\theta(S) \geq \setLc{S, \bark{S}} \cdot \setWdl{x ; \mc{X}(S, \bark{S})}$, for each~$S \in \mc{S}$. Otherwise, we call Algorithm~\ref{algorithm:partial_route} to generate a collection of partial routes~$\mc{H}$. For each partial route~$H \in \mc{H}$, we try to separate valid inequalities in the following way: if~$D = 2$, we test both the set cut with~$S = V_+(H)$ and the partial route cut~$\theta(V_+(H)) \geq \Lc{H} \cdot \Whs(x ; \Xh{H})$; if instead~$D = 1$, then the set cuts may not be valid, so we only separate the partial route cut~$\theta_v \geq \Lc{H} \cdot \Whs(x ; \Xh{H})$, where~$v$ is the customer in~$H$ with the smallest index.

Observe that when~$D = 2$ and~$H \in \mc{H}$ correspond to a route~$R$, we could instead use the (partial) route cut~$\theta(\Om(R)) \geq \Qc(R) \cdot \Whs(x; \Xh{R})$ (note that~$\theta(\Om(R)) \leq \theta(V_+(R))$). However, preliminary experiments suggest no benefit in doing so. A detailed description of our separation routine is given in Appendix~\ref{appendix:separation}.

\section{Computational study}
\label{section:experiments}

We conducted computational experiments on the VRPSD with scenarios under the classical recourse policy, that is, the problem~$\vrpr(\Qc, \Xcvrp)$ defined in Section~\ref{subsection:vrpsd}. All algorithms were executed in single-thread mode on a machine with an Intel(R) Xeon(R) Gold 6142 CPU @ 2.60GHz processor. We implemented the algorithms in C++, using Gurobi 12 as the LP/MIP solver and the Lemon graph library~\citep{lemon}. The time limit for each run was set to 1800 seconds.

We considered two benchmark sets of instances. The first set is based on the 270 instances proposed by~\cite{jabali2014} for a VRPSD problem under the classical recourse policy and with demands following independent normal distributions. For each instance, we generated 200 scenarios by sampling from these distributions and rounding the demand values to the nearest integer in the interval~$[0, C]$. The second benchmark set consists of the 20 instances proposed for the chance-constrained vehicle routing problem by~\cite{Dinh2018}. For each instance, they generated 200 demand scenarios by sampling from a multivariate normal distribution with correlations between customers proportional to their distances. The benchmark sets and tables containing detailed computational results were submitted jointly with the paper.

\paragraph{\textbf{Numerical results.}}
\label{subsection:numerical_results}

Recall from Section~\ref{subsection:vrpsd_separation} that~$D = 1$ corresponds to the route-based recourse disaggregation used in Formulation~\eqref{formulation:partial_route}, while~$D = 2$ corresponds to the recourse disaggregation in Remark~\ref{remark:disaggregation}. We evaluated three different algorithms:
\begin{itemize}[leftmargin=*]
\item \algBest: uses only partial route cuts with the disaggregation~$D = 1$;
\item \algBestD: uses only partial route cuts with the disaggregation~$D = 2$;
\item \algBestDset: extends~\algBestD~with the ILS set cuts~$\theta(S) \geq \setLc{S} \cdot \setWdl{x ; \mc{X}(S, \bark{S})}$.
\end{itemize}

The numerical results are summarized in Figure~\ref{figure:experiments} and Table~\ref{table:solved}. The table indicates the number of instances solved by each algorithm variant in the time limit of 1800 seconds. Each of Figures~\ref{figure:experiments1} and~\ref{figure:experiments2} has two parts: the left part shows the execution time, and the right part shows the final optimality gaps (relative to the best solution found by all algorithms). In the left part, a point~$(p_1, p_2) \in \R^2$ on the curve of an algorithm indicates that~$p_2\%$ of all instances were solved by the algorithm within~$p_1$ seconds. In the right part, a point~$(p'_1, p'_2)$ means that, for a given algorithm variant,~$p'_2\%$ of the instances had a final optimality gap of at most~$p'_1\%$.

We first note that the partial route cuts with~$D = 1$ dominate those with~$D = 2$. To see this, consider a partial route~$H$ and recall that with~$D = 2$ we use the inequality~$\theta(V_+(H)) \geq \Lc{H} \cdot \Whs(\bar{x}, \Xh{H})$, while with~$D = 1$ we use~$\theta_v \geq \Lc{H} \cdot \Whs(\bar{x}, \Xh{H})$, where~$v$ is the customer in~$H$ with the smallest index. Hence, any~$(\bar{x}, \bar{\theta})$ that satisfies the latter inequality also satisfies the former. However, Figure~\ref{figure:experiments} (and Table~\ref{table:solved}) suggests that this dominance has a limited impact on the overall algorithm performance.

In contrast, our results show that adding ILS set cuts significantly improves performance, indicating that it is better to choose recourse disaggregations that enable additional valid inequalities (such as the ILS set cuts~\eqref{ineq:set_cut}) rather than using the previous dominance argument. This reinforces our discussion in Section~\ref{subsection:disaggregated} on the importance of explicitly defining the disaggregation and the feasibility region~$\F{\mc{Q}, \mc{X}, \Omega}$. In doing so, we leveraged the results from Section~\ref{section:generalizing} to safely combine the partial route and the set cuts. In particular, our approach allowed us to use set cuts to solve the VRPR problem with a recourse function that is not restrictively superadditive (Definition~\ref{def:weak_superadditivity}), which is something that has not been done before.

Table~\ref{table:solved} confirms that~\algBestDset~solves more instances to optimality on both benchmark sets. We note that the previous ILS-based state-of-the-art algorithm for recourse functions that are not restrictively superadditive is due to~\cite{hoogendoorn2023improved}, which solves 246 instances from the benchmark set of~\cite{jabali2014} within one hour, assuming that demands are independent and normally distributed. While our experiments use different probability distributions, the scenarios were sampled from the same normal distributions, making our instances closely related to those used by~\cite{hoogendoorn2023improved}. Moreover, our approach can approximate a much broader class of probability distributions, and we solve 262 of the scenario-based instances (adapted from~\cite{jabali2014}) within a shorter time limit. Finally, we remark that our approach is the only exact algorithm currently available that can handle the correlations present in the instances of~\cite{Dinh2018}.

\begin{figure}[htb]
    \centering
    \subfloat[\cite{jabali2014} instances.]{
        \includegraphics[width=0.45\linewidth]{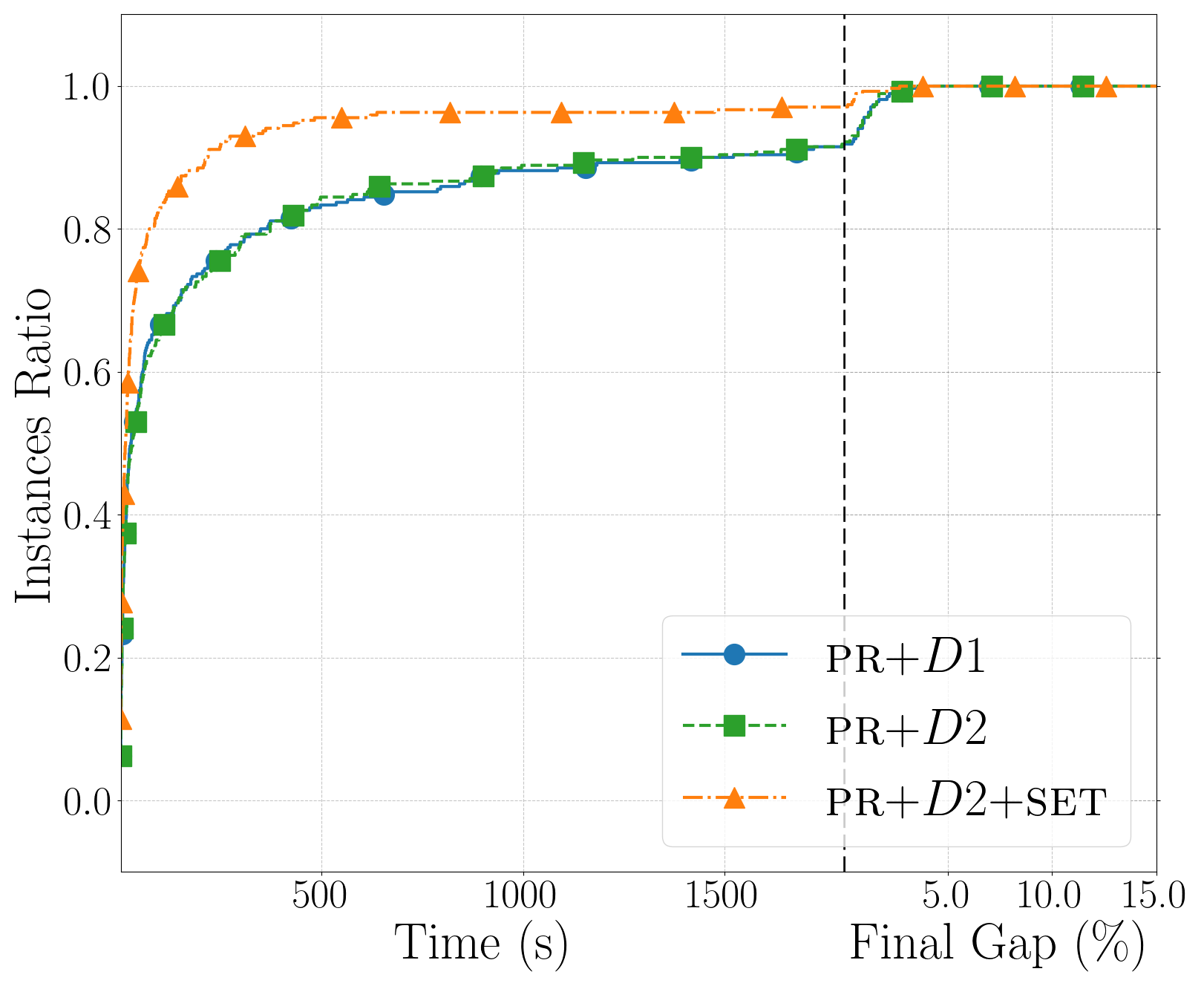}
        \label{figure:experiments1}
    }
    \hfill
    \subfloat[\cite{Dinh2018} instances.]{
        \includegraphics[width=0.45\linewidth]{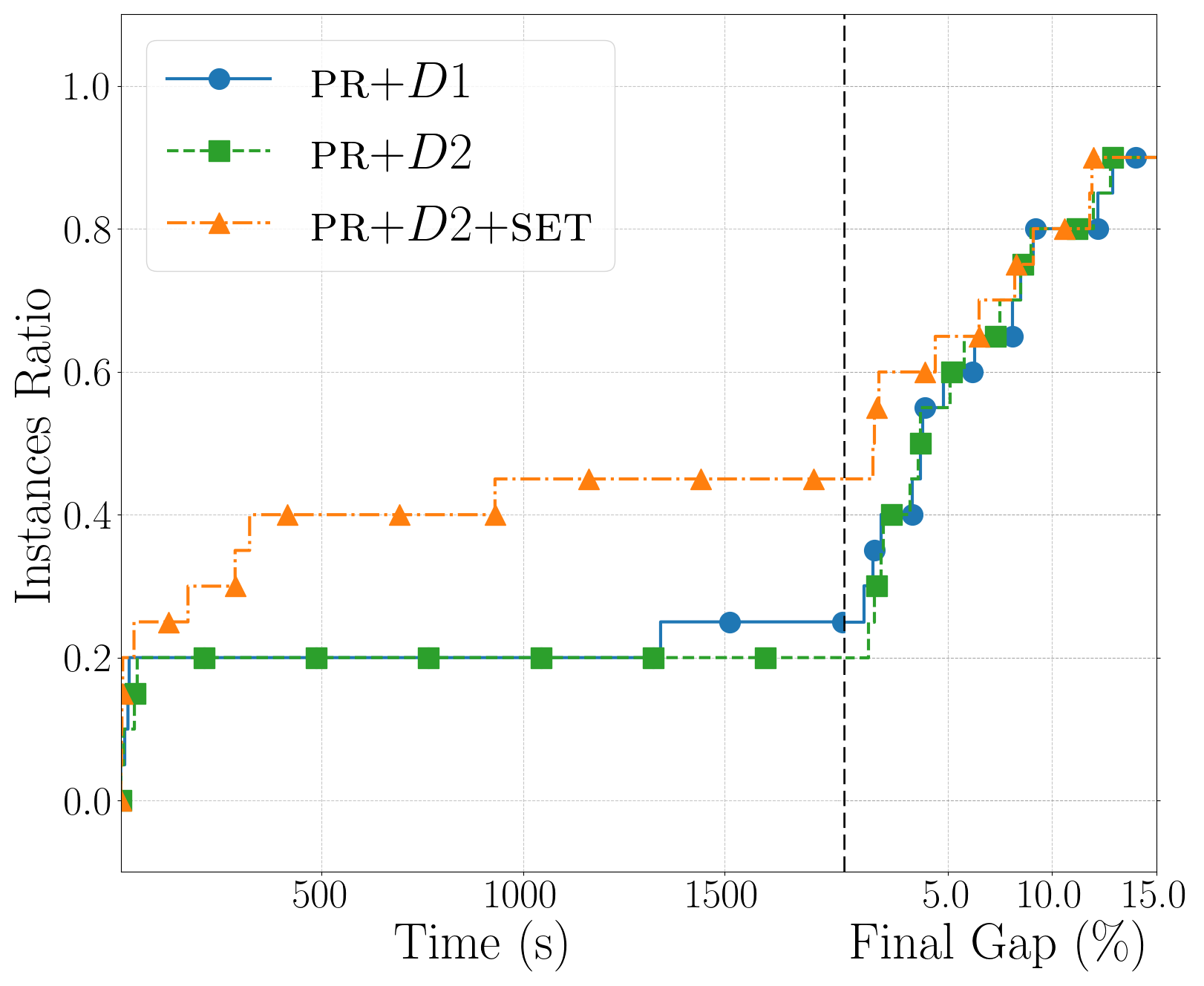}
        \label{figure:experiments2}
    }
    \\[0.3cm]
    \caption{
        Execution times and optimality gaps for different ILS branch-and-cut algorithms.
    }
    \label{figure:experiments}
\end{figure}

\begin{table}[htb!]
\centering
\begin{tabular}{l@{\hskip 2em}r@{\hskip 2em}r@{\hskip 2em}r@{\hskip 2em}r}
\toprule
Instance Set & \# Instances & \algBest & \algBestD & \algBestDset \\ 
\midrule
\cite{jabali2014} & 270 & 247 & 248 & 262 \\
\cite{Dinh2018} & 20 & 5 & 4 & 9 \\
\bottomrule
\end{tabular}
\caption{Total number of instances solved by each algorithm within the time limit of 1800 seconds.}
\label{table:solved}
\end{table}

\section{Concluding remarks}
\label{section:conclusion}

In this work, we proposed a framework for integer L-shaped (ILS) cuts that tries to explicitly identify common themes in previous ILS-based formulations. In particular, our results clarify the assumptions required to derive each component of specific ILS cuts. This generic perspective not only allowed us to generalize previous ILS cuts, but also enabled capabilities that were not possible before. For instance, we can now apply simultaneously cuts that could not be combined without our generalization. In addition, our framework allowed us to design a new ILS-based branch-and-cut approach for the VRPSD with demands given by scenarios, overcoming a well-known difficulty in the field of dealing with correlations.

We believe that future research on ILS-based methods for the VRPSD should try to follow our framework as much as possible. Specifically, future work should, if possible, identify the recourse disaggregation, feasible region, activation functions, and recourse lower bounds. Doing so would make these approaches applicable to a wider variety of problem variations, thus broadening their benefits to the community.

\bibliographystyle{plainnat}
\bibliography{bibliography}

\begin{APPENDICES}
\section{Notation summary}
\label{appendix:notation}

Table~\ref{table:symbols} summarizes the main symbols used in the paper. We abbreviate ``recourse lower bound'' as RLB.

\begin{table}[H]
\centering
\footnotesize
\renewcommand{\arraystretch}{1.2}
\begin{tabular}{l@{\hskip 0.5em}l@{\hskip 0.5em}l}
\toprule
\textbf{Symbol} & \textbf{Description} & \textbf{Reference / Page} \\
\midrule
$V_+$ & Set of customers. & Section~\ref{subsection:problem} / Page~\pageref{sym:V+} \\
$R, \oa{R}, \cev{R}$ & Routes and directed routes. & Section~\ref{subsection:problem} / Page~\pageref{sym:dir_route} \\
$\mc{Q}$ & Recourse function. & Section~\ref{subsection:problem} / Page~\pageref{sym:Q} \\
$G(\bar{x})$ & Support graph associated with~$\bar{x}$. & Section~\ref{subsection:problem} / Page~\pageref{sym:support_graph} \\
$\Xsub$ & Polytope containing all routing plans. & Section~\ref{subsection:problem} / Page~\pageref{set:subtour} \\
$\X$ & Polytope containing feasible routing plans. & Assumption~\ref{assumption:formulation} / Page~\pageref{assumption:formulation} \\
$\epiQ{\mc{Q}, \mc{X}}$ & Epigraph of~$\mc{Q}(x)$ (i.e., feasible region of~\ref{problem:vrpr}). & Section~\ref{subsection:problem} / Page~\pageref{problem:vrpr} \\
$d^\xi, \bar{d}$ & Demand in scenario~$\xi \in [N]$ and the expectation~$\mb{E}[d]$. & Section~\ref{subsection:vrpsd} / Page~\pageref{sym:demands} \\
$\Xcvrp$ & Polytope containing feasible CVRP solutions. & Section~\ref{subsection:cvrp} / Page~\pageref{set:cvrp} \\
$\Qc$ & Expected recourse cost of~$R$ under the classical recourse policy. & Section~\ref{subsection:classical_recourse} / Page~\pageref{eq:formula_dror} \\
$W(x ; \mc{X}')$ & Activation function. & Definition~\ref{def:activation_function} / Page~\pageref{def:activation_function} \\
$\H, V_+(\H)$ & Partial route~$\H$ and its customers~$V_+(\H)$. & Definition~\ref{def:partial_route} / Page~\pageref{def:partial_route} \\
$\Xh{H}$ & Solutions containing routes that adhere to~$\H$. & Definition~\ref{def:exact_adherence} / Page~\pageref{def:exact_adherence} \\
$\supsetXh{R'}$ & Solutions containing~$R'$ as a subroute. & Section~\ref{subsection:review_dl-shaped} / Page~\pageref{set:path_cut} \\
$\supsetXh{\H}$ & Solutions containing subroutes that adhere to~$\H$. & Appendix~\ref{appendix:activation_function} \\
$\supsetWof{x}{H}$ & Activation function for~$\supsetXh{\H}$. & Appendix~\ref{appendix:activation_function} \\
$\Whs(x ; \Xh{H})$ & Activation function for~$\Xh{H}$. & Remark~\ref{remark:minor_contribution} / Page~\pageref{remark:minor_contribution} \\
$\mc{L}(\H)$ & Lower bound on~$\mc{Q}(R)$, for every route~$R$ that adheres to~$\H$. & Formulation~\ref{formulation:partial_route} / Page~\pageref{formulation:partial_route} \\
$\Xset{S, k'}$ & Vectors~$x \in \X \cap \Z^E$ such that~$x(E(S)) = |S| - k'$. & Section~\ref{subsection:review_dl-shaped} / Page~\pageref{set:x_set} \\
$\Wdl{x ; \,\cdot\,}$ & Activation function of the DL-shaped method. & Section~\ref{subsection:review_dl-shaped} / Page~\pageref{eq:activation_set} \\ 
$\mc{L}_{DL}(S)$ & RLB for set cuts used by the DL-shaped method. & Section~\ref{subsection:review_dl-shaped} / Page~\pageref{eq:activation_set} \\
$R_1 \oplus R_2$ & Route concatenation. & Section~\ref{subsection:limitations} / Page~\pageref{sym:concat} \\
$\{\mc{Q}(R, \omega)\}_{\omega\in \Om}$ & Recourse disaggregation of~$\mc{Q}$. & Definition~\ref{def:recourse_disaggregation} / Page~\pageref{def:recourse_disaggregation} \\
$\Om(R)$ & Set of~$\omega \in \Om$ such that~$\mc{Q}(R, \omega) > 0$. & Definition~\ref{def:recourse_disaggregation} / Page~\pageref{def:recourse_disaggregation} \\
$\F{\mc{Q}, \X, \Om}$ & Feasible region associated with a disaggregation of~$\mc{Q}$. & Section~\ref{subsection:disaggregated} / Page~\pageref{set:F} \\
$\projrho(\,\cdot\,)$ & Projection from~$(x, \theta) \in \R^E \times \R^{\Om}_+$ onto~$(x, \rho) \in \R^E \times \R_+$. & Section~\ref{subsection:disaggregated} / Page~\pageref{sym:projrho} \\
$\L(U, \mc{X}')$ & RLB with respect to~$U \subseteq \Om$ and~$\mc{X}' \subseteq \X \cap \Z^E$. & Definition~\ref{def:recourse_lower_bound} / Page~\pageref{def:recourse_lower_bound} \\
$\mc{C}$ & Family of ILS cuts. & Definition~\ref{def:ils_cut} / Page~\pageref{def:ils_cut} \\
$\Wg(x ; \{\bar{x}\})$ & Activation function of Gendreau et al. (1995). & Section~\ref{subsection:gendreau} / Page~\pageref{sym:Wg} \\
$\routes$ & Set of realizable routes. & Definition~\ref{def:realizable_routes} / Page~\pageref{def:realizable_routes} \\
$\failfunction{\alpha}{\beta}$ & Nb. of failures observed while collecting~$\beta$ with initial load~$\alpha$. & Section~\ref{subsection:vrpsd_lower_bounds} / Page~\pageref{def:fail_function} \\
$\vrpsdLB{\nu}{\alpha}{S}$ & RLB for serving~$S \subseteq V_+$ with~$\nu$ vehicles and initial load~$\alpha$. & Section~\ref{subsection:vrpsd_lower_bounds} / Page~\pageref{def:vrpsd_lower_bound} \\
$\setLc{S, k'}, \Lc{\H}$ & RLBs for the classical recourse policy. & Section~\ref{subsection:vrpsd_lower_bounds} / Pages~\pageref{prop:set_recourse_lower_bound}, \pageref{prop:partial_route_lower_bound} \\
$D \in \{1, 2\}$ & Disaggregation type. & Section~\ref{subsection:vrpsd_separation} / Page~\pageref{sym:D} \\
\bottomrule
\end{tabular}
\caption{Summary of the main symbols used in the paper.}
\label{table:symbols}
\end{table}

\section{Globally valid recourse lower bounds}
\label{appendix:laporte}

Suppose that~$\texttt{LB} \in \Q^{\Om}_+$ is a vector of \emph{globally valid recourse lower bounds}, that is, for every~$\omega \in \Om$, we have~$\sum_{R \in \mc{R}(\bar{x})} \mc{Q}(R, \omega) \geq \texttt{LB}_{\omega}$, for all~$\bar{x} \in \X \cap \Z^E$. Define the \emph{translated feasible region}~$$\bar{\mc{F}} \coloneqq \F{\hat{\mc{Q}}, \mc{X}, \Omega} - \{\texttt{LB}\} = \{(x, \theta - \texttt{LB}) : (x, \theta) \in \F{\hat{\mc{Q}}, \mc{X}, \Omega}\}$$ and observe that
$$\min\{c^\T x + \allones^\T \theta : (x, \theta) \in \F{\hat{\mc{Q}}, \mc{X}, \Omega}\} = \allones^\T \texttt{LB} + \min\{c^\T x + \mathbbm{1}^\T \bar{\theta} : (x, \bar{\theta}) \in \bar{\mc{F}}\}.$$
So we can solve the original problem~\ref{problem:vrpr} by optimizing over~$\bar{\mc{F}}$.

The advantage of doing this translation step is that it may allow us to improve the coefficients in an ILS cut. Specifically, if we are given an ILS cut
\begin{equation}
    \label{ineq:global_lower_bound1}
    \theta(U) \geq \L(U, \mc{X}') \cdot \W(x ; \mc{X}'),
\end{equation}
then the inequality
\begin{equation}
    \label{ineq:global_lower_bound2}
    \bar{\theta}(U) \geq (\L(U, \mc{X}') - \texttt{LB}(U)) \cdot \W(x ; \mc{X}')
\end{equation}
is valid for~$\bar{\mc{F}}$. Moreover, if~$(x', \bar{\theta}')$ satisfies~$W(x' ; \mc{X}') \leq 1$ and~\eqref{ineq:global_lower_bound2}, then~$(x', \bar{\theta}' + \texttt{LB})$ satisfies the ILS cut~\eqref{ineq:global_lower_bound1}, since
\begin{align*}
    \bar{\theta}'(U) + \texttt{LB}(U) & \geq \texttt{LB}(U) + (\L(U, \mc{X}') - \texttt{LB}(U)) \cdot \W(x' ; \mc{X}') \\
    & \geq (\texttt{LB}(U) + \L(U, \mc{X}') - \texttt{LB}(U)) \cdot \W(x' ; \mc{X}') \\
    & = \L(U, \mc{X}') \cdot \W(x' ; \mc{X}').
\end{align*}

Therefore, given a family of ILS cuts~$\mc{C}$, rather than using Formulation~\ref{formulation:ILS}, it might be beneficial to instead use the translated formulation:
\begin{subequations}
\label{formulation:ILS2}
\begin{align*} 
\min ~~& c^\T x + \allones^\T \bar{\theta}, & \nonumber \\
\text{s.t.~~} & x \in \X \cap \Z^E, & \\
& \bar{\theta}(U) \geq (\mc{L}(U, \mc{X}') - \texttt{LB}(U)) \cdot (\alpha^\T x + \beta), & \forall (U, \mc{X}', \mc{L}(U, \mc{X}'), \alpha, \beta) \in \mc{C}, \\
& \bar{\theta} \in \R^{\Omega}_+.
\end{align*}
\end{subequations}

We remark, however, that for VRPSDs, globally valid recourse lower bounds are usually only available when no disaggregation is used, that is, when~$\Om = \{\hat{\omega}\}$. In this setting,~\cite{laporte2002} (see also~\cite{jabali2014, Salavati2019175, hoogendoorn2023improved}) replace the ILS cuts of~\cite{gendreau95} with the inequalities
\begin{equation*}
\label{ineq:laporte}
\bar{\theta}_{\hat{\omega}} = \theta_{\hat{\omega}} - \texttt{LB}_{\hat{\omega}} \geq (\mc{Q}(\bar{x}) - \texttt{LB}_{\hat{\omega}}) \cdot \Wg(x ; \{\bar{x}\}).
\end{equation*}
On the other hand, when recourse disaggregation is used (e.g., $\Om = V_+$), globally valid recourse lower bounds are typically not available. Still, ILS algorithms that use recourse disaggregation without globally valid recourse lower bounds tend to outperform algorithms that use such lower bounds but do not use recourse disaggregation (for instance, see the experiments in~\cite{hoogendoorn2023improved}).

\section{Activation function for partial routes}
\label{appendix:activation_function}

In this subsection, we present an alternative proof of the activation function~$\Whs(x ; \Xh{H})$ designed by~\cite{hoogendoorn2023improved}. Our goal here is both to keep the paper self-contained and to provide a simpler, more intuitive argument than the one given in their online appendix. In particular, we believe that their case analysis and algorithmic proof provide limited insight into the choice of their activation function coefficients (see Appendix~\ref{appendix:hoogendoorn} for their formula). In contrast, our derivation expresses~$\Whs(x ; \Xh{H}) - 1$ as the sum of nonpositive terms, making the correctness of our construction more transparent.

Let~$\H$ be a partial route. We first extend the definition of~$\mc{X}_{=}(\H)$ by relaxing the adherence 
condition. Recall that
\begin{align}
    & \mc{X}_{=}(\H) \coloneqq \left\{ x \in \X \cap \Z^E :~\text{there exists~$R 
    \in \mc{R}(x)$ such that~$R$ adheres to~$\H$} \right\}. \tag{$\mc{X}_{=}(\H)$} 
    \label{set:x_h_2} 
\end{align}
We define
\begin{align}
    & \mc{X}_{\supseteq}(\H) \coloneqq \left\{ x \in \X \cap \Z^E :~\text{$\exists R \in \mc{R}(x)$ and~$R' \subseteq R$ s.t.~$R'$ adheres 
    to~$\H$} \right\}, \tag{$\mc{X}_{\supseteq}(\H)$} \label{set:x_supseteq_h}
\end{align}
and note that~$\Xh{\H} \subseteq \supsetXh{\H}$.

\begin{figure}[htb!]
    \centering
    \includegraphics[width=0.35\textwidth]{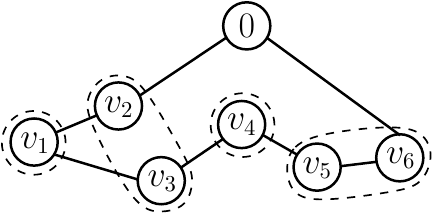}
    \caption{Solution~$\bar{x}$ consisting of a single route~$R = (v_2, v_1, v_3, v_4, v_5, v_6)$. Observe that~$R$ does not adhere to~$\H = (\{v_1\}, \{v_2, v_3\}, \{v_4\}, \{v_5, v_6\})$, so~$\bar{x} \notin \Xh{H}$. However, there exists a subroute of~$R$ that adheres to~$\H' = (\{v_4\}, \{v_5, v_6\})$, so~$\bar{x} \in \supsetXh{H}$.\label{figure:partial_route}}
\end{figure}

It follows from the definition of partial routes that membership with respect to~$\Xh{H}$ and~$\supsetXh{H}$ can be easily verified:

\begin{fact}
\label{fact:adherence}
Let~$\H = (S_1, \ldots, S_\ell)$ be a partial route and let~$\bar{x} \in \X \cap \Z^E$. Then the following hold:
\begin{enumerate}[(1)]
    \item $\bar{x}$ belongs to~$\supsetXh{H}$ if and only if~$\bar{x}(E(S_i)) = |S_i| - 1$, for every~$i \in [\ell]$, and~$\bar{x}(E(S_i, S_{i + 1})) = 1$, for every~$i \in [\ell - 1]$; and \label{item:fact_adherence1}
    \item $\bar{x}$ belongs to~$\Xh{H}$ if and only if~$\bar{x} \in \supsetXh{H}$ and
    \begin{enumerate}[(2a)]
        \item $\bar{x}(E(0, S_1)) = 2$, if~$\ell = 1$; or
        \item $\bar{x}(E(0, S_1)) = \bar{x}(E(0, S_\ell)) = 1$, if~$\ell \geq 2$.
    \end{enumerate}
    \label{item:fact_adherence2}
\end{enumerate}
\end{fact}
\begin{proof}
    We first prove~\ref{item:fact_adherence1}. Write~$\bar{x} = \sum_{R \in \mc{R}(\bar{x})} \mathbbm{1}_R$, where~$\mathbbm{1}_R$ is the characteristic vector of~$E(R)$. Suppose that~$\bar{x} \in \supsetXh{H}$, so there exists~$R' \in \mc{R}(\bar{x})$ such that~$R'$ adheres to~$\H$. By the definition of adherence, for every~$i \in [\ell]$, we have~$\mathbbm{1}_{R'}(E(S_i)) = |S_i| - 1$ and, if~$i \leq \ell - 1$,~$\mathbbm{1}_{R'}(E(S_i, S_{i + 1})) = 1$. Moreover, any other route~$R \in \mc{R}(\bar{x}) \setminus \{R'\}$ satisfies~$\mathbbm{1}_{R}(E(S_i)) = 0$ and~$\mathbbm{1}_{R}(E(S_i, S_{i + 1})) = 0$ (if~$i \leq \ell - 1$), as desired.
    
    To prove the converse, suppose now that~$\bar{x}(E(S_i)) = |S_i| - 1$, for every~$i \in [\ell]$, and~$\bar{x}(E(S_i, S_{i + 1})) = 1$, for every~$i \in [\ell - 1]$. Recall that~$G(\bar{x})$ is the support graph with respect to~$\bar{x}$. For every~$i \in [\ell]$, the edges in~$E(G(\bar{x})) \cap E(S_i)$ induce a path~$P_i$. As~$\bar{x}(E(S_i, S_{i + 1})) = 1$, for all~$i \in [\ell - 1]$, we can concatenate~$P_1, \ldots, P_\ell$ to form a route~$R = (v_1, \ldots, v_{|V_+(\H)|})$ such that~$E(R) \setminus \delta(0) \subseteq E(G(\bar{x}))$. Hence,~$R \in \mc{R}(\bar{x})$ adheres to~$\H$.

    The forward direction of~\ref{item:fact_adherence2} is easy to verify, so we only show the converse. Suppose that~$\bar{x} \in \supsetXh{H}$, so there exists a route~$R = (v_1, \ldots, v_{|V_+(\H)|}) \in \mc{R}(\bar{x})$ such that~$R$ adheres to~$\H$. If~$\bar{x}(E(0, S_1)) = \bar{x}(E(0, S_\ell)) = 1$ (or~$\bar{x}(E(0, S_1)) = 2$, if~$\ell = 1$), then both~$v_1$ and~$v_{|V_+(\H)|}$ are adjacent to the depot in~$G(\bar{x})$ (if~$R = (v_1)$, then~$\bar{x}_{0v_1} = 2$). So~$R$ adheres to~$\H$ and~$\bar{x} \in \Xh{H}$. 
\end{proof}

Recall that for a partial route~$H = (S_1, \ldots, S_\ell)$, we have~$E(\H) = \cup_{i \in [\ell]} (E(S_i, S_{i - 1} \cup S_{i + 1}) \cup E(S_i))$. We rewrite the condition in item~\ref{item:fact_adherence1} of Fact~\ref{fact:adherence} as follows.

\begin{lemma}
    \label{lemma:equivalence_adherence}
    Let~$\H = (S_1, \ldots, S_\ell)$ be a partial route and let~$\bar{x} \in \X \cap \Z^E$. Then~$\bar{x} \in \supsetXh{H}$ if and only if~$\bar{x}(E(\H) \setminus \delta(0)) = |V_+(\H)| - 1$ and~$\bar{x}(E(S_i)) = |S_i| - 1$, for every~$i \in [\ell]$.
\end{lemma}
\begin{proof}
    It follows from Fact~\ref{fact:adherence} that the condition in the statement is necessary for~$\bar{x} \in \supsetXh{H}$. To show sufficiency, assume that~$\bar{x}(E(\H) \setminus \delta(0)) = |V_+(\H)| - 1$ and~$\bar{x}(E(S_i)) = |S_i| - 1$, for all~$i \in [\ell]$. By Fact~\ref{fact:adherence}, we need to show that~$\bar{x}(E(S_i, S_{i + 1})) = 1$, for every~$i \in [\ell - 1]$. Thus, we may assume that~$\ell \geq 2$.

    Since~$\bar{x}$ satisfy the SECs, for any~$i \in [\ell - 1]$, we have that
    \begin{align*}
        (|S_i| - 1) + (|S_{i + 1}| - 1) + \bar{x}(E(S_i, S_{i + 1})) & = \bar{x}(E(S_i)) + \bar{x}(E(S_{i + 1})) + \bar{x}(E(S_i, S_{i + 1}))\\
        & \leq |S_i| + |S_{i + 1}| - 1,
    \end{align*}
    meaning that~$\bar{x}(E(S_i, S_{i + 1})) \leq 1$. Hence,
    \begin{align*}
        |V_+(\H)| - 1 & = \bar{x}(E(\H) \setminus \delta(0)) = \sum_{i \in [\ell]} \bar{x}(E(S_i)) + \sum_{i \in [\ell - 1]} \bar{x}(E(S_i, S_{i + 1})) \\
        & \leq \sum_{i \in [\ell]} (|S_i| - 1) + (\ell - 1) = |V_+(\H)| - 1,
    \end{align*}
    which implies that~$\bar{x}(E(S_i, S_{i + 1})) = 1$, for every~$i \in [\ell - 1]$.
\end{proof}

A consequence of Lemma~\ref{lemma:equivalence_adherence} is that the expression $$1 + (x(E(\H) \setminus \delta(0)) - |V_+(\H)| + 1) + \sum_{i \in [\ell]} (x(E(S_i)) - |S_i| + 1)$$ already gives an activation function with respect to~$\supsetXh{H}$. The next result exploits the structure of partial routes to derive a stronger activation function.

\begin{proposition}
    \label{proposition:activation_function_adheres}
    Let~$\H = (S_1, \ldots, S_\ell)$ be a partial route. Then
    \begin{equation*}
        W_{OF}(x; \mc{X}_{\supseteq}(\H)) \coloneqq 1 + (x(E(\H) \setminus \delta(0)) - |V_+(\H)| + 1) + \sum_{i \in \{2, \ell - 1\} \cap [\ell]}(x(E(S_i)) - |S_i| + 1).
    \end{equation*}
    is an activation function.
\end{proposition}
\begin{proof}
    By Lemma~\ref{lemma:equivalence_adherence}, any~$x' \in \supsetXh{H}$ satisfies~$\supsetWof{x'}{H} = 1$. Conversely, suppose that~$\bar{x} \in \X \cap \Z^E$ is such that~$\supsetWof{\bar{x}}{H} = 1$, meaning that~$\bar{x}(E(\H) \setminus \delta(0)) = |V_+(\H)| - 1$ and~$\bar{x}(E(S_i)) = |S_i| - 1$, for every~$i \in \{2, \ell - 1\} \cap [\ell] $. It follows from Lemma~\ref{lemma:equivalence_adherence} that it suffices to show that~$\bar{x}(E(S_i)) = |S_i| - 1$, for every~$i \in [\ell] \setminus \{2, \ell - 1\}$, so we may assume that~$\ell \geq 3$.
    
    We first consider~$i = 1$. Since~$\bar{x}(E(\H) \setminus \delta(0)) = |V_+(\H)| - 1$, we know that~$G(\bar{x})$ contains a path spanning~$V_+(\H)$. Therefore,~$\bar{x}(E(S_1, S_2)) \geq 1$ and~$\bar{x}(E(S_2, S_3)) \geq 1$, meaning that~$\bar{x}(\delta(S_2)) \geq 2$. Summing the degree constraints for the vertices in~$S_2$ and using~$\bar{x}(E(S_2)) = |S_2| - 1$, we learn that~$2 |S_2| = 2 \, \bar{x}(E(S_2)) + \bar{x}(\delta(S_2)) \geq 2 (|S_2| - 1) + 2 = 2 |S_2|$, which implies that~$\bar{x}(E(S_1, S_2)) = \bar{x}(E(S_2, S_3)) = 1$. Now take~$H' = (S_2, \ldots, S_\ell)$, then
    \begin{align*}
        |V_+(\H)| - 1 & = \bar{x}(E(\H) \setminus \delta(0)) = \bar{x}(E(S_1)) + \bar{x}(E(S_1, S_2)) + \bar{x}(E(H') \setminus \delta(0)) \\
        & \leq (|S_1| - 1) + 1 + (|V_+(H')| - 1) = |V_+(\H)| - 1,
    \end{align*}
    proving that~$\bar{x}(E(S_1)) = |S_1| - 1$. By symmetry, this also shows that~$\bar{x}(E(S_\ell)) = |S_\ell| - 1$.

    To close the proof, let~$i \in \{3, \ldots, \ell - 2\}$. If~$S_i$ is a singleton we are done, otherwise, we know that~$S_{i - 1} = \{u\}$ and~$S_{i + 1} = \{v\}$. As~$G(\bar{x})$ contains a path spanning~$V_+(\H)$, we know that~$\bar{x}(E(u, S_{i - 2})) \geq 1$ and~$\bar{x}(E(u, S_i)) \geq 1$, which implies that equality is attained in both cases (since~$\bar{x}(\delta(u)) = 2$). The same argument can be applied to conclude that~$\bar{x}(E(v, S_i)) = 1$. Next, consider the partial routes~$H_1 = (S_1, \ldots, S_{i - 1})$ and~$H_2 = (S_{i + 1}, \ldots, S_\ell)$, then
    \begin{align*}
        |V_+(\H)| - 1 & = \bar{x}(E(\H) \setminus \delta(0)) \\
        & = \bar{x}(E(H_1) \setminus \delta(0)) + \bar{x}(E(u, S_i)) + \bar{x}(E(S_i)) + \bar{x}(E(S_i, v)) + \bar{x}(E(H_2) \setminus \delta(0)) \\
        & \leq (|V_+(H_1)| - 1) + 1 + (|S_i| - 1) + 1 + (|V_+(H_2)| - 1) = |V_+(\H)| - 1, 
    \end{align*}
    which implies that~$\bar{x}(S_i) = |S_i| - 1$.
\end{proof}

Combining Proposition~\ref{proposition:activation_function_adheres} with item~\ref{item:fact_adherence2} of Fact~\ref{fact:adherence} we obtain the desired activation function for~$\Xh{H}$. In Appendix~\ref{appendix:hoogendoorn} we show that the resulting activation function is indeed equivalent to the one proposed by~\cite{hoogendoorn2023improved}.

\begin{theorem}
    \label{thm:partial_route_activation_exact_adheres}
    Let~$\H = (S_1, \ldots, S_\ell)$ be a partial route. Then
    \begin{align*}
        \Whs(x; \Xh{H}) \coloneqq \supsetWof{x}{H} + \sum_{i \in \{1, \ell\}} (x(\delta(S_i) \cap E(\H)) + 2 x(E(S_i)) - 2 |S_i|)
    \end{align*}
    is an activation function.
\end{theorem}
\begin{proof}
    We first show that if~$\bar{x} \in \Xh{\H}$, then~$\Whs(\bar{x}; \Xh{H}) = 1$. Since~$\Xh{\H} \subseteq \supsetXh{H}$, we know that~$\supsetWof{\bar{x}}{H} = 1$. If~$\ell = 1$, it follows from Fact~\ref{fact:adherence} that~$\bar{x}(E(0, S_1)) + 2 \bar{x}(E(S_1)) = 2 + 2(|S_1| - 1) = 2|S_1|$. Similarly, if~$\ell \geq 2$, then~$\bar{x}(E(S_{i - 1}, S_i)) + 2 \bar{x}(E(S_i)) + \bar{x}(E(S_i, S_{i + 1})) = 2 + 2(|S_i| - 1) = 2|S_i|$, for every~$i \in \{1, \ell\}$ (recall that~$S_0 = S_{\ell + 1} = \{0\}$). Overall, this proves that~$\Whs(\bar{x}; \Xh{H}) = 1$.
    
    To prove the converse, we assume~$\Whs(\bar{x}; \Xh{\H}) = 1$ and we show that~$\bar{x} \in \Xh{\H}$. Since~$\supsetWof{x}{H}$ is an activation function and all the terms in~$\Whs(\bar{x}; \Xh{\H}) - \supsetWof{x}{H}$ are nonpositive for any~$x \in \X$, we have that~$1 = \Whs(\bar{x}; \Xh{H}) \leq \supsetWof{\bar{x}}{H} \leq 1$. This implies that~$\supsetWof{\bar{x}}{H} = 1$ (i.e.,~$\bar{x} \in \supsetXh{H}$) and all the terms in~$\Whs(\bar{x}; \Xh{\H}) - \supsetWof{\bar{x}}{H}$ evaluate to zero. 
    
    Now suppose that~$\ell = 1$. Then~$\bar{x} \in \supsetXh{H}$ implies that~$\bar{x}(E(S_1)) = |S_1| - 1$, and since~$\bar{x}(E(0, S_1)) + 2 \bar{x}(E(S_1)) = 2 |S_1|$, it follows that~$\bar{x}(E(0, S_1)) = 2$. Applying Fact~\ref{fact:adherence} we conclude that~$\bar{x} \in \Xh{H}$. Similarly, if~$\ell \geq 2$,~$\bar{x} \in \supsetXh{H}$ implies that~$\bar{x}(E(S_1)) = |S_1| - 1$ and~$\bar{x}(E(S_1, S_2)) = 1$. Since~$\bar{x}(E(0, S_1)) + 2 \bar{x}(E(S_1)) + \bar{x}(E(S_1, S_2)) = 2 |S_1|$, we have that~$\bar{x}(E(0, S_1)) = 1$. A symmetric argument shows that~$\bar{x}(E(0, S_\ell)) = 1$, so we are done by Fact~\ref{fact:adherence}.
\end{proof}

\section{Equivalence with the activation function of Hoogendoorn and Spliet (2023)}
\label{appendix:hoogendoorn}

Let~$\H = (S_1, \ldots, S_\ell)$ be a partial route (with~$S_0 = \{0\}$ and~$S_{\ell + 1} = \{0\}$). \cite{hoogendoorn2023improved} proposed the following activation function:
\begin{equation*}
    \Whs(x; \Xh{H}) = \gamma + \sum_{i = 1}^{\ell} \alpha_i (x(E(S_i)) - |S_i| + 1) + \sum_{i = 0}^{\ell} \beta_i \left( x(E(S_i, S_{i + 1})) - 1 \right),
\end{equation*}
where
\begin{align*}
(\alpha_1, \ldots, \alpha_\ell) & =
    \begin{cases}
    (3), & \text{if~$\ell = 1$,} \\
    (4, 4), & \text{if~$\ell = 2$,} \\
    (3, 2, 3), & \text{if~$\ell = 3$,} \\
    (3, 2, 1, \ldots, 1, 2, 3), & \text{if~$\ell \geq 4$.}
    \end{cases} \\
(\beta_0, \ldots, \beta_\ell) & =
    \begin{cases}
    (1, 0), & \text{if~$\ell = 1$,} \\
    (1, 3, 1), & \text{if~$\ell = 2$,} \\
    (1, 2, 1, \ldots, 1, 2, 1), & \text{if~$\ell \geq 3$,}
    \end{cases} \\
\gamma & =
    \begin{cases}
    0, & \text{if~$\ell = 1$,} \\
    1, & \text{if~$\ell \geq 2$.}
    \end{cases}
\end{align*}

To see that~$\Whs$ is equivalent to the activation function in Theorem~\ref{thm:partial_route_activation_exact_adheres}, we check the different cases for~$\ell$:

\noindent
\textbf{Case $\ell = 1$:} 
\begin{flalign*}
    \Whs(x ; \Xh{H}) & = (x(E(0, S_1)) - 1) + 3 (x(E(S_1)) - |S_1| + 1) \\
    & = 1 + (x(E(\H) \setminus \delta(0)) - |V_+(\H)| + 1) + (x(E(0, S_1)) + 2 x(E(S_1)) - 2 |S_1|) \\
    & = \supsetWof{x}{H} + (x(E(0, S_1)) + 2 x(E(S_1)) - 2 |S_1|)
\end{flalign*}

\noindent
\textbf{Case $\ell = 2$:}
\begin{flalign*}
    \Whs(x ; \Xh{H}) & = 
    \begin{aligned}[t]
    & 1 + (x(E(0, S_1)) - 1) + 4 (x(E(S_1)) - |S_1| + 1) + 3 (x(S_1, S_2) - 1) \\
    & + 4 (x(S_2) - |S_2| + 1) + (x(0, S_2) - 1)
    \end{aligned} \\
    & = 
    \begin{aligned}[t]
    & 1 + (x(E(\H) \setminus \delta(0)) - |V_+(\H)| + 1)\\
    & + (x(E(S_1)) - |S_1| + 1) + (x(E(S_2)) - |S_2| + 1) \\
    & + (x(E(0, S_1)) + 2 x(E(S_1)) + x(E(S_1, S_2)) - 2 |S_1|) \\
    & + (x(E(0, S_2)) + 2 x(E(S_2)) + x(E(S_1, S_2)) - 2 |S_2|)
    \end{aligned}\\
    & = 
    \begin{aligned}[t]
    & \supsetWof{x}{H} \\
    & + (x(E(0, S_1)) + 2 x(E(S_1)) + x(E(S_1, S_2)) - 2 |S_1|) \\
    & + (x(E(0, S_2)) + 2 x(E(S_2)) + x(E(S_1, S_2)) - 2 |S_2|) \\
    \end{aligned}
\end{flalign*}

\noindent
\textbf{Case $\ell = 3$:}
\begin{flalign*}
    \Whs(x ; \Xh{H}) & = 
    \begin{aligned}[t]
    & 1 + (x(E(0, S_1)) - 1) + 3 (x(E(S_1)) - |S_1| + 1) + 2 (x(E(S_1, S_2)) - 1) \\
    & + 2 (x(E(S_2)) - |S_2| + 1) + 2 (x(E(S_2, S_3)) - 1) \\
    & + 3 (x(E(S_3)) - |S_3| + 1) + (x(E(0, S_3)) - 1)
    \end{aligned} \\
    & = 
    \begin{aligned}[t]
    & 1 + (x(E(\H) \setminus \delta(0)) - |V_+(\H)| + 1) + (x(E(S_2)) - |S_2| + 1)\\
    & + (x(E(0, S_1)) + 2 x(E(S_1)) + x(E(S_1, S_2)) - 2 |S_1|) \\
    & + (x(E(0, S_3)) + 2 x(E(S_3)) + x(E(S_2, S_3)) - 2 |S_3|)
    \end{aligned}\\
    & =
    \begin{aligned}[t]
    & \supsetWof{x}{H} \\
    & + (x(E(0, S_1)) + 2 x(E(S_1)) + x(E(S_1, S_2)) - 2 |S_1|) \\
    & + (x(E(0, S_3)) + 2 x(E(S_3)) + x(E(S_2, S_3)) - 2 |S_3|)
    \end{aligned}\\
\end{flalign*}

\noindent
\textbf{Case $\ell \geq 4$:}
\begin{flalign*}
    \Whs(x ; \Xh{H}) & = 
    \begin{aligned}[t]
    & 1 + (x(E(0, S_1)) - 1) + 3 (x(E(S_1)) - |S_1| + 1) + 2 (x(E(S_1, S_2)) - 1) \\
    & + 2 (x(E(S_2)) - |S_2| + 1) + (x(E(S_2, S_3)) - 1) \\
    & + \sum_{i = 3}^{\ell - 2} (x(E(S_i)) - |S_i| + 1) + (x(E(S_i, S_{i + 1})) - 1) \\
    & + 2 (x(E(S_{\ell - 1})) - |S_{\ell - 1}| + 1) + 2 (x(E(S_{\ell - 1}, S_\ell)) - 1) \\
    & + 3 (x(E(S_\ell)) - |S_\ell| + 1) + (x(E(0, S_\ell)) - 1)
    \end{aligned}\\
    & = 
    \begin{aligned}[t]
    & 1 + (x(E(\H) \setminus \delta(0)) - |V_+(\H)| + 1)\\
    & + (x(E(S_2)) - |S_2| + 1) + (x(E(S_{\ell - 1})) - |S_{\ell - 1}| + 1) \\
    & + (x(E(0, S_1)) + 2 x(E(S_1)) + x(E(S_1, S_2)) - 2 |S_1|) \\
    & + (x(E(0, S_\ell)) + 2 x(E(S_\ell)) + x(E(S_{\ell - 1}, S_\ell)) - 2 |S_\ell|) 
    \end{aligned}\\
    & = 
    \begin{aligned}[t]
    & \supsetWof{x}{H} \\
    & + (x(E(0, S_1)) + 2 x(E(S_1)) + x(E(S_1, S_2)) - 2 |S_1|) \\
    & + (x(E(0, S_\ell)) + 2 x(E(S_\ell)) + x(E(S_{\ell - 1}, S_\ell)) - 2 |S_\ell|) 
    \end{aligned}
\end{flalign*}

\section{Separating partial route inequalities}
\label{appendix:partial_route_separation}

In this appendix section, we present a separation heuristic for partial route inequalities that also gives a slight (theoretical) improvement upon the one proposed by~\cite{hoogendoorn2023improved}. Like in their heuristic, given a fractional solution~$\bar{x} \in \X$, our algorithm uses the classical depth-first search procedure of~\cite{hopcroft1973algorithm} to detect biconnected components and cut-vertices (or articulation points) of the support graph~$G(\bar{x}) \setminus \{0\}$. However, \cite{hoogendoorn2023improved} only consider the components of~$G(\bar{x}) \setminus \{0\}$ whose total flow to the depot is exactly 2. Instead, we leverage the concept of \emph{block-cut trees} (also known as \emph{block graphs}, see Section~3.1 of~\cite{diestel}) to generalize their approach and possibly separate more partial route inequalities.

To describe our heuristic, we recall some graph-theoretic concepts. Let~$G'$ be an arbitrary graph. A \emph{cut vertex} is a vertex~$v \in V(G')$ such that~$G' - v$ has more connected components than~$G'$, and we say that~$G'$ is \emph{biconnected} if it is connected and has no cut vertex. For our purposes, a \emph{block}~$B \subseteq V(G')$ is the vertex set of a maximal (inclusionwise) biconnected subgraph of~$G'$. Let~$\mc{B}(G')$ (resp.~$\mc{C}(G')$) denote the set of all blocks (resp. cut vertices) of~$G'$. Assuming that~$G'$ is connected, the \emph{block-cut tree} of~$G'$ is a bipartite graph~$T$ with vertex set~$\mc{B}(G') \cup \mc{C}(G')$ and edges~$\{B, v\}$ for every block~$B \in \mc{B}(G')$ and cut vertex~$v \in B$. It is easy to verify that~$T$ is a tree: if~$T$ contained a cycle, then that cycle would contain a vertex in~$\mc{C}(G')$, contradicting the definition of a cut vertex (see~\cite{diestel}). When~$G'$ is not connected, we have a block-cut tree for each connected component~$G'_1, \ldots, G'_t$ of the graph~$G'$. In this case, the notations~$\mc{B}(G')$ and~$\mc{C}(G')$ refer to the sets~$\cup_{i \in [t]} \mc{B}(G'_i)$ and~$\cup_{i \in [t]} \mc{C}(G'_i)$, respectively.

Let~$(\bar{x}, \bar{\theta}) \in \X \times \R^{V_+}_+$ be a given candidate solution. We construct block-cut trees from~$G(\bar{x})$ as follows. Let~$\bar{V}$ be a set containing exactly one ``dummy vertex''~$\texttt{dummy}(v)$ for each customer~$v \in V_+$ with~$1 \leq \bar{x}_{0v} < 2$, that is,~$\bar{V} = \{\texttt{dummy}(v): 1 \leq \bar{x}_{0v} < 2, v \in V_+\}$. Let~$\bar{G}$ be an auxiliary graph obtained from~$G(\bar{x})$ by deleting all edges connected to the depot and connecting each dummy vertex~$\texttt{dummy}(v)$ to its corresponding customer~$v \in V_+$. Note that, due to the introduction of the dummy vertices, every vertex~$v \in V_+$ with~$1 \leq \bar{x}_{0v} < 2$ is a cut vertex of~$\bar{G}$. Let~$\bar{G}_1, \ldots, \bar{G}_t$ be the connected components of~$\bar{G}$. The \emph{block-cut forest associated with~$G(\bar{x})$} is given by the set~$\mc{T} = \{T_i\}_{i \in [t]}$, where, for each~$i \in [t]$,~$T_i$ is a block-cut tree of~$\bar{G}_i$. We represent each vertex in~$\mc{T}$ by the symbol~$S$, where~$S$ is either a block or a cut vertex. Moreover, for each~$i \in [t]$, we use~$V_+(T_i)$ to refer to the set of customers in~$T_i$, that is,~$V_+(T_i) \coloneqq ((\cup_{S \in \mc{B}(T_i)} S) \cup \mc{C}(T_i)) \cap V_+$.

\begin{figure}[htb]
    \centering
    \subfloat[A fractional solution~$\bar{x} \in \X$ with support graph~$G(\bar{x})$. Solid edges have~$\bar{x}_e = 1$, while dashed edges have~$\bar{x}_e = 0.5$.]{
        \includegraphics[width=0.3\textwidth]{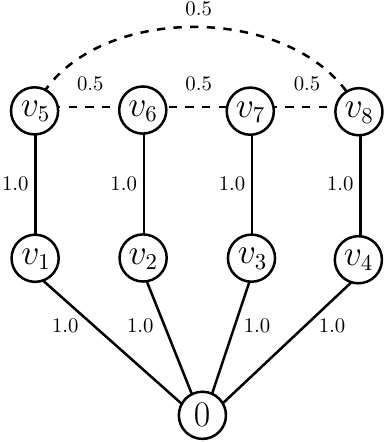}
        \label{figure:tree1}
    }
    \hspace{2.5cm}
    \subfloat[The auxiliary graph~$\bar{G}$ associated with~$G(\bar{x})$. The dashed circled regions indicate the blocks.]{
        \includegraphics[width=0.3\textwidth]{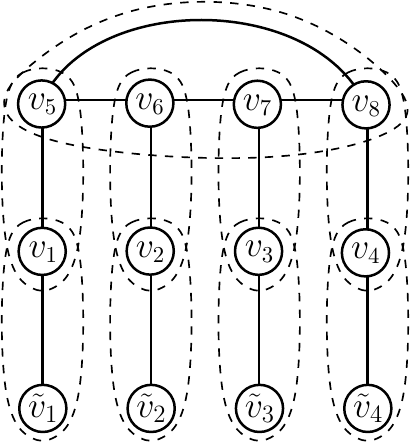}
        \label{figure:tree2}
    }
    \\[0.4cm] 
    \caption{Example of a solution~$\bar{x} \in \X$ where~$G(\bar{x}) \setminus \{0\}$ is made of a single component whose total flow to the depot is 4. The block–cut tree associated with~$G(\bar{x})$ contains a branching vertex corresponding to the block~$\{v_5, v_6, v_7, v_8\}$. In this case, our algorithm considers partial route~$\H = (\{v_1\}, \{v_5\}, \{v_6\}, \{v_2\})$ and we have that~$\Whs(\bar{x}; \Xh{H}) = 0.5$. \label{figure:block-cut-tree}}
\end{figure}

\begin{algorithm}[htb]
  \hspace*{\algorithmicindent} \textbf{Input:} Vectors~$\bar{x} \in \X$ and~$\bar{\theta} \in\Q^{V_+}_+$.\\
  \hspace*{\algorithmicindent} \textbf{Output:} A set of partial routes~$\mc{H}$ for which we try to separate partial route inequalities.
\begin{algorithmic}[1]
    \Procedure {\textsc{GetPartialRoutes}}{$\bar{x}, \bar{\theta}$}
    \State {$\mc{H} \gets \emptyset$}
    \State {Construct~$\bar{G}$ and let~$\mc{T}$ be the block-cut forest associated with~$G(\bar{x})$. \label{alg:block_cut_forest}}
    \State {Let~$\mc{P}$ be the set of all simple paths~$P = (S_1, \ldots, S_\ell) \subseteq \mc{T}$ connecting leaves of~$\mc{T}$ and such that either~$\ell = 1$ or~$S_1 \prec S_\ell$.}
    \For{$P = (S_1, \ldots, S_\ell) \in \mc{P}$ \label{alg:loop_start}}
        \State {$H \gets ( )$}
        \For {$j \in [\ell]$}
            \If {$S_j$ \text{is a block and~$(S_j \cap V_+) \setminus \mc{C}(\bar{G}) \neq \emptyset$}}
                \State {Append~$(S_j \cap V_+) \setminus  \mc{C}(\bar{G})$ to the end of~$H$.}
            \ElsIf {$S_j$ \text{is a cut-vertex~$v$}}
                \State {Append~$\{v\}$ to the end of~$H$.}
            \EndIf
        \EndFor
        \State {$\mc{H} \gets \mc{H} \cup \{H\}$}
    \EndFor \label{alg:loop_end}
    \State {\textbf{return}~$\mc{H}$}
\EndProcedure
\end{algorithmic}
\caption{\textsc{GetPartialRoutes}}
\label{algorithm:partial_route}
\end{algorithm}

Using the notion of block-cut forests, we present our heuristic for separating partial route inequalities in Algorithm~\ref{algorithm:partial_route}. In line~\ref{alg:block_cut_forest}, we build the graph~$\bar{G}$ and apply the algorithm of~\cite{hopcroft1973algorithm} to construct the block-cut forest~$\mc{T}$. Then, in lines~\ref{alg:loop_start}-\ref{alg:loop_end}, we generate a partial route~$H$ for each simple path~$P \in \mc{P}$ connecting distinct leaves of~$\mc{T}$ (see Figure~\ref{figure:block-cut-tree}). To simplify the presentation, we represent the paths by (one of) its corresponding sequence of vertices in~$\mc{T}$. 

In our implementation, we do not store the entire set of paths~$\mc{P}$. Instead, we iteratively construct the paths by running depth-first searches from each leaf of~$\mc{T}$. Additionally, to avoid examining symmetric partial routes, we assume that we have a total order~$\prec$ on all the vertices of the block-cut forest~$\mc{T}$.

We remark that, although we described Algorithm~\ref{algorithm:partial_route} as a separation heuristic, the algorithm itself does not find violated inequalities. Instead, it generates a candidate set of partial routes~$\mc{H}$, which can later be tested for violation depending on the form of the partial route inequalities and the choice of the recourse lower bounds. For example, in Section~\ref{section:application_vrpsd} we define a partial route lower bound~$\Lc{H}$ for the VRPSD with scenarios under the classical recourse policy. Then, given a partial route~$\H \in \mc{H}$, if we use the customer-based disaggregation from Remark~\ref{remark:disaggregation}, we check if~$\bar{\theta}(V_+(H)) > \Lc{H} \cdot \Whs(\bar{x}; \Xh{H})$. If instead we use the route-based disaggregation in Formulation~\eqref{formulation:partial_route}, we check if~$\bar{\theta}_v > \Lc{H} \cdot \Whs(\bar{x}; \Xh{H})$, where~$v \in V_+(\H)$ is the customer with the smallest index in~$\H$. More details on the separation routine are given in Section~\ref{subsection:vrpsd_separation} and Appendix~\ref{appendix:separation}.

To examine the time complexity of Algorithm~\ref{algorithm:partial_route}, we first note that line~\ref{alg:block_cut_forest} of the algorithm can be executed in~$\mc{O}(|E|)$ using the algorithm of~\cite{hopcroft1973algorithm}. Since~$G$ is complete, constructing the set~$\mc{P}$ also takes~$\mc{O}(|E|)$: for each leaf of~$\mc{T}$, we run a depth-first search in a tree~$T \in \mc{T}$ with~$\mc{O}(|V|)$ nodes. (Given a graph~$G'$ with~$|V(G')| \geq 2$, one can show that~$|\mc{B}(G')| \leq |V(G')| - 1$, so the number of nodes in the block-cut tree of~$G'$ is at most~$2 |V(G')| - 1$, see Proposition 1 of~\cite{li2021number}.) Therefore, since~$|\mc{P}| \leq |E|$ and each iteration of the loop in lines~\ref{alg:loop_start}-\ref{alg:loop_end} runs in~$\mc{O}(|V|)$, the overall time complexity of Algorithm~\ref{algorithm:partial_route} is~$\mc{O}(|V||E|)$. 

In contrast, the separation heuristic proposed by Hoogendoorn and Spliet~\cite{hoogendoorn2023improved} has time complexity of~$\mc{O}(|E|)$. 
Their algorithm is equivalent to Algorithm~\ref{algorithm:partial_route} but only considers the components of~$G(\bar{x}) \setminus \{0\}$ whose total flow to the depot is exactly 2. In fact, they construct~$\bar{V}$ by only considering the customers~$v$ with~$\bar{x}_{0v} = 1$, rather than~$1 \leq \bar{x}_{0v} < 2$, but this does not invalidate our next argument. 

More precisely, they replace the block-cut forest~$\mc{T}$ in Algorithm~\ref{algorithm:partial_route} with a forest~$\mc{T}' \subseteq \mc{T}$ that contains all block-cut trees~$T \in \mc{T}$ such that~$\bar{x}(E(0, V_+(T))) = 2$. In what follows, we demonstrate that~$\mc{T}'$ is made up of a disjoint set of paths, which indeed reduces the time complexity of Algorithm~\ref{algorithm:partial_route} to~$\mc{O}(|E|)$ (since each vertex in~$\mc{T}'$ is examined just once).

\begin{proposition}
\label{proposition:hoogendoorn_path}
Let~$\bar{x} \in \X$ and let~$\mc{T}$ be the block-cut forest associated with~$G(\bar{x})$. Then every tree~$T \in \mc{T}$ with~$\bar{x}(E(0, V_+(T))) = 2$ is a path.
\end{proposition}
\begin{proof}
Take a tree~$T \in \mc{T}$ with~$\bar{x}(E(0, V_+(T))) = 2$. If~$T$ has a single vertex, we are done. Otherwise, let~$L$ be the set of leaves of~$T$. Since~$T$ has at least two vertices, we know that~$|L| \geq 2$. Our goal is to show that~$|L| = 2$.

Note that each leaf~$S \in L$ is either of the form~$S = \{v, \texttt{dummy}(v)\}$ or satisfies~$S \subseteq V_+$. Let~$L_1$ be the set of leaves containing dummy vertices, and let~$L_2 = L \setminus L_1$ be the remaining leaves. We construct a mapping from leaves to customer sets as follows. For each~$S \in L$, define
\begin{equation*}
    \phi(S) \coloneqq
    \begin{cases}
        \{v\}, & \text{if~$S = \{v, \texttt{dummy}(v)\} \in L_1$},\\
        S \setminus \mc{C}(T), & \text{if~$S \in L_2$.}
    \end{cases}
\end{equation*}

We claim that~$\phi(S) \cap \phi(S') = \emptyset$ for every~$S, S' \in L$ with~$S \neq S'$. First observe that if~$S = \{v, \texttt{dummy}(v)\}$ and~$S' = \{u, \texttt{dummy}(u)\}$, then~$u \neq v$. Hence, without loss of generality, assume~$S \in L_2$. By the construction of block-cut trees, the blocks can only intersect only at the cut vertices. Since~$\phi(S)$ contains no cut vertex, the intersection~$\phi(S) \cap \phi(S')$ is empty. This holds whether~$S' \in L_1$ (in which case~$\phi(S')$ is a singleton containing a cut vertex) or~$S' \in L_2$ (in which case~$(S \setminus \mc{C}(T)) \cap (S' \setminus \mc{C}(T)) = \emptyset$).

Therefore, the sets in the image of~$\phi$ do not intersect. Using~$\bar{x}(E(0, V_+(T))) = 2$ we then learn that~$\sum_{S \in L} \bar{x}(E(0, \phi(S))) \leq 2$. To prove that~$|L| = 2$, it suffices to show that every leaf~$S \in L$ satisfies~$\bar{x}(E(0, \phi(S))) \geq 1$. 

If~$S = \{v, \texttt{dummy}(v)\} \in L_1$, we know that~$\bar{x}(E(0, \phi(S))) = \bar{x}_{0v} \geq 1$, so we assume that~$S$ belongs to~$L_2$. Since~$S$ is a leaf and a block, it has a unique neighbor which is a cut vertex~$v \in \mc{C}(T)$. Hence,~$\phi(S) = S \setminus \mc{C}(T) = S \setminus \{v\}$ is nonempty, as otherwise~$v$ would not be a cut vertex. 

Recall that the subtour elimination constraint (SEC) for a set~$S' \subseteq V_+$ can be expressed as~$x(\delta(S')) \geq 2$. Call~$S' = \phi(S)$ and apply the SEC for~$S'$ to learn that
\begin{align}
    & \bar{x}(\delta(S') \geq 2 \nonumber \\
    \iff & \bar{x}(E(0, S') + \bar{x}(E(v, S')) \geq 2. \label{ineq:proof_block_cut_tree1}
\end{align}
Similarly, by the SEC for the set~$S' \cup \{v\}$,
\begin{align}
    & \bar{x}(\delta(S' \cup \{v\})) \geq 2 \nonumber \\
    \iff & \bar{x}(E(0, S')) + \bar{x}(E(v, V \setminus (S' \cup \{v\}))) \geq 2 \label{ineq:proof_block_cut_tree2}
\end{align}
(note that~$V$ includes the depot). Summing~\eqref{ineq:proof_block_cut_tree1} and~\eqref{ineq:proof_block_cut_tree2} gives
$$2 \, \bar{x}(E(0, S')) \geq 4 - \bar{x}(\delta(v)) = 2,$$
as desired.
\end{proof}

Proposition~\ref{proposition:hoogendoorn_path} also provides a formal justification for the description given by~\cite{hoogendoorn2023improved} of their separation algorithm, where the authors state that, for each considered component, ``In the appropriate order, a singleton corresponding to each articulation point of the connected component is added to the partial route, as well as the sets of vertices in between two articulation points that are added as an unstructured component''. Since block-cut trees may contain branching vertices (see Figure~\ref{figure:block-cut-tree}), this description may appear somewhat imprecise, but Proposition~\ref{proposition:hoogendoorn_path} shows that, up to symmetry, the ``appropriate order'' is well-defined.

We conclude this appendix by mentioning that, although Algorithm~\ref{algorithm:partial_route} can (in principle) find more violated partial route inequalities than the separation routine proposed by~\cite{hoogendoorn2023improved}, our preliminary experiments indicate no advantage in considering trees~$T \in \mc{F}$ with~$\bar{x}(E(0, V_+(T))) > 2$. Therefore, in practice, we only consider trees with~$\bar{x}(E(0, V_+(T))) = 2$, which, as argued earlier, is equivalent to the algorithm of~\cite{hoogendoorn2023improved}.

\section{Proof of Claim~\ref{claim:monotone}}
\label{appendix:proof_claim_monotone}

\claimmonotone*
\begin{proof}
Take an arbitrary route~$R = (v_1, \ldots, v_\ell) \in \routes$. By superadditivity, for every~$i \in [\ell]$,~$\hat{\mc{Q}}(R, v_i)$ is nonnegative. Moreover,~$\sum_{i = 1}^\ell \hat{\mc{Q}}(R, v_i)$ is a telescoping sum that evaluates to~$\mc{Q}(R)$. Therefore,~$\hat{\mc{Q}}$ is indeed a disaggregation of~$\mc{Q}$ (i.e., it satisfies Definition~\ref{def:recourse_disaggregation}). 
    
Now let~$R' = (v_a, \ldots, v_b) \subseteq R$. We need to show that~$\sum_{i = a}^b \hat{\mc{Q}}(R, v_i) \geq \mc{Q}(R')$. If~$a = 1$, we apply again the telescoping argument to learn that~$\sum_{i = a}^b \hat{\mc{Q}}(R, v_i) = \mc{Q}((v_1, \ldots, v_b)) = \mc{Q}(R')$. In a similar way, if~$a \geq 2$, then~$\sum_{i = a}^b \hat{\mc{Q}}(R, v_i) = \mc{Q}((v_1, \ldots, v_b)) - \mc{Q}((v_1, \ldots, v_{a - 1})) \geq \mc{Q}(R')$, where the last inequality follows by superadditivity, since~$\mc{Q}((v_1, \ldots, v_b)) \geq \mc{Q}((v_1, \ldots, v_{a - 1})) + \mc{Q}(R')$.     
\end{proof}

\section{Proof of Claim~\ref{claim:restrictively_superadditive}}
\label{appendix:proof_claim_restrictively_superadditve}

We first show a simple fact.
\begin{fact}
    \label{fact:superraditive_monotone}
    Suppose that~$\mc{Q}$ is superadditive. For every~$R \in \routes$ and~$R'' \subseteq R' \subseteq R$, we have~$\mc{Q}(R') \geq \mc{Q}(R'')$.
\end{fact}
\begin{proof}
    Let~$R'' \subseteq R' \subseteq R \in \routes$ be as in the statement. There exist (possibly empty) routes~$R_1$ and~$R_2$ such that~$R' = R_1 \oplus R'' \oplus R_2$. By superadditivity and nonnegativity of~$\mc{Q}$, we have that~$\mc{Q}(R') \geq \mc{Q}(R_1 \oplus R'') + \mc{Q}(R_2) \geq \mc{Q}(R_1) + \mc{Q}(R'') \geq \mc{Q}(R'')$.
\end{proof}

\claimrestrictivelysuperadditive*
\begin{proof}
Take an arbitrary~$R \in \routes$ and let~$R_1, \ldots, R_t$ be disjoint subroutes of~$R$. There exists disjoint subroutes~$\tilde{R}_1, \ldots, \tilde{R}_t \subseteq R$ such that~$R = \bigoplus_{i \in [t]} \tilde{R}_i$ and, for every~$i \in [t]$,~$\tilde{R}_i$ covers~$R_i$, i.e.,~$R_i \subseteq \tilde{R}_i$. By Fact~\ref{fact:superraditive_monotone}, it follows that~$\mc{Q}(\tilde{R}_i) \geq \mc{Q}(R_i)$, for every~$i \in [t]$. Hence, repeatedly applying superadditivity yields~$\mc{Q}(R) = \mc{Q}\left(\bigoplus_{i \in [t]} \tilde{R}_i \right) \geq \sum_{i \in [t]} \mc{Q}(\tilde{R}_i) \geq \sum_{i \in [t]} \mc{Q}(R_i)$.
\end{proof}

\section{Proof of Theorem~\ref{thm:supperadditive}}
\label{appendix:proof_thm_superadditive}

\thmsuperadditive*
\begin{proof}
For~\ref{item:superadditive_a}, let~$R \in \routes$ and~$R_1, \ldots, R_t$ be disjoint subroutes in~$R$. Then
$$\mc{Q}(R) = \sum_{v \in V_+(R)} \hat{\mc{Q}}(R, v) \geq \sum_{i = 1}^t \sum_{v \in V_+(R_i)} \hat{\mc{Q}}(R, v) \geq \sum_{i = 1}^t \mc{Q}(R_i),$$
where the last inequality follows from the monotonicity of~$\hat{\mc{Q}}$ (Definition~\ref{def:monotonicity}).

Proving item~\ref{item:superadditive_b} requires more work, as we need to build a monotone disaggregation~$\hat{\mc{Q}}$ using restricted superadditivity alone. Let~$R = (v_1, \ldots, v_\ell)$ be a route. If~$R$ is not realizable (i.e.,~$R \notin \routes$), we can simply set~$\hat{\mc{Q}}(R, v) = \mc{Q}(R)$, for an arbitrarily chosen vertex~$v \in V_+(R)$. (Note that the conditions for restricted superadditivity only apply to realizable routes.) So we may safely assume that~$R$ is realizable. 

In the remainder of the proof, we show that the following dynamic-programming-based Algorithm~\ref{algorithm:get_disaggregation} finds values~$\hat{\mc{Q}}(R, v_1), \ldots, \hat{\mc{Q}}(R, v_\ell)$ such that~$\mc{Q}(R) = \sum_{i = 1}^\ell \hat{\mc{Q}}(R, v_i)$ and~$\sum_{v \in V_+(R')} \hat{\mc{Q}}(R, v) \geq \mc{Q}(R')$, for every~$R' \subseteq R$.

\begin{algorithm}{htb}
  \hspace*{\algorithmicindent} \textbf{Input:} A route~$R = (v_1, \ldots, v_\ell)$.\\
  \hspace*{\algorithmicindent} \textbf{Output:} The disaggregated recourse values~$\hat{\mc{Q}}(R, v_1), \ldots, \hat{\mc{Q}}(R, v_\ell)$.
\begin{algorithmic}[1]
    \Procedure {\textsc{GetDisaggregation}}{$R = (v_1, \ldots, v_\ell)$}
    \State{$\mc{A}_0 \gets \emptyset$}
    \For{$b = 1, \ldots, \ell$ \label{line:superadditive_loop_begin}}
        \State{$\hat{\mc{Q}}(R, v_b) \gets 0$}
        \State{$a \gets \argmax_{a' \in [b]} \{\mc{Q}((v_{a'}, \ldots, v_b)) - \sum_{i = a'}^b \hat{\mc{Q}}(R, v_i)\}$ \label{line:superadditive_a}}
        \State{$\Delta_b \gets \mc{Q}((v_{a}, \ldots, v_b)) - \sum_{i = a}^b \hat{\mc{Q}}(R, v_i)$ \label{line:superadditive_delta}}
        \If{$\Delta_b > 0$}
            \State{$\mc{A}_b \gets \mc{A}_{a - 1} \cup \{(v_a, \ldots, v_b)\}$ \label{line:set_A}}
            \State{$\hat{\mc{Q}}(R, v_b) \gets \Delta_b$}
        \Else
            \State{$\mc{A}_b \gets \mc{A}_{b - 1}$}
        \EndIf
    \EndFor \label{line:superadditive_loop_end}
    \State{$\Delta_R \gets \mc{Q}(R) - \sum_{i = 1}^\ell \hat{\mc{Q}}(R, v_i)$ \label{line:superadditive_delta_r}}
    \If{$\Delta_R > 0$}
        \State{$\hat{\mc{Q}}(R, v_1) \gets \hat{\mc{Q}}(R, v_1) + \Delta_R$}
    \EndIf
    \State {\textbf{return}~$\hat{\mc{Q}}(R, v_1), \ldots, \hat{\mc{Q}}(R, v_\ell)$}
\EndProcedure
\end{algorithmic}
\caption{\textsc{GetDisaggregation}}
\label{algorithm:get_disaggregation}
\end{algorithm}

Let~$b \in [\ell]$ and consider the corresponding iteration of the loop in lines~\ref{line:superadditive_loop_begin}--\ref{line:superadditive_loop_end}. For now, we focus on the variables~$a$ in line~\ref{line:superadditive_a} and~$\Delta_b$ in line~\ref{line:superadditive_delta}, the purpose of the collections of routes~$\mc{A}_b$ will be explained later. The way that Algorithm~\ref{algorithm:get_disaggregation} sets variable~$\hat{\mc{Q}}(R, v_b)$ guarantees that the returned disaggregated recourse values satisfy~$\sum_{v \in V_+(R')} \hat{\mc{Q}}(R, v) \geq \mc{Q}(R')$, for every subroute~$R' \subseteq R$ that ends at~$v_b$. To see this, note that after assigning~$\Delta_b$ in line~\ref{line:superadditive_delta}, it follows from the choice of~$a$ that, every subroute~$R_{a'} = (v_{a'}, \ldots, v_b)$ with~$a' \in [b]$ satisfies~$\Delta_b \geq \mc{Q}(R_{a'}) - \sum_{v \in V_+(R_{a'}) \setminus \{b\}} \hat{\mc{Q}}(R, v)$. Since by the end of this iteration we set~$\hat{\mc{Q}}(R, v_b) = (\Delta_b)^+$, it follows that~$\sum_{v \in V_+(R_{a'})} \hat{\mc{Q}}(R, v) \geq \mc{Q}(R_{a'})$, as desired. Repeating this argument for every iteration of~$b \in [\ell]$ we learn that, by the end of Algorithm~\ref{algorithm:get_disaggregation}, the returned values indeed satisfy~$\sum_{v \in V_+(R')} \hat{\mc{Q}}(R, v) \geq \mc{Q}(R')$, for every subroute~$R' \subseteq R$.

It remains to show that~$\mc{Q}(R) = \sum_{v \in V_+(R)} \hat{\mc{Q}}(R, v)$, or equivalently, that the value of~$\Delta_R$ in line~\ref{line:superadditive_delta_r} is always nonnegative. To prove this, we use induction to show that, for every~$b \in \{0, \ldots, \ell\}$, (i)~$\mc{A}_b$ is made of disjoint subroutes of~$(v_1, \ldots, v_b)$ (if~$b = 0$, then~$\mc{A}_b = \emptyset$), and (ii)~$\sum_{i \in [b]} \hat{\mc{Q}}(R, v_i) = \sum_{R' \in \mc{A}_b} \mc{Q}(R')$. Hence, applying restricted superadditivity we conclude that~$\mc{Q}(R) \geq \sum_{R' \in \mc{A}_\ell} \mc{Q}(R') = \sum_{i = 1}^\ell \hat{\mc{Q}}(R, v_i)$.

If~$b = 0$, we have that~$\mc{A}_0 = \emptyset$ and~$\sum_{i \in [0]} \hat{\mc{Q}}(R, v_i) = \sum_{R' \in \mc{A}_0} \mc{Q}(R') = 0$, so assume that~$b \geq 1$. If~$\Delta_b \leq 0$, then~$\mc{A}_b = \mc{A}_{b - 1}$, so by the induction hypothesis,~$\mc{A}_b = \mc{A}_{b - 1}$ is made of disjoint subroutes of~$(v_1, \ldots, v_b)$. Moreover, by the end of the loop in lines~\ref{line:superadditive_loop_begin}--\ref{line:superadditive_loop_end} we have that~$\hat{\mc{Q}}(R, v_b) = 0$, which gives
~$$\sum_{i \in [b]} \hat{\mc{Q}}(R, v_i) = \sum_{i \in [b - 1]} \hat{\mc{Q}}(R, v_i) = \sum_{R' \in \mc{A}_{b - 1}} \mc{Q}(R') = \sum_{R' \in \mc{A}_{b}} \mc{Q}(R').$$
If~$\Delta_b > 0$, then~$\mc{A}_b = \mc{A}_{a - 1} \cup \{(v_a, \ldots, v_b)\}$, where~$a \in [b]$. By the induction hypothesis,~$\mc{A}_{a - 1}$ is either empty or is made of disjoint subroutes of~$(v_1, \ldots, v_{a - 1})$, meaning that the routes in~$\mc{A}_b$ are disjoint and contained in~$(v_1, \ldots, v_b)$. By the end of the loop in lines~\ref{line:superadditive_loop_begin}--\ref{line:superadditive_loop_end}, we have that~$\hat{\mc{Q}}(R, v_b) = \Delta_b > 0$. So, by the choice of~$a$ in line~\ref{line:superadditive_a},
\begin{equation}
    \label{eq:proof_superadditive_1}
    \sum_{i = a}^b \hat{\mc{Q}}(R, v_i) = \mc{Q}((v_a, \ldots, v_b)).
\end{equation}
Furthermore, since~$a \geq 1$, we know by the induction hypothesis that
\begin{equation}
    \label{eq:proof_superadditive_2}
    \sum_{i \in [a - 1]} \hat{\mc{Q}}(R, v_i) = \sum_{R' \in \mc{A}_{a - 1}} \mc{Q}(R').
\end{equation}
Summing Equations~\eqref{eq:proof_superadditive_1} and~\eqref{eq:proof_superadditive_2} yields~$\sum_{i = 1}^b \hat{\mc{Q}}(R, v_i) = \sum_{R' \in \mc{A}_b} \mc{Q}(R')$, as desired.
\end{proof}

\newpage
\section{Separation algorithm for the VRPSD with scenarios under the classical recourse policy}
\label{appendix:separation}

~
\begin{algorithm}
  \hspace*{\algorithmicindent} \textbf{Input:} A candidate solution~$(\bar{x}, \bar{\theta}) \in \R^E \times \Q^{V_+}_+$ and an integer~$D \in \{1, 2\}$, which indicates if we use the disaggregation from~\cite{hoogendoorn2023improved} ($D = 1$) or the disaggregation in Remark~\ref{remark:disaggregation} ($D = 2$).
\begin{algorithmic}[1]
    \Procedure {\textsc{SeparateVRPSD}}{$\bar{x}, \bar{\theta}, D$}
    \State {Call CVRPSEP to get a family of customer sets~$\mc{S} \subseteq 2^{V_+}$.}
    \For {$S \in \mc{S}$}
        \State {Add RCI~$x(E(S)) \leq |S| - \bark{S}$.}
        \If{$D = 2$ and~$\bar{\theta}(S) < \setLc{S, \bar{k}(S)} \cdot (\bar{x}(E(S)) - |S| + \bark{S})$}
            \State {Add set cut~$\theta(S) \geq \setLc{S, \bar{k}(S)} \cdot (x(E(S)) - |S| + \bark{S})$.}
        \EndIf
    \EndFor
    \If {$\mc{S} \neq \emptyset$}
        \State{\textbf{return}}
    \EndIf
    \State {$\mc{H} \gets \tsc{GetPartialRoutes}(\bar{x}, \bar{\theta})$}
    \For {$H \in \mc{H}$}
        \If {$D = 1$}
            \State {$v \gets \text{customer in~$V_+(H)$ with the smallest index.}$}
            \If{$\bar{\theta}_v < \Lc{H} \cdot \Whs(x; \Xh{H})$}
                \State {Add partial route cut~$\theta_v \geq \Lc{H} \cdot \Whs(x; \Xh{H})$.}
            \EndIf
        \Else
            \State {$S \gets V_+(H)$}
            \If{$\bar{\theta}(S) < 
            \setLc{S, \bar{k}(S)} \cdot (\bar{x}(E(S)) - |S| + \bark{S})$}
                \State {Add set cut~$\theta(S) \geq \setLc{S, \bar{k}(S)} \cdot (x(E(S)) - |S| + \bark{S})$.}
            \Else
                \If{$\bar{\theta}(S) < \Lc{H} \cdot \Whs(x; \Xh{H})$}
                    \State {Add partial route cut~$\theta(S) \geq \Lc{H} \cdot \Whs(x; \Xh{H})$.}
                \EndIf
            \EndIf
        \EndIf
    \EndFor
\EndProcedure
\end{algorithmic}
\caption{\textsc{SeparateVRPSD}}
\label{algorithm:vrpsd_separation}
\end{algorithm}

\end{APPENDICES}

\end{document}